%% file: Revision_v4_arxiv.tex
\documentclass[11pt, a4paper, oneside, reqno]{amsart}

\input{style_arXiv}

\input{commands}
\graphicspath{{../}}

\title[]{The Nonconvex Geometry of Linear Inverse Problems}

\author[A. Eftekhari]{Armin Eftekhari}
\author[P. {Mohajerin Esfahani}]{Peyman {Mohajerin Esfahani}}

\thanks{The authors are with the Department of Mathematics and Mathematical Statistics, Umea University, Sweden,
({\tt Armin.Eftekhari@umu.se}), and the Delft Center for Systems and Control, Delft University of Technology, Netherlands ({\tt P.MohajerinEsfahani@tudelft.nl}).} 

\begin{document}
\maketitle

\begin{abstract}   
The gauge function, closely related to the atomic norm,  {measures the}  complexity of a statistical model, and has found broad applications in machine learning {and statistical signal processing}. In a high-dimensional learning problem, the gauge function attempts to safeguard against overfitting by promoting a sparse (concise) representation within the learning alphabet. 

In this work, within the context of linear inverse problems, we pinpoint the source of its success, but also argue that the applicability of the gauge function is inherently limited by its convexity, and showcase several learning problems where the classical gauge function theory fails. We then introduce a  new notion of {statistical} complexity, gauge$_p$ function, which overcomes the limitations of the gauge function. The {gauge$_p$ function is a simple generalization of the gauge function that can tightly control the sparsity of a statistical model within the learning alphabet and, perhaps surprisingly, draws further inspiration from the Burer-Monteiro factorization in computational mathematics.}

We {also} propose a {new learning machine}, with the building block of gauge$_p$ function, and arm this machine with a number of statistical guarantees. The potential of the {proposed gauge$_p$ function theory} is then studied  for two stylized applications. {Finally, we discuss the computational aspects and, in particular, suggest a tractable numerical algorithm for implementing the new learning machine.} 
\end{abstract}


\section{Introduction}\label{sec:intro}

While data is abundant, information is often sparse, and can be characterized mathematically using a small number of atoms, drawn from an  alphabet $\A \subset \R^d$. Concretely, an $r$-sparse  model $x^\n$ is specified as $x^\n := \sum_{i=1}^r c_i^\n A_i^\n$ for nonnegative coefficients $\{c_i^\n\}_{i=1}^r$ and atoms $\{A_i^\n\}_{i=1}^r\subset \A$. 

Complexity of the  model~$x^\n$ is often measured by its (convex) gauge function $\gauge_\A$~\cite{chandrasekaran2012convex,bach2012optimization,rockafellar2015convex}, to be defined later. Serving as a safeguard against {overfitting}, the gauge function has become a mainstay in linear inverse problems, a large class of learning problems with diverse applications in statistical signal processing and machine learning.

More specifically, to discover the true model~$x^\n$ or its atoms~$\{A_i^\n\}_{i=1}^r$, the classical gauge function theory studies the (convex) learning machine
\begin{equation}
    \min_x \|\L(x)-y\|_2^2 \,\,\text{subject to}\,\, \gauge_\A(x) \le \gamma,
    \label{eq:introCvx}
\end{equation}
\rev{Above,~$\L\rev{:\R^d\rightarrow\R^m}$ is a linear operator and the vector~$y$ typically in the form of $\L(x^\n)$ stores~$m$ (possibly inexact) observations of the true model~$x^\n$. Alternatively, as briefly discussed later, one can consider the basis pursuit or lasso reformulations of the problem~\eqref{eq:introCvx}.} 

\begin{wrapfigure}[12]{r}{0.32\textwidth}
\centering
\includegraphics[width=.25\textwidth]{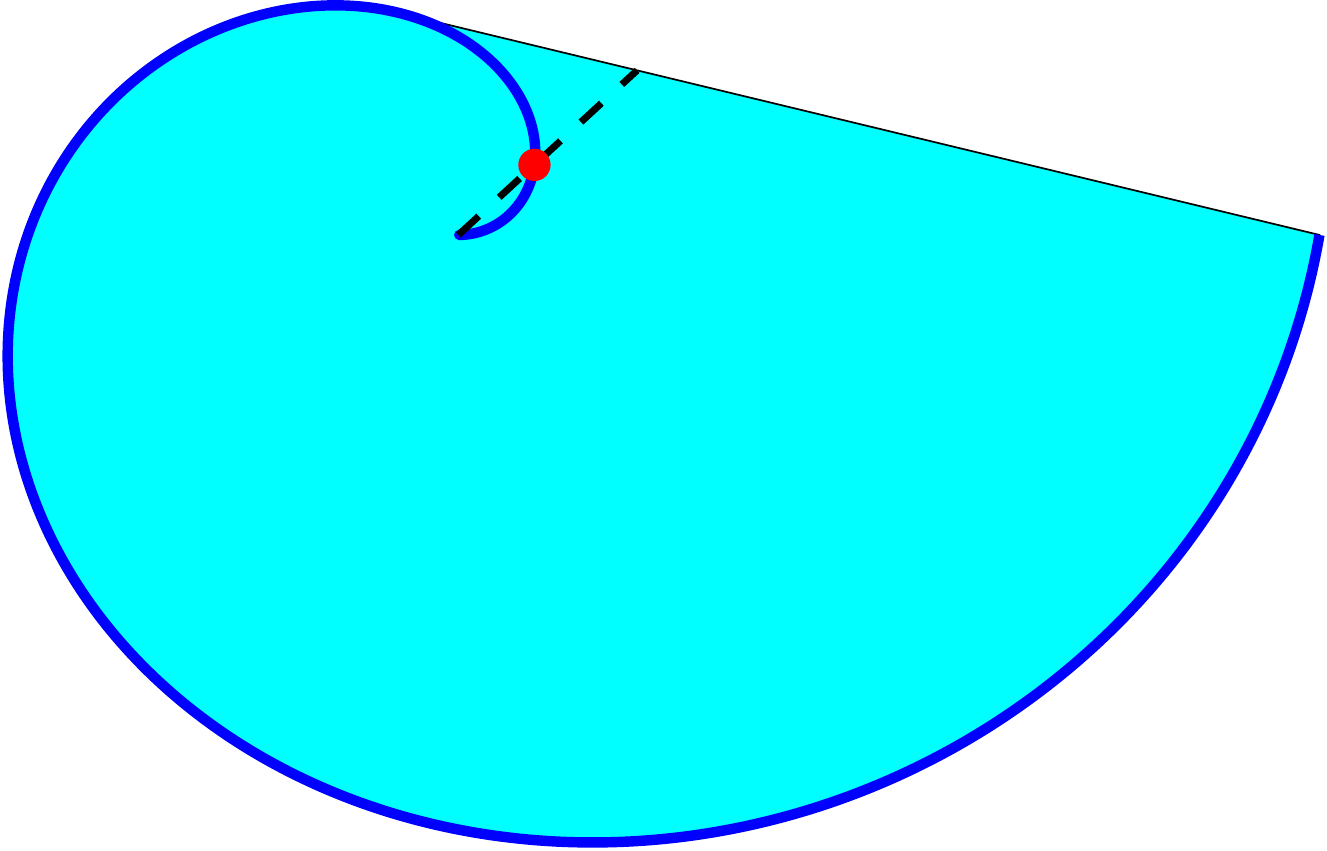}
\caption{\footnotesize{
\rev{As a toy example, the (blue) curve represents the learning alphabet~$\A$. A level set of the corresponding gauge function~$\gauge_{\A}$ is filled with cyan. The (convex) gauge function  evidently loses the geometric details of the alphabet~$\A$.} 
}} 
\label{fig:manifolds}
\end{wrapfigure}

A  certificate of correctness for the output of the machine~\eqref{eq:introCvx} is at the heart of the classical gauge function theory. This certificate can be constructed, for example, when~$\L$ is a generic linear operator and we have access to sufficiently many observations~\cite[Corollary~3.3.1]{chandrasekaran2012convex}.

The literature of the gauge function features numerous successful applications in different areas including statistics~\cite{hastie2015statistical,negahban2012unified,wainwright2009sharp} and signal processing~\cite{bhaskar2013atomic,candes2013simple,ahmed2013blind,shah2012linear,candes2008introduction}, to name a few. In all these success stories, the gauge function  successfully captures the underlying geometry of the learning alphabet.

The applicability of the gauge function is, however, inherently limited by its convexity. Indeed, there is anecdotal and numerical evidence suggesting  that the gauge function is incapable of capturing the  geometric details of many learning alphabets. {For example,} see~\cite{bach2008convex,richard2014tight} for sparse principal component analysis~(PCA) and {see}~\cite{schiebinger2018superresolution,li2016super} in the context of super-resolution. Motivated by these examples, {in this work} we will develop a theoretical foundation for a nonconvex counterpart of the gauge function, along with some basic computational tools. 

\paragraph{{\bf Contributions}} Our main objective is to develop a generalized theory, dubbed the gauge$_p$ function theory, {that addresses} the statistical limitations of the classical gauge function theory. More specifically, the following summarizes the contributions of this study: 

\begin{enumerate}[label=(\roman*), itemsep = 0mm, topsep = 0mm, leftmargin = 12mm]
    \item  This work proposes and studies  the gauge$_p$ function, a simple generalization of the classical gauge function,  as a new notion for statistical complexity  that can tightly control the sparsity level of a model within the learning alphabet~(Proposition~\ref{prop:gauge2EquivDefn}).
    
    \item The gauge$_p$ function motivates a new learning machine, for which 
    we develop statistical guarantees that parallel those of the classical gauge function theory (Theorem~\ref{thm:constructOpt}).
    The new theory is showcased with two stylized applications to manifold models and sparse PCA.
    \item This work also studies the computational aspects of implementing the new learning machine and proposes a tractable algorithm (Proposition~\ref{prop:alg}).
\end{enumerate}

\rev{\paragraph{\bf Additional details} 
We now  provide a section-by-section overview \rev{of this work,} punctuated \rev{by a} few bibliographic notes: Section~\ref{sec:old} \rev{reviews} the classical gauge function theory. \rev{Several successful and failed applications of this theory are  highlighted in Section~\ref{sec:old} and in the appendices.}}

Sections~\ref{sec:gaugep} and~\ref{sec:exGaugep} propose and study  the gauge$_p$ function, denoted by~$\gauge_{\A,p}$, as a new notion of statistical complexity. The gauge$_p$ function~$\gauge_{\A,p}$ generalizes the classical gauge function~$\gauge_\A$ and can tightly control the sparsity level of a model within the learning alphabet $\A$. Gauge$_p$ function draws further inspiration from the idea of Burer-Monteiro factorization~\cite{burer2005local}. In the success stories of the classical theory, gauge and gauge$_p$ functions nearly coincide. In contrast, whenever the classical theory fails, gauge$_p$ function behaves  more favourably compared to the classical gauge function, as detailed in Section~\ref{sec:exGaugep}. Motivated by this observation, Section~\ref{sec:newMachine} introduces the learning machine 
    \begin{equation}
    \min_x \|\L(x)-y\|_2^2 \quad \text{subject to}\quad \gauge_{\A,p}(x) \le \gamma,
    \label{eq:introNonCvx}
    \end{equation}
in which the \rev{new} gauge$_p$ function plays the role of regularizer in place of the classical gauge function \rev{in}~\eqref{eq:introCvx}.  \rev{Section~\ref{sec:newMachine} elucidates that, as~$p$ varies, the new machine interpolates between two extremes:
\begin{itemize}[itemsep = 0mm, topsep = 0mm, leftmargin = 12mm]
    \item the classical convex machine~\eqref{eq:introCvx}; and 

    \item the $\ell_0$-pursuit: 
    $\min_x \|\L(x)-y\|_2^2 \,\,\text{subject to}\,\, x \text{ has an } r\text{-sparse decomposition in }\A$. 
\end{itemize}
We recall that $r$ is the sparsity level of the true model~$x^\n$ within the alphabet~$\A$. Moreover, the new machine~\eqref{eq:introNonCvx} extends the Burer-Monteiro idea to any alphabet in the following sense: As detailed in Section~\ref{sec:newMachine}, the new machine~\eqref{eq:introNonCvx} coincides with the  widely-used Burer-Monteiro factorization when the learning alphabet~$\A$ is the set of unit-norm rank-$1$ matrices. We also note that implementing the new machine~\eqref{eq:introNonCvx} often requires solving a nonconvex optimization problem.}
    
Section~\ref{sec:newGaugeTheory} develops some statistical guarantees for the new machine~\eqref{eq:introNonCvx}. In particular, Lemma~\ref{lem:nonCvxLearn} therein introduces a family of certificates for verifying the correctness of the \rev{solutions of the optimization problem~\eqref{eq:introNonCvx}},
analogous to Lemma~\ref{lem:duality} for the convex machine~\eqref{eq:introCvx}.

When~$\L$ is a generic linear operator, $p$ is small and~$m$ is sufficiently large, we also develop a probabilistic approach to construct these certificates, as detailed in  Theorem~\ref{thm:constructOpt}, loosely analogous to~\cite[Corollary~3.3.1]{chandrasekaran2012convex} for the convex machine~\eqref{eq:mainCvx}. The proof technique for Theorem~\ref{thm:constructOpt} appears to be new in this context and might be of independent interest. More specifically, instead of a single certificate, the proof of Theorem~\ref{thm:constructOpt} constructs a family of certificates that \emph{jointly} certify the learning outcome. 
    
In Section~\ref{sec:action}, we showcase the  new theory with two stylized applications, namely, manifold-like models~\cite{peyre2009manifold} and sparse PCA~\cite{zou2006sparse}. Both applications span highly active research areas and {it is not our intention to improve over the state of art for these applications, but rather to merely convince the reader that the new machine~\eqref{eq:introNonCvx}  merits further investigation and  research.}
    
\paragraph{\bf \rev{Computational aspects}}
Implementing the new machine~\eqref{eq:introNonCvx} often requires solving a nonconvex optimization problem. For certain learning alphabets, such as the one in matrix sensing~\cite[Chapter~5]{chi2019nonconvex} or~\cite[Section~2.1]{eftekhari2020implicit}, the landscape of the optimization problem~\eqref{eq:introNonCvx} is benign for \rev{a sufficiently small~$p$. That is, the optimization  problem does not have any spurious stationary points when $p$ is small. For such alphabets, problem~\eqref{eq:introNonCvx} can  be solved efficiently~\cite{jin2017escape}.}

For certain other alphabets, such as smooth manifolds~\cite{eftekhari2015new}, the optimization landscape of~\eqref{eq:introNonCvx} might in general contain spurious stationary points which could trap first- or second-order optimization algorithms, such as gradient descent. Nevertheless, problem~\eqref{eq:introNonCvx} can be solved efficiently to (near) stationarity, rather than global optimality. This compromise is common in machine learning: As an example,  empirical risk minimization is known to be intractable for neural networks in general. Instead the practitioners {often} seek local (rather than global) optimality by means of first- or second-order optimization algorithms~\cite[Chapter 20]{shalev2014understanding}.
    
{For yet other learning alphabets, such as the one in sparse regression~\cite{hastie2015statistical}, the problem~\eqref{eq:introNonCvx} might be NP-hard in the worst case. Nevertheless, not all is lost here and we draw inspiration from recent developments in mixed-integer programming~\cite{bertsimas2020sparse,bertsimas2016best}. Indeed, after decades of research, modern mixed-integer optimization algorithms that directly solve {the}  problem~\eqref{eq:l0machine} for sparse regression can now outperform convex heuristics in speed and scalability, and without incurring the well-documented bias of the shrinkage methods. More specifically, inspired by~\cite{bertsimas2020sparse}, we develop in Section~\ref{sec:algorithm} a tractable optimization algorithm to numerically solve the new problem~\eqref{eq:introNonCvx} when the alphabet is finite and, consequently, problem~\eqref{eq:introNonCvx} is  NP-hard.} 

\paragraph{\bf {Kurzgesagt}}
To summarize, motivated by the limitations of the classical gauge function theory, this work studies a new learning machine for solving linear inverse problems. The  gauge$_p$ function theory, introduced in this work, is far from complete and this study raises several research questions, which require further investigation. For example, 
this first work is largely focused on the statistical aspects of the new theory. Beyond the preliminary results presented in Section~\ref{sec:algorithm}, {more} effort is required to better understand the computational aspects of the new machine.

\paragraph{\textbf{Notation}}
Throughout this study, we adopt the notation from~\cite{barvinok2002course} to denote by $\lin(\cdot)$, $\aff(\cdot)$, $\cone(\cdot)$,  and $\hull(\cdot)$, the linear, affine, conic, and convex hulls of a set, respectively. A cone is a positive homogeneous subset of a vector space. For a convex set $\mathcal{C}$, its tangent cone at $x\in \mathcal{C}$ is $\cone(\mathcal{C}-x)$, where the subtraction is in the Minkowski's sense.
We use $\|x\|_p$ to denote the $\ell_p$-norm of a vector $x \in \R^n$. For a function $f:\R^d\to \R$, its convex conjugate is defined as $f^*(z) \Let \sup_x \langle x, z\rangle - f(x)$.  Given a linear operator~$\L:\mathbb X \to \mathbb Y$, defined on a pair of vector spaces~$\mathbb X$ and~$\mathbb Y$, the corresponding adjoint operator is denoted by $\L^*$, i.e., $\langle \L(x), y \rangle = \langle x, \L^*(y)\rangle$ for all $(x,y) \in \mathbb X \times \mathbb Y$. When the spaces are equipped with the norms $(\mathbb X, \|\cdot\|_{\mathbb X}), (\mathbb Y,\|\cdot\|_{\mathbb Y})$ , the induced operator norm is denoted by~$\|\L\|_{\op}\Let \sup_{x \in \mathbb X} \|\L(x)\|_{\mathbb Y} / \|x\|_{\mathbb X}$. We also use the notation $[l]:=\{1,\cdots,l\}$ for an integer $l$. Throughout, we always use the convention that $0/0=0$. 

\section{Classical (Convex) Gauge Function Theory}\label{sec:old}

In this section, we first review the classical gauge function theory. We then highlight both successful and failed applications of the theory in order to motivate the main contribution of this work, which is a generalized gauge function theory that addresses the statistical limitations of the classical theory.

To review the gauge function theory, this section takes a somewhat different perspective, which appears to be new, to the best of our knowledge. 
The different geometric perspective of this section will later help us generalize the classical theory in Section~\ref{sec:new}. 

We now begin with a few definitions. The notion of {slice} below,
visualized in Figure~\ref{fig:slices}, appears to be new even though it has  implicitly appeared before, see for example~\cite{bi2016refined}.

\begin{figure}[ht!]
\centering
    \subfloat[Two dimensional slices]{\includegraphics[scale=0.4, clip]{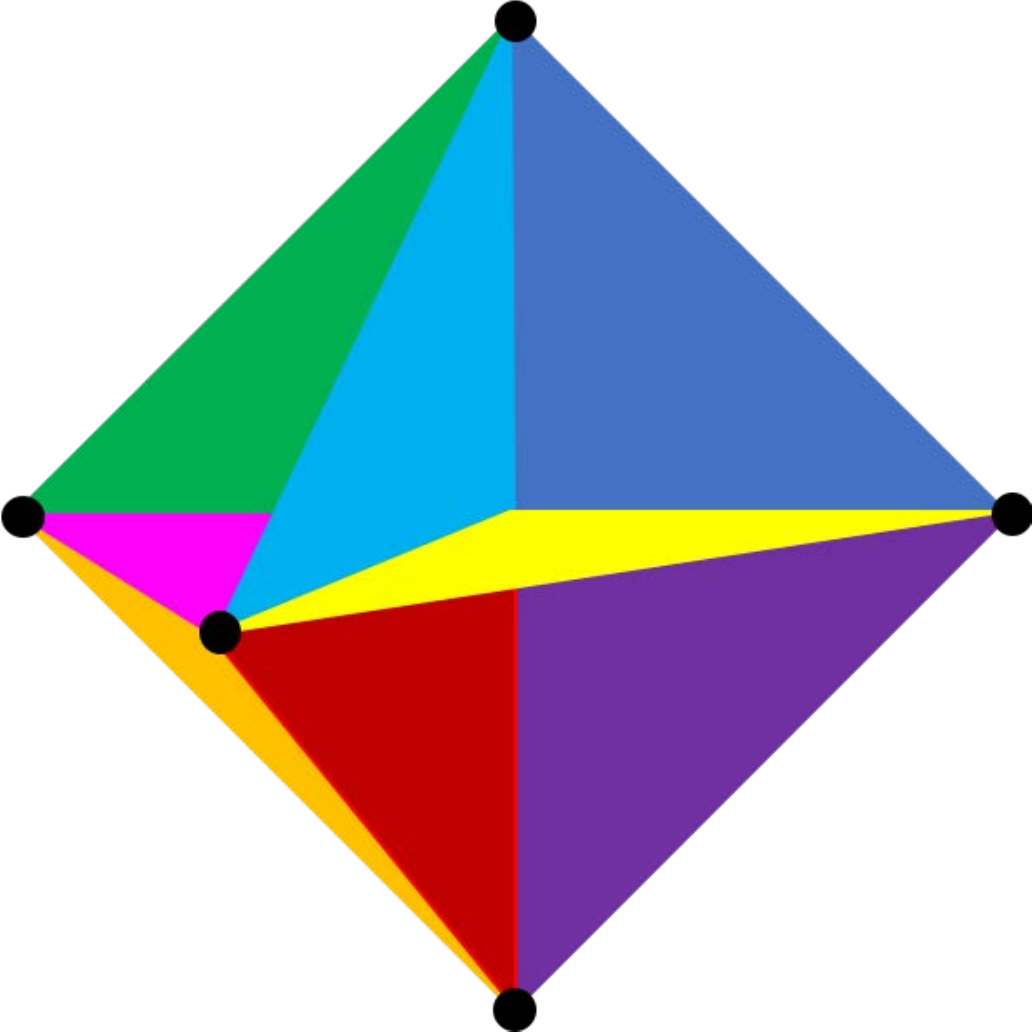}\label{fig:slices}} \hfil
    \subfloat[Exposed versus hidden faces]{\includegraphics[scale = 0.44, clip]{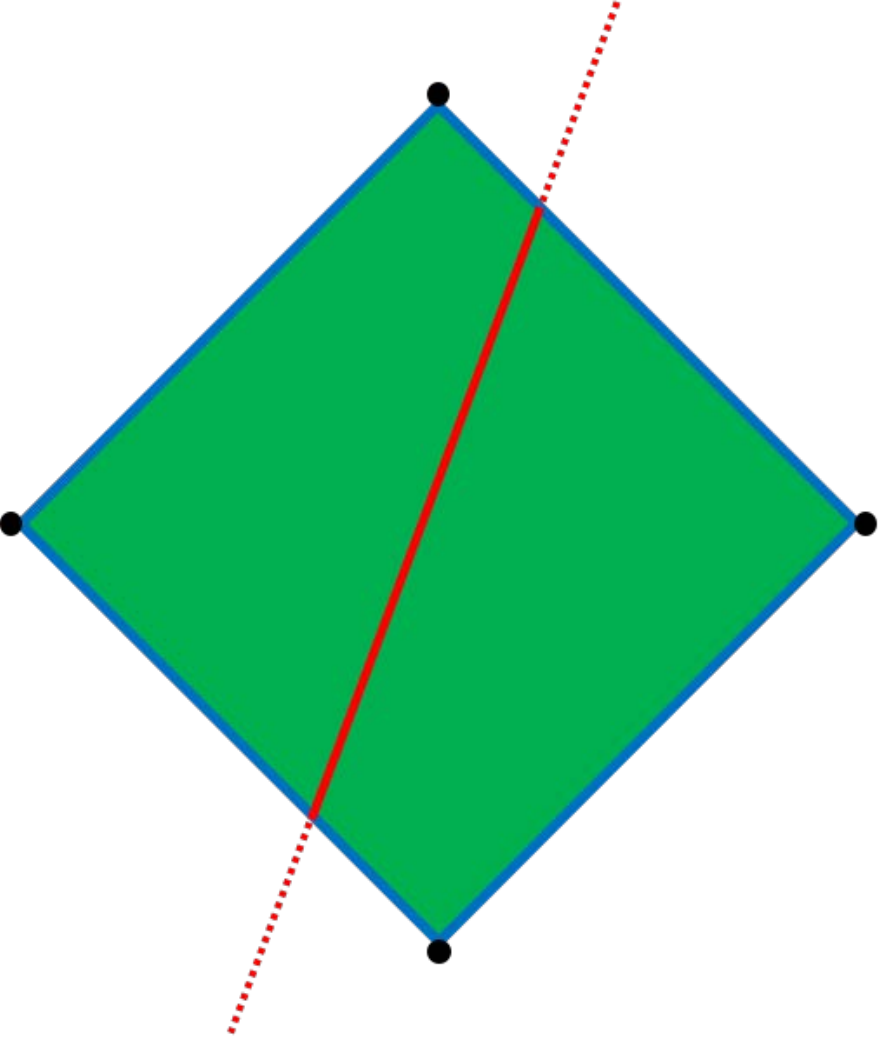}\label{fig:extremes}}
     \caption{
    Figure~\ref{fig:slices} depicts several two-dimensional slices of the alphabet~$\{\pm e_i\}_{i=1}^3$ in different colors where~$e_i$ is the $i^{\text{th}}$ canonical vector. In 
    Figure~\ref{fig:extremes}, the black dots and the blue line segments
    are exposed faces while the solid red line segment is a  hidden 
    face. The extreme points coincide with the black dots.}
	\label{fig:pictorial}
\end{figure}

\begin{defn}[{\sc Slice}]\label{defn:slice}
For an alphabet $\A\subset \R^d$,  an integer $r$ and  atoms $\{A_i\}_{i=1}^r \subset \A$, the corresponding slice of~$\hull(\A)$ is \rev{defined to be the set $\hull(\{A_i\}_{i=1}^r \cup \{0\})$. We also let $\Slice_r(\A)$ be the set of all slices of $\hull(\A)$ formed by at most $r$ atoms. Note that $\Slice_r(\A)$ is a set of sets.}
\end{defn}

Definition~\ref{defn:slice} allows us to {re}write the $r$-sparse  model $x^\n = \sum_{i=1}^r c_i^\n A_i^\n$ {in Section~\ref{sec:intro}} as 
\begin{align}
    x^\n \in \cone(\slice^\n), \qquad \slice^\n \in \Slice_r(\A),
    \label{eq:mainModelRewritten}
\end{align}
where the slice $\slice^\n$ {in~\eqref{eq:mainModelRewritten}} is formed by the atoms $\{A_i^\n\}_{i=1}^r$. \rev{That is, $\mathcal{S}^\n=\mathrm{conv}(\{A_i^\n\}_{i=1}^r \cup \{0\})$. The representation in \eqref{eq:mainModelRewritten} is particularly helpful when we are more interested in the atoms $\{A_i^\n\}_{i=1}^r$ rather than their coefficients $\{c_i^\n\}_{i=1}^r$.} In convex statistical learning, the complexity of a model, such as~$x^\n$, is commonly measured by its gauge function~\cite{chandrasekaran2012convex,rockafellar2015convex}.
\begin{defn}[{\sc Gauge function}] \label{defn:gauge}
For an alphabet $\A \subset \R^d$, the gauge function{~$\gauge_\A:\R^d \to \R$} is 
\begin{align}\label{eq:gaugeFcn1}
    \gauge_\A(x) \Let &  \inf \big\{ t \,:\, x/t \in  \hull(\A),\, t \ge 0 \big\} \\
    = &   \inf \l\{ \sum_{i=1}^l c_i
    \,:\, x = \sum_{i=1}^l c_i A_i,  
    \,c_i \ge 0, \,\, A_i\in \A,\,i\in [l]\r\},
    \notag 
\end{align}
with the convention that $0/0=0$. Above, $[l]:=\{1,\cdots,l\}$. \rev{Lastly, the second infimum above is taken over~$l$ and $\{c_i\}_{i=1}^l$ and~$\{A_i\}_{i=1}^l$.}
\end{defn}

\rev{Let us next  collect some standard assumptions on learning alphabets. Throughout this work, each technical result is centered around an alphabet that satisfies a subset of the assumptions below.}

\begin{assumption}[{\sc Alphabet regularity}] \label{a:alph}
The following assumptions are in order:
 \begin{enumerate}[label=(\roman*), itemsep = 0mm, topsep = 0mm, leftmargin = 12mm]

\item (Origin:) \label{assumption:containsOrigin}
The alphabet $\A$ contains the origin, i.e.,  $0\in \A $. 

\item (Symmetry:) \label{assumption:alphabet0}
The alphabet $\A$ is symmetric, i.e., $\A = -\A$. 
\item (Boundedness:) 
\label{assumption:bounded}
The alphabet $\A$ is bounded, i.e., $\sup_{A\in \A}\|A\|_2<\infty$. 

 \item (Unit sphere:)
 \label{assumption:unitSphere}
 The alphabet $\A $ belongs to the unit sphere, i.e., $\|A\|_2=1$ for every $A\in \A$.
\end{enumerate}
\end{assumption}

Under Assumption~\ref{a:alph}\ref{assumption:alphabet0}, the gauge function~$\gauge_\A$ is in fact a norm {for}~$\R^d$~\cite{rockafellar2015convex}, and the unit ball of this norm is \rev{the convex hull of $\A$}, i.e.,  
\begin{align}
    \hull(\A) = \{ x: \gauge_\A (x) \le 1\}. 
    \label{eq:unitBall}
\end{align}
Moreover, the dual norm corresponding to $\gauge_\A$ is denoted by $\dual_\A:\R^d\rightarrow\R$, and defined as
\begin{align}
    {\dual_\A(z) \Let \sup \big\{ \langle z, x \rangle : \gauge_\A(x) \le 1 \big\} = \sup \big\{ \langle z, A \rangle : A\in \A \big\}}.
    \label{eq:dual}
\end{align}
As a {device to control the} statistical complexity of learning, the gauge function has found broad applications in statistical signal processing and machine learning.  We are particularly interested in linear inverse problems~\cite{chandrasekaran2012convex}, which unify a wide range learning problems and \rev{we will encounter a few of them throughout this work. More specifically, for a linear operator~$\L:\R^d\rightarrow\R^m$ and an integer~$r$, consider the (exact) setup}
\begin{align}
    y := \L(x^\n) \in \R^m,\qquad x^\n\in \cone(\slice^\n),\, \slice^\n \in \Slice_r(\A),
    \label{eq:mainModel}
    \tag{exact}
\end{align}
where $\Slice_r(\A)$ was defined in Definition~\ref{defn:slice}. \rev{For example, in {statistical inference} or signal processing, $\L$ is the measurement operator and $y$ is the vector of observations~\cite{eldar2012compressed,van2000asymptotic}. \eqref{eq:mainModel} is called an exact setup because it does not account for any noise or numerical inaccuracy that might distort~$\L(x^\n)$.}
Given $y$, in order to learn $x^\n$ or its sparse decomposition {in the alphabet~$\A$},  consider the learning machine   
\begin{equation}
    \min_{x}\,\, \| \L(x) -y\|_2^2 \text{ subject to } \gauge_\A(x) \le \gauge_\A(x^\n).
    \label{eq:mainCvx0}
\end{equation}
\rev{Even though the gauge function  appears in the constraints of the problem~\eqref{eq:mainCvx0},} what follows in this section also holds true for the basis pursuit reformulation of~\eqref{eq:mainCvx0}, in which the objective and constraints are swapped and the \rev{exact} knowledge of~$\gauge_\A(x^\n)$ is not required~\cite{chen2001atomic}. \rev{It seems more convenient for us to work with~\eqref{eq:mainCvx0}, compared to its basis pursuit formulation.}
\rev{Using the definition of gauge function in~\eqref{eq:gaugeFcn1}, we can also reformulate the above machine as}
\rev{\begin{align}
     \inf\l\{ \Big\| \sum_{i=1}^l c_i \L(A_i) - y \Big\|_2^2 : \sum_{i=1}^l c_i \le \gauge_\A(x^\n),\, c_i \ge 0,\, A_i\in \A,\, i\in [l]  \r\}, 
    \tag{gauge} 
    \label{eq:mainCvx} 
\end{align}
where the infimum is over $l$ and $\{c_i\}_{i=1}^l$ and $\{A_i\}_{i=1}^l$.  It is the above reformulation of~\eqref{eq:mainCvx0} that we will often work with in this paper.}
To study~\eqref{eq:mainCvx}, let us recall two basic concepts from convex geometry, \rev{both} {visualized in Figure~\ref{fig:extremes}}.  See~\cite[Definitions 2.6 and 3.1]{barvinok2002course}.  

\begin{defn}[{\sc Extreme point}]\label{defn:extreme}
    An extreme point of a closed convex set $\mathcal{C}$ is a point in $\mathcal{C}$ that cannot be written as a convex combination of other points in $\mathcal{C}$. Let {also the set} $\ext(\mathcal{C})$ collect {all} the extreme points of $\mathcal{C}$. 
\end{defn}

\begin{defn}[{\sc Face}] \label{defn:face}
    For a closed convex set $\mathcal{C}$, the subset $\face\subset\mathcal{C}$ is  a face of $\mathcal{C}$ if there exists a hyperplane $\mathcal{H}$ such that $\face = \mathcal{C}\cap \mathcal{H}$. Dimension of a face $\face$  is the dimension of the affine hull of~$\face$, i.e., $\dim(\face)=\dim(\aff(\face))$. Moreover, we say that $\face$ is an exposed face of~$\mathcal{C}$ if one of the two halfspaces formed by~$\mathcal{H}$ contains~$\mathcal{C}$.  A face~$\face$ is hidden if it is not exposed. Lastly, for an integer $r$, we let~$\F_r(\A)$ denote the set of all faces of~$\hull(\A)$ with dimension at most $r$.
\end{defn}

For an alphabet $\A$, a simple inclusion that we will use frequently in this work  is that
 \begin{align}
 \ext(\hull(\A)) \subseteq \A,
 \label{eq:extInclusionGen}
 \end{align}
which states that the extreme points of the convex hull of a set belong to that set.
{Note also that an exposed $0$-dimensional face of a convex set $\mathcal{C}$ is simply an extreme point of $\mathcal{C}$.} 

\rev{Equipped with the above two definitions, the following  lemma exemplifies learning with the gauge function. In effect, the lemma below states} that the machine~\eqref{eq:mainCvx} successfully learns the model~$x^\n$, provided that a certain certificate of correctness exists. 
The next lemma is in essence a standard result, see for example~\cite[Lemma 2.1]{candes2006robust}, though it has {not} appeared in the literature from the geometric perspective adopted in this section, to the best of our knowledge. 

\begin{lem}[\rev{{\sc Certificate of correctness}}]\label{lem:duality}
Consider the model $x^\n$ in~\eqref{eq:mainModel}. 
Suppose that Assumptions~\ref{a:alph}\ref{assumption:alphabet0} and~\ref{assumption:bounded} are met.
If $x^\n=0$, then the machine~\eqref{eq:mainCvx} correctly returns $0$. \rev{That is, $x^\n=\sum_{i=1}^{\wh{l}} \wh{c}_i \wh{A}_i =0$, where $\{\wh{c}_i,\wh{A}_i\}_{i=1}^{\wh{l}}$ is a solution of the optimization problem~\eqref{eq:mainCvx}.}

Otherwise, let $\face^\n$  be an exposed face of $\hull(\A)$ such that~$x^\n/\gauge_\A(x^\n) \in \face^\n$. 
Suppose also that the following holds: 
\begin{enumerate}[label=(\roman*), itemsep = 0mm, topsep = 0mm, leftmargin = 12mm]
    \item \label{lem:duality:injective}
    The linear operator~$\L$ in~\eqref{eq:mainModel} is injective when restricted to the subspace $\lin(\face^\n)$, i.e.,
    \begin{align*}
        x \in \lin(\face^\n) \quad \text{and} \quad \L(x) = 0 \qquad \Longleftrightarrow\quad x = 0\,.
    \end{align*}
    \item 
    The face~$\face^\n$ has a support vector within the range of $\L^*$, where $\L^*$ is the adjoint {of the operator $\L$}, i.e., there exists $Q\in \range (\L^*)$ such that 
    \begin{align} 
    & \langle Q, x-x' \rangle <0 , \qquad \forall x\in \hull(\A) -   \face^\n,\,\,  \forall  x'\in \face^\n\,. 
    \label{eq:dualCertCnds}
    \end{align}
\end{enumerate}
Then the machine~\eqref{eq:mainCvx} successfully returns~$x^\n$. \rev{That is, $x^\n=\sum_{i=1}^{\wh{l}} \wh{c}_i \wh{A}_i$, where $\{\wh{c}_i,\wh{A}_i\}_{i=1}^{\wh{l}}$ is a solution of the optimization problem~\eqref{eq:mainCvx}.} 
\end{lem}

Lemma~\ref{lem:duality} also immediately extends to the basis pursuit formulation of the problem~\eqref{eq:mainCvx}, in which the objective and constraint are swapped~\cite{chen2001atomic}. 

\rev{The correctness certificate $Q$ in Lemma~\ref{lem:duality} can be successfully designed for a variety of linear inverse problems, e.g., in compressive sensing~\cite{candes2008introduction} and low-rank matrix completion~\cite{davenport2016overview}. Often the starting point is  the construction of a pre-certificate within $\range(\L^*)$, for which the assertion~\eqref{eq:dualCertCnds} is then verified, see for example~\cite{li2015overcomplete}.}

Despite these success stories, there are many linear inverse problems for which the classical gauge function theory fails. To highlight the statistical failures of the  machine~\eqref{eq:mainCvx}, we focus in this section on structured data factorization, i.e.,  the \rev{particular}  linear inverse problem 
in (\ref{eq:mainCvx}), \rev{for which}~$\L$ is the identity operator. \rev{In this special case, it is not difficult to verify that solving~\eqref{eq:mainCvx} is equivalent to finding a minimal decomposition of $x^\n$ that achieves $\gauge_\A(x^\n)$. That is, when $\L$ is the identity operator,~\eqref{eq:mainCvx} reduces to the optimization problem}
\begin{align}
    \gauge_\A(x^\n) = \inf\l\{ \sum_{i=1}^l c_i : x^\n = \sum_{i=1}^l c_i A_i,\,  c_i \ge 0,\, A_i\in \A,\, \forall i\in [l] \r\}, 
    \label{eq:cvxFacNoiseFre}
    \tag{gauge : $\mathcal{L}=\mathrm{id}$}
\end{align}
where the infimum above is taken over the integer $l$ and coefficients$\{c_i\}_i$ and the atoms $\{A_i\}_i$. 
Let us list two examples of structured data factorization for which the classical theory fails.

\begin{example}[{\sc Manifold models}]\label{ex:1}
In a multitude of  problems, the alphabet $\A$ is naturally an  embedded submanifold of the Euclidean space~\cite{eftekhari2015new,boumal2020introduction}. 
Our first example {showcases} the failure of the gauge function theory for manifold manifolds. As a toy example, here we consider the alphabet
\begin{align}
    \A := \{ (t\cos(\pi t), t \sin(\pi t)): t\in [0,2]\} \subset \R^2,
    \label{eq:spiral}
    \tag{spiral}
\end{align}
which forms a spiral in $\R^2$. \rev{That is, $\A$ is the one-dimensional manifold with boundary visualized in Figure~\ref{fig:manifolds}.} For the above alphabet, it is important to note that the inclusion in~\eqref{eq:extInclusionGen} is strict, i.e., 
\begin{align}\label{eq:inclusionStrict} 
    \ext(\hull(\A)) \subset \A.   
\end{align}
In particular, the {$1$-sparse} model 
\begin{align}
    x^\n := A^\n = \frac{1}{4\sqrt{2}} (1,1)\in \A,
    \label{eq:exModelManifold}
\end{align}
is {not} an extreme point of $\hull(\A)$, i.e., $x^\n$ above belongs to the right-hand side but {not} to the left-hand side of~\eqref{eq:inclusionStrict}. \rev{The model $x^\n$ in~\eqref{eq:exModelManifold} is represented with a red dot in Figure~\ref{fig:manifolds}. We note that the model~$x^\n$ in~\eqref{eq:exModelManifold} has the alternative decomposition}
\begin{align}
    x^\n = \frac{A_1}{8\sqrt{2}}  + \frac{A_2}{2\sqrt{2}} , \qquad 
    A_1 = \frac{1}{2}(0,1) \in \A,\qquad A_2 = 2(1,0)\in \A.
\end{align}
By comparing the two alternative representations of $x^\n$ above, we find that 
\begin{align}
    \gauge_\A(x^\n) \le \min\l(1,\frac{1}{8\sqrt{2}}+\frac{1}{2\sqrt{2}} \r)=\frac{5}{8\sqrt{2}}<1.
    \qquad \text{(see \eqref{eq:gaugeFcn1})}
    \label{eq:gaugeFailManEx}
\end{align}
\rev{That is, \eqref{eq:exModelManifold} is not the minimal decomposition of $x^\n$ that achieves $\gauge_\A(x^\n)$ in~\eqref{eq:cvxFacNoiseFre}.}
In fact, {a} visual inspection of Figure~\ref{fig:manifolds}  reveals that the minimal decomposition of $x^\n$ that achieves~$\gauge_\A(x^\n)$  is not $1$-sparse.  {We conclude that} the machine~\eqref{eq:cvxFacNoiseFre} fails to learn any $1$-sparse decomposition for the model $x^\n$.
\end{example}

\rev{As as another failed application of the gauge function theory, we turn to sparse PCA. Here the objective is to decompose a data matrix into a small number of rank-$1$ and sparse components.} More specifically, suppose that the rows and columns of the data matrix $x^\n \in \R^{d_1\times d_2}$ correspond to  samples and features, respectively. In general, the {leading} principal components of $x^\n$ are not sparse, which renders them difficult to interpret. That is, it is often not possible to single out the key features in a data matrix from its leading principal components. 

In contrast, for an integer $r$, sparse PCA in effect models the data matrix~$x^\n$ as 
    \begin{align}
       & x^\n\in \cone(\slice^\n),\qquad  \slice^\n \in \Slice_r(\A), \nonumber\\
       & \A := \{ uv^\top: \|u\|_2=\|v\|_2=1,\, \|v\|_0 \le k\} \subset \R^{d_1\times d_2},
       \label{eq:spca1}
       \tag{sparse PCA}
    \end{align}
    where $\|v\|_0$  denotes the number of nonzero entries of~$v$. 
    \rev{In words, the data matrix $x^\n$ is a conic combination of $r$ atoms from the alphabet $\A$.}
    \rev{The representation in~\eqref{eq:spca1} is closely related to~\cite{bach2008convex,richard2014tight} and~\cite[Equation~3.12]{zou2006sparse}. Note that $k$ is the sparsity level of vector~$v$ (i.e.,  the number of its nonzero entries) and should not to be confused with the sparsity level~$r$ of the model~$x^\n$ (i.e., the number of atoms from the alphabet~$\A$ that make up~$x^\n$).}
    Note also that~$\A$  may be identified with an alphabet in~$\R^d$ with~$d=d_1 d_2$. One may also revise the definition of $\A$ by including $\|v\|_\infty = O(1/\sqrt{k})$ to ensure that the atoms that make up $x^\n$ are diffuse on their support. 
    Importantly, note that  {the} model~\eqref{eq:spca1} is exact \rev{, namely, we have complete access to the matrix $x^\n$. Often,~$x^\n$ is distorted by noise and we will study this general case in Section~\ref{sec:action}.}

\begin{example}[{\sc Sparse PCA}]\label{ex:2} \rev{In this example, we consider the model~\eqref{eq:spca1}
with $d_1=d_2=3$ and~$k=2$.}
\rev{To be specific,} we consider the {$2$-sparse} model 
\begin{align}
    x^\n := \frac{A_1^\n}{2}+\frac{A_2^\n}{2} \in \R^{3\times 3}, 
    \label{eq:rep1}
\end{align}
where the atoms $A_1^\n,A_2^\n\in \A$ are 
specified as  
\begin{align}
    & A_1^\n := u_1^\n (v_1^\n)^\top = \l[
    \begin{array}{ccc}
       \frac{1}{\sqrt{3}}   &
        \frac{1}{\sqrt{3}} &
      \frac{1}{\sqrt{3}}  
    \end{array}
    \r]^\top \cdot 
    \l[
    \begin{array}{ccc}
        0.3122 &
          0.95 &
         0
    \end{array}
    \r] \in \R^{3\times 3}, \nonumber\\
    & 
    A_2^\n:= u_2^\n (v_2^\n)^\top 
    = \l[
    \begin{array}{ccc}
        \frac{1}{\sqrt{2}} &
         -\frac{1}{\sqrt{2}}  & 
        0
    \end{array}
    \r]^\top \cdot 
    \l[
    \begin{array}{ccc}
       0   &
       0.95    &
       0.3122  
    \end{array}
    \r] \in \R^{3\times 3}. 
    \label{eq:atomsSparsePCAEx}
\end{align}
Recalling Definition~\ref{defn:slice}, let $\slice^\n$ be the two-dimensional slice formed by $\{A_1^\n,A_2^\n\}$. \rev{That is, $\mathcal{S}^\n = \mathrm{conv}(\{A_1^\n,A_2^\n,0\})$.} \rev{In particular, in view of~\eqref{eq:rep1}, note that~$x^\n \in \slice^\n$.} On the other hand, note also that $x^\n$ above  has the alternative decomposition 
\begin{align}
    x^\n & = c_1 u_1 v_1^\top+c_2 u_2 v_2^\top+c_3 u_3 v_3^\top\nonumber
    \qquad ( u_1v_1^\top, u_2v_2^\top , u_3 v_3^\top \in \A) 
    \\
    & = 0.1561 \cdot 
    \l[
    \begin{array}{ccc}
         \frac{1}{\sqrt{3}}  &
         \frac{1}{\sqrt{3}}& 
        \frac{1}{\sqrt{3}}
    \end{array}
    \r]^\top \cdot 
    \l[
    \begin{array}{ccc}
        1 &
         0 &
         0
    \end{array}
    \r]+
     0.6717 \cdot 
    \l[
    \begin{array}{ccc}
        0.9082 & -0.0918 &   0.4082
    \end{array}
    \r]^\top \cdot 
    \l[
    \begin{array}{ccc}
        0 & 1 &   0
    \end{array}
    \r]\nonumber\\
    & \qquad + 
       0.1561 \cdot 
    \l[
    \begin{array}{ccc}
        \frac{1}{\sqrt{2}}& 
         -\frac{1}{\sqrt{2}} & 
         0
    \end{array}
    \r]^\top \cdot 
    \l[
    \begin{array}{ccc}
         0 &
        0  &
         1
    \end{array}
    \r].
    \label{eq:example2}
\end{align}
By comparing the two alternative representations of $x^\n$ in~\eqref{eq:rep1} and~\eqref{eq:example2}, we 
 observe that  
\begin{align}
\gauge_{\A}(x^\n) \le \min\l( \frac{1}{2}+\frac{1}{2},  0.1561 + 0.6717 +  0.1561\r)  = 0.984 < 1, \qquad \textup{(see \eqref{eq:gaugeFcn1})}
\label{eq:gaugeFailsSparsePCA}
\end{align}
and thus the $2$-sparse decomposition of $x^\n$ in~\eqref{eq:rep1} is not  minimal, \rev{i.e., the decomposition in \eqref{eq:rep1} does not achieve $\gauge_\A(x^\n)$.} In fact, we may verify that the machine~\eqref{eq:cvxFacNoiseFre} fails to find any $2$-sparse decomposition for the model $x^\n$. 


We also remark that the failure of~\eqref{eq:cvxFacNoiseFre} in this example  persists even after imposing an incoherence requirement on the alphabet $\A$, i.e., {after} including $\|v\|_\infty =O(1/\sqrt{k})$ in~\eqref{eq:spca1}. Indeed, the true atoms $\{A_1^\n,A_2^\n\}$ in~\eqref{eq:atomsSparsePCAEx} are already sufficiently diffuse on their support. 
To close this example, we note that the potential failure of the gauge function theory, in the context of sparse PCA, has also been documented in~\cite{bach2008convex}.
\end{example}

\rev{In both Examples~\ref{ex:1} and~\ref{ex:2}, it is easy to verify that $x^\n/\gauge_\A(x^\n)$ belongs to a high- (rather than low-) dimensional exposed face of~$\hull(\A)$. While both Examples~\ref{ex:1} and~\ref{ex:2} fall under the broad umbrella of structured data factorization, it is also not difficult to find other \rev{examples} for which the classical gauge function theory fails beyond structured data factorization, e.g., see Section~\ref{sec:numerics}.} 

\section{\texorpdfstring{Gauge$_p$}{Gp} Function Theory and Main Results}\label{sec:new}

\rev{In Section~\ref{sec:old}, we reviewed the gauge function theory and highlighted its statistical limitations. This section aims to overcome some of the limitations of the classical theory by introducing a generalized gauge function theory, dubbed gauge$_p$ function theory. This generalized theory is the main contribution of this study.}

Below, in Section~\ref{sec:gaugep}, we will first introduce the gauge$_p$ function, which is the central object of the new theory. We then compare the gauge$_p$ function with the classical gauge function  in Section~\ref{sec:exGaugep}. In short, the gauge$_p$ function is a simple generalization of the gauge function that can tightly  control the sparsity of a model within the learning alphabet. We finally introduce the  new learning machine at the heart of the gauge$_p$ function theory in Section~\ref{sec:newMachine}, and {list} its statistical guarantees in Section~\ref{sec:newGaugeTheory}. This new learning machine uses, as the regularizer, the gauge$_p$ function instead of the classical gauge function. Aside from {its} statistical benefits, the computational {aspects} of the proposed  machine are discussed later in Section~\ref{sec:algorithm}. 

\subsection{\texorpdfstring{Gauge$_p$}{Gp} Function}\label{sec:gaugep}

Central to the generalized theory is a new notion of complexity for statistical models, dubbed gauge$_p$ function, which generalizes the gauge function in Definition~\ref{defn:gauge}. Gauge$_p$ function can tightly control the sparsity level of a model within the learning alphabet. Before defining the gauge$_p$ function, we begin below with a few {necessary} geometric concepts. To be specific, let us first introduce a geometric object which, as we will see shortly, generalizes the notion of the convex hull of a set.  

\begin{defn}[$\mathrm{conv}_p(\A)$]\label{defn:hall}
For an  alphabet $\A$ and integer $p$, we define
\begin{align}
    \hull_p(\A) :=  \bigcup_{\slice \in \Slice_p(\A)}  \slice,
    \label{eq:far-right}
\end{align}
to be the union of all slices of $\hull(\A)$ {formed by at most $p$ atoms.} The notation $\Slice_p(\A)$ above was introduced in Definition~\ref{defn:slice}. 
\end{defn}

For example, for the alphabet~$\A:=\{\pm e_i\}_{i=1}^3$ in Figure~\ref{fig:slices},~$\hull_2(\A)$ is the union of all colored triangles, some of which are hidden from the view.
We next record a  simple property of $\hull_p(\A)$. 

\begin{prop}[{\sc Nested hulls}]
\label{prop:propsHullP} 
Suppose that Assumption~\ref{a:alph}\ref{assumption:containsOrigin} {or~\ref{assumption:alphabet0}} is met. Then,
\begin{align}
    \bigcup_{0\le \tau \le 1} \tau \A = \hull_1(\A) \subset \hull_2(\A) \subset  \cdots \subset \hull_{d+1}(\A) = \hull_{d+2}(\A) = \cdots  = \hull(\A),
    \label{eq:nested}
\end{align}
where $\tau\A = \{\tau A: A\in \A\}$. 
\end{prop}
In words, {the sets $\{\hull_p(\A)\}_p$ provide a nested sequence of approximations to the alphabet $\A\subset \R^d$. That is, as~$p$ decreases, $\hull_p(\A)$ becomes an increasingly {better} approximation to~$\A$.} {In view of \eqref{eq:nested}, we note that the sets $\{\hull_p(\A)\}_{p=1}^d$  in~\eqref{eq:nested} might be nonconvex.} In contrast, the sets $\{\hull_p(\A)\}_{p\ge d+1}$ are  convex and, in fact, {all} coincide with $\hull(\A)$. Indeed, the identities in~\eqref{eq:nested} follow from an application of the Carath\'eodory theorem~\cite[Theorem 2.3]{barvinok2002course}. 

To each set~$\hull_p(\A)$ above, we  associate a  gauge$_p$ function in analogy with~\eqref{eq:gaugeFcn1}. 

\begin{defn}[{\sc Gauge$_p$ function}]\label{def:gauge_p} 
For an  alphabet $\A \subset \R^d$ and integer $p$, the corresponding gauge$_p$ function $\gauge_{\A,p}:\R^d\rightarrow\R$ is defined as  
\begin{align}
    \gauge_{\A,p}( x):=\inf \l\{t : x/t\in \hull_p(\A), \, t\ge 0  \r\}.
    \label{eq:gauge2Hull}
\end{align}
\end{defn}

\rev{Gauge$_p$ function generalizes the notion of gauge function in Definition~\ref{defn:gauge} in the sense that $\gauge_{\A,p}=\gauge_\A$ when $p> d$, see \eqref{eq:gaugeFcn1} and \eqref{eq:gauge2Hull}. In contrast, when $p\le d$, the gauge$_p$ function tightly controls the sparsity level of a model. That is,} $\gauge_p(x)$ is finite only when~$x$ has a $p$-sparse decomposition in the alphabet~$\A$. This and other basic properties of gauge$_p$ functions are collected below, and it is straightforward to prove them using Proposition~\ref{prop:propsHullP} and  Definition~\ref{def:gauge_p}.

\begin{prop}[{\sc Properties of gauge$_p$ functions}]\label{prop:gauge2EquivDefn}
Consider the gauge$_p$ function in Definition~\ref{def:gauge_p} for an integer~$p$ and an alphabet~$\A \subset \R^d$. Suppose that Assumptions~\ref{a:alph}\ref{assumption:containsOrigin} and~\ref{assumption:bounded} are both met. Then the following statements are true: 

\begin{enumerate}[label=(\roman*), itemsep = 0mm, topsep = 0mm, leftmargin = 12mm]
\item The gauge$_p$ function {has the} equivalent {definition} 
\begin{align}
    \gauge_{\A,p}(x):=  
    {\inf} \l\{\sum_{i=1}^p c_i \, :\,  x = \sum_{i=1}^p c_i A_i ,\, c_i \ge 0, \, A_i\in \A, \, \forall i\in [p] \r\}. 
    \label{eq:gauge2}
\end{align}

\item \label{prop:gauge_p:origin}
$\gauge_{\A,p}(x)=0$ if and only if $x=0$. 

\item \label{prop:gauge_p:infeasible}
If $x$ does not admit a $p$-sparse decomposition in the alphabet $\A$, then $\gauge_{\A,p}(x)=\infty$. 

\item \label{prop:gauge_p:sparse}
If $\gauge_{\A,p}(x) < \infty$, then any minimal decomposition of $x$ that achieves $\gauge_{\A,p}(x)$ in~\eqref{eq:gauge2} is   $p$-sparse {in the alphabet $\A$.} 

\item \label{prop:gauge_p:conjugate}
The convex conjugate {of $\gauge_{\A,p}$, denoted here} by~$\gauge_{\A,p}^* $, {is specified as}
\begin{align}
    \gauge_{\A,p}^*(z) := 
    \begin{cases}
    0 & \dual_\A(z) \le 1\\
    \infty & \text{otherwise}.
    \end{cases}
\end{align}

\item \label{prop:gauge_p:envelop}
The convex envelope of $\gauge_{\A,p}$ is the gauge function~$\gauge_\A$ in~\eqref{eq:gaugeFcn1}, i.e., $\gauge_{\A,p}^{**} = \gauge_{\A}$.

\item \label{prop:gauge_p:nested}
The gauge$_p$ functions {satisfy} the nested property
\begin{align}
    \gauge_{\A,1}(x) \ge \gauge_{\A,2}(x) \ge \cdots \ge \gauge_{\A,d+1}(x) {= \gauge_{\A,d+2}(x) = \cdots }= \gauge_\A(x),
    \label{eq:nestedGauges}
\end{align}
where  $\{\gauge_{\A,p}\}_{p=1}^d$ above may be nonconvex functions.
\end{enumerate}
\end{prop}

\rev{Let us take a moment to parse the above result: Proposition~\ref{prop:gauge2EquivDefn}\ref{prop:gauge_p:nested} posits that the gauge$_p$ function coincides with the gauge function for $p> d$.} Moreover, even when $p\le d$, the gauge$_p$ function still preserves certain properties of the gauge function, see Proposition~\ref{prop:gauge2EquivDefn}\ref{prop:gauge_p:origin}-\ref{prop:gauge_p:envelop}.  \rev{Crucially, when $p\le d$, the gauge$_p$ {function} directly controls the sparsity level, as articulated in Proposition~\ref{prop:gauge2EquivDefn}\ref{prop:gauge_p:infeasible} and~\ref{prop:gauge_p:sparse}. This property of the gauge$_p$ function remedies the key shortcoming of the classical gauge function concerning the lack of sparsity of the outcome.} Indeed, in the negative examples {of} Section~\ref{sec:old}, the {classical} gauge function failed to enforce sparsity, i.e., its minimal decomposition failed to be sparse at all. \rev{In contrast,} any minimal decomposition of~$x$ that achieves $\gauge_{\A,p}(x)$ is $p$-sparse by Proposition~\ref{prop:gauge2EquivDefn}\ref{prop:gauge_p:sparse}. 
Put differently, in the negative examples highlighted of
Section~\ref{sec:old}, {at least one} of the inequalities in~\eqref{eq:nestedGauges} {was} strict. We will further investigate these key differences in Section~\ref{sec:exGaugep} by comparing the gauge and gauge$_p$ functions for several learning alphabets. 

\rev{Let us also describe a natural interpretation of \eqref{eq:nestedGauges}:} In view of~\cite[Definition~2.2]{kerdreux2017approximate} or~\cite{ekeland1999convex},~$\gauge_{\A,1}$ is the ``most nonconvex'' among the gauge$_p$ functions, followed by~$\gauge_{\A,2}$, {then $\gauge_{\A,3}$} and so on. {Here,} the nonconvexity of gauge$_p$ functions is measured by {their} distance from their shared convex envelope~$\gauge_\A$, {see Proposition~\ref{prop:gauge2EquivDefn}\ref{prop:gauge_p:envelop}.} Another interesting perspective is offered by the the approximate Carath\'eodory theorem of Maurey~\cite{kerdreux2017approximate,pisier1981notes}. \rev{In principle, this theorem can quantify just how well the sets~$\hull_p(\A)$ approximate~$\hull(\A)$. However, we will not further pursue these connections.} 

\subsection{Why \texorpdfstring{Gauge$_p$}{Gp} Function?}\label{sec:exGaugep}

We earlier reviewed the gauge function in Definition~\ref{defn:gauge} as the notion of statistical complexity at the heart of the classical gauge function theory. We then introduced the gauge$_p$ function in Definition~\ref{def:gauge_p} as a new device for measuring   the  complexity of statistical models.

To develop a better understanding {of these concepts}, {we} next compare the gauge and gauge$_p$ functions for {a few} {learning} alphabets. Our discussion in this section is limited to the statistical benefits of gauge$_p$ functions and we defer {their} computational {aspects} 
to Section~\ref{sec:algorithm}. 

\begin{example}[{\sc Sparsity}]\label{ex:sparsity}
\rev{Consider the alphabet
\begin{align}
    \A := \{ \pm e_i\}_{i=1}^d \subset \R^d,
\end{align}
which is central to sparse signal processing and {high-dimensional} statistical inference~\cite{hastie2015statistical,candes2008introduction,mallat2008wavelet}. 
Here, $e_i\in \R^d$ is the $i^{\text{th}}$~canonical vector, {with its $i^{\text{th}}$ entry equal to one and the remaining entries equal to zero.}  
For this alphabet, a} simple calculation shows that 
\begin{align}
    \gauge_{\A,p}(x) & :=
    \begin{cases}
    \gauge_\A(x) = \|x\|_1 & \|x\|_0 \le p\\
    \infty & \text{otherwise},
    \end{cases}
    \label{eq:gaugepSparseRevisit}
\end{align}
where $\|x\|_0$ is the number of nonzero entries of $x$. \rev{Above, $\|\cdot\|_1$ stands for the  $\ell_1$-norm.}
In particular, it follows from~\eqref{eq:gaugepSparseRevisit}  that  $\gauge_{\A,p}=\gauge_\A$ when $p\ge d$.
{On the other hand, if $\|x\|_0> p$, then note that~$\gauge_{\A,p}(x)=\infty$.
This property of the gauge$_p$ function will later enable us to limit the learning {outcome} to sufficiently sparse models.} If one replaces the~$\ell_1$-norm in~\eqref{eq:gaugepSparseRevisit} with the~$\ell_2$-norm, then the convex envelope of the resulting function would coincide with the~$k$-support norm~\cite{argyriou2012sparse}, an alternative to elastic net regularization~\cite{zou2005regularization}.
We close by noting that, for this choice of  the alphabet $\A$, the relation $\gauge_{\A,p}=\gauge_\A$ for $p\ge d$ is  an improvement over {the conservative but} general result  in~\eqref{eq:nestedGauges}.
\end{example}

In Example~\ref{ex:sparsity}, observe that the gauge$_p$ function of a $p$-sparse model coincides with the corresponding gauge function.  However, in general, recall from~\eqref{eq:nestedGauges} that $\gauge_{\A,p}(x) \ge \gauge_\A(x)$ for an {arbitrary} alphabet $\A$ and model $x$. As we will see {below}, this inequality might be strict, which precisely explains  the failures of the classical gauge function theory in Section~\ref{sec:old}. To see this, let us continue {below} with the {manifold} example. 

\begin{example}[{\sc Manifold models, continued}]\label{ex:manCntd}
When the alphabet $\A$ is an embedded submanifold, we saw in Section~\ref{sec:old} {that} the gauge function $\gauge_\A$ might lose vital geometric details about the manifold~$\A$. In contrast,~$\gauge_{\A,1}$ captures far more information about the manifold~$\A$. \rev{More specifically, the gauge$_1$ value of a model is infinity unless that model is $1$-sparse, see Proposition~\ref{prop:gauge2EquivDefn}\ref{prop:gauge_p:infeasible}. That is, $\gauge_{\A,1}(x^\n)=\infty$, unless $x^\n=tA$ for some $t\ge 0$ and atom~$A\in \A$.} {As an example}, for the $1$-sparse model $x^\n$  in~\eqref{eq:exModelManifold}, it is easy to verify that 
\begin{align}
    \gauge_\A(x^\n) \le \frac{5}{8\sqrt{2}} < 1 = \gauge_{\A,1}(x^\n).
    \label{eq:strictManifold}
\end{align}
The strict inequality above means that the machine~\eqref{eq:mainCvx} fails to find any $1$-sparse decomposition of~$x^\n$, which in turn explains the failure of the classical gauge function theory in Section~\ref{sec:old}. {To see why~\eqref{eq:strictManifold} holds, note that} the far-left inequality in~\eqref{eq:strictManifold} was established in~\eqref{eq:gaugeFailManEx} and the identity in~\eqref{eq:strictManifold} is evident from a visual inspection of Figure~\ref{fig:manifolds}.
\end{example}

One can also verify that the strict inequality between gauge and gauge$_p$ functions also holds in Example~\ref{ex:2} about sparse PCA, thus explaining the failure of the gauge function in that example. 
\vspace{5pt}

\notebox{gray!20}
{The gauge$_p$ function in Definition~\ref{def:gauge_p} is a simple generalization of the classical gauge function that better controls the sparsity of a model within the learning alphabet. In every failed example of the classical gauge function theory, we observed that $\gauge_{\A}(x^\n)<\gauge_{\A,r}(x^\n)$. Consequently, in these examples, any minimal decomposition of $x^\n$ in the alphabet~$\A$ that achieves~$\gauge_\A(x^\n)$ cannot be $r$-sparse. In contrast, any minimal decomposition  of~$x^\n$ that achieves~$\gauge_{\A,r}(x^\n)$ is $r$-sparse by definition, see the result in Proposition~\ref{prop:gauge2EquivDefn}\ref{prop:gauge_p:sparse}. Motivated by this {encouraging} observation, we next introduce a new learning machine  that replaces the gauge function in~\eqref{eq:mainCvx} with a gauge$_p$ function.}

\subsection{A New Learning Machine}\label{sec:newMachine}

\rev{As reviewed in Section~\ref{sec:old}, the classical gauge function theory is a theory for (convex) statistical learning  that leverages the gauge function to promote sparsity within the learning alphabet. However, the classical gauge function might fail to enforce sparsity  and, in such cases, the gauge function theory fails too. To overcome the  statistical limitations of the classical gauge function theory, this section introduces a new learning machine, the main building block of which is  the gauge$_p$ function in Definition~\ref{def:gauge_p}. As we will see shortly,} the new learning machine is a simple generalization of the convex machine~\eqref{eq:mainCvx} which allows us to tightly control the sparsity of the learning outcome. 

To begin, consider  the (inexact) setup
\begin{align}
    y := \L(x^\n)+e ,\qquad x^\n\in \cone(\slice^\n), \,  \slice^\n\in \Slice_r(\A), \qquad \|e\|_2 \le \epsilon,
    \label{eq:fullIneaxctModel}
    \tag{inexact}
\end{align}
\rev{where $y \in \R^m$ is the available measurement, and $\epsilon$ reflects the inexactness of the setup.} In signal processing, for example, $\epsilon$ quantifies the noise level in our measurements~$y$. In statistical inference,~$e$ in the setup~\eqref{eq:fullIneaxctModel} is also assigned a probability distribution~\cite{efron2016computer,van2000asymptotic}, but we avoid this additional layer of complexity here. In particular, when the distribution assigned to $e$ is light-tailed, \rev{it often suffices to take $\epsilon$ large enough and then work with the \rev{setup}~\eqref{eq:fullIneaxctModel}.}

To learn the  model $x^\n$ in~\eqref{eq:fullIneaxctModel} or its sparse decomposition {in the alphabet $\A$}, {consider}
\begin{align}
    \min_x \, \| \L(x)-y\|_2^2 \quad \text{subject to} \quad  \gauge_{\A,p}(x)  \le \gauge_{\A,p}(x^\n),
    \label{eq:mainNonCvx0}
\end{align}
which uses the  gauge$_p$ function $\gauge_{\A,p}$ \rev{defined in~\eqref{eq:gauge2Hull}. Using the equivalent definition of gauge$_p$ function in~\eqref{eq:gauge2}, we can also reformulate the above learning machine as}
\rev{\begin{align}
     \inf\l\{ \Big\| \sum_{i=1}^p c_i \L(A_i) - y \Big\|_2^2 : \sum_{i=1}^p c_i \le \gauge_\A(x^\n),\, c_i \ge 0,\, A_i\in \A,\, i\in [p]  \r\}, 
\label{eq:mainNonCvx}\tag{$\mathrm{gauge}_p$}
\end{align}
where the infimum is over $\{c_i\}_{i=1}^p$ and $\{A_i\}_{i=1}^p$.  It is the above reformulation of~\eqref{eq:mainNonCvx0} that we will primarily work with in this paper.} \rev{In particular, if $\{\wh{c}_i,\wh{A}_i\}_{i=1}^p$ is a solution of \eqref{eq:mainNonCvx}, then the learning outcome of~\eqref{eq:mainNonCvx} is the $p$-sparse model $\wh{x}=\sum_{i=1}^p \wh{c}_i \wh{A}_i$. }

The computational aspects of the new learning machine are discussed later in Section~\ref{sec:algorithm}, where we provide a tractable numerical scheme for solving the  optimization problem~\eqref{eq:mainNonCvx}. 
As detailed below, \rev{the proposed machine \eqref{eq:mainNonCvx} generalizes both the classical machine~\eqref{eq:mainCvx} and the familiar $\ell_0$-pursuit approach.} 
\rev{
\begin{remark}[{\sc A generalization of \eqref{eq:mainCvx}}]
    \label{rem:genGaugeRem}
    The proposed machine~\eqref{eq:mainNonCvx} reduces to the classical convex machine~\eqref{eq:mainCvx} for $p\ge d+1$. Indeed, this claim follows from comparing the equivalent formulations~\eqref{eq:mainCvx0} and~\eqref{eq:mainNonCvx0}, and then invoking~Proposition~\ref{prop:gauge2EquivDefn}\ref{prop:gauge_p:nested}. On the other hand, when $p\le d$, the machine~\eqref{eq:mainNonCvx} only searches over $p$-sparse models and thus always returns a $p$-sparse solution by design. More specifically, the learning output of \eqref{eq:mainNonCvx} is the $p$-sparse model $\wh{x}=\sum_{i=1}^p \wh{c}_i \wh{A}_i$, where $\{\wh{c}_i,\wh{A}_i\}_{i=1}^p$ is a solution of~\eqref{eq:mainNonCvx}. This guaranteed sparsity of the learning outcome immediately rectifies the key failure of the classical gauge function theory, i.e., the  lack of sparsity that we observed in Examples~\ref{ex:1} and~\ref{ex:2}. 
\end{remark}

Suppose that $\epsilon=0$ in~\eqref{eq:fullIneaxctModel}, \rev{i.e., for simplicity, consider the exact \rev{setup} $\L(x^\n)=y$.} \rev{Recall the familiar $\ell_0$-pursuit, }
\begin{equation}
    \min_x \,\, \| \L(x)-y\|_2^2 \,\,\textup{ subject to}\,\, x \,\, \textup{having a $p$-sparse decomposition in the alphabet $\A$},
    \label{eq:l0machine}
    \tag{$\ell_0$-pursuit}
\end{equation}
which \rev{finds a $p$-sparse model $\wt{x}$ that satisfies $\L(\wt{x})=y$.} \rev{For instance, in the context of Example~\ref{ex:1}, the machine~\eqref{eq:l0machine} reduces to the familiar optimization problem in~\cite{bertsimas2020sparse}:
\begin{equation}
    \min_x \,\, \| \L(x)-y\|_2^2 \,\,\textup{ subject to}\,\, \|x\|_0\le p.
\end{equation}}
\begin{remark}[{\sc\rev{A generalization of the $\ell_0$-pursuit}}]\label{rem:sparsityPure} 
The new machine also generalizes the well-known $\ell_0$-pursuit:
\rev{For $\epsilon=0$, recall from} Remark~\ref{rem:genGaugeRem} \rev{that} the \rev{proposed} machine~\eqref{eq:mainNonCvx} finds a $p$-sparse model~$\wh{x}$ \rev{that satisfies}
$\L(\wh{x})=y$ \rev{and}~$\gauge_{\A,p}(\wh{x})\le \gauge_{\A,p}(x^\n)$.  
When~$p=r$, it not difficult to verify that the \rev{new} machine~\eqref{eq:mainNonCvx} \rev{reduces to} \eqref{eq:l0machine}, provided that~$\L$ is an injective map when \rev{its domain is} restricted to $r$-sparse models. This \rev{restricted injectivity} assumption is reasonable because \rev{the true model}~$x^\n$ would not be identifiable otherwise. (Recall from~\eqref{eq:fullIneaxctModel} that~$r$ denotes the sparsity level of the true model~$x^\n$.) \rev{When $p>r$, however, a key} advantage of \rev{the new machine}~\eqref{eq:mainNonCvx} 
is that the former is regularized with the gauge$_p$ function. Indeed, if~$p$ is relatively large, then the operator~$\L$ might not be injective when restricted to~$p$-sparse models. Consequently, 
\rev{searching over $p$-sparse models might} fail to \rev{recover} the true model~$x^\n$. 
\end{remark}

\newpage
We summarize this subsection so far in the following note: \vspace{3mm}} 

\notebox{gray!20}{
\rev{The proposed machine~\eqref{eq:mainNonCvx} generalizes both of the familiar machines~\eqref{eq:mainCvx} and~\eqref{eq:l0machine}. That is, \eqref{eq:mainNonCvx} interpolates two extremes: 

\begin{enumerate}[label=(\roman*), itemsep = 0mm, topsep = 0mm, leftmargin = 12mm]
    \item \rev{In one extreme}, when~$p\ge d+1$, the \rev{new} machine~\eqref{eq:mainNonCvx} reduces to the \rev{classical} convex machine~\eqref{eq:mainCvx}, see Remark~\ref{rem:genGaugeRem}.
    \item \rev{In} the other \rev{extreme}, when~$p=r$, the \rev{new} machine~\eqref{eq:mainNonCvx} coincides with~\eqref{eq:l0machine}, provided that~$\L$ is an injective map when restricted to $r$-sparse models. See Remark~\ref{rem:sparsityPure}. Recall from~\eqref{eq:fullIneaxctModel} that~$r$ denotes the sparsity level of the true model~$x^\n$.
\end{enumerate}

\rev{As we decrease the value of~$p$ from $d+1$ towards $r$}, the  machine~\eqref{eq:mainNonCvx} unlocks a range of \rev{potentially} new statistical and computational trade-offs. Informally speaking, as~$p$ decreases, the  statistical accuracy of the new machine~\eqref{eq:mainNonCvx}  improves and we will quantify this improvement \rev{later in this section}.
\rev{However,} the computational trade-offs \rev{of~\eqref{eq:mainNonCvx}} are more complex and \rev{we defer their discussion to} Section~\ref{sec:algorithm}.}}

\vspace{2mm}

In addition to generalizing~\eqref{eq:mainCvx} \rev{and~\eqref{eq:l0machine}}, the new machine~\eqref{eq:mainNonCvx} can also be interpreted as a natural extension of the Burer-Monteiro idea~\cite{burer2005local}: 

\begin{remark}[{\sc{\rev{Generalization of the Burer-Monteiro factorization}}}]\label{rem:bm-rem} 
A predecessor of the proposed machine~\eqref{eq:mainNonCvx}  appears in the context of matrix factorization. To explain their connection, for simplicity, consider the optimization problem
\begin{align}
    \min_{x}\,\,\|\L(x)-y\|_2^2\,\, \textup{ subject to}\,\, \tr(x) \le \gamma^2  \,\, \textup{ and}\,\, x\in \R^{d\times d}\, \textup{  is positive semi-definite},
    \label{eq:msensing}
\end{align}
where {$\L:\R^{d\times d}\rightarrow \R^m$} is a linear operator and $\gamma\ge 0$. Because~$x$ is {a} positive semi-definite {matrix} above,  $\tr(x)$   coincides  with the nuclear norm of~$x$ and the problem~\eqref{eq:msensing} is therefore a  variant of the (convex) learning machines {that are} widely studied in matrix completion and sensing~\cite{davenport2016overview}. Note that the {direct} computational cost of solving problem~\eqref{eq:msensing} grows  rapidly as the dimension~$d_1$ grows. {Instead,} it is common to solve the Burer-Monteiro factorization of problem~\eqref{eq:msensing}. More specifically, for an integer $p\le d$, the factorized version of problem~\eqref{eq:msensing} is
\begin{align}
    \min_{u\in \R^{d\times p}}\,\, \|\L(uu^\top)-y\|_2^2 \,\, \textup{ subject to}\,\, \|u\|_{\mathrm{F}} \le \gamma,
    \label{eq:factorizedRem}
\end{align}
where $\|\cdot\|_{\mathrm{F}}$ stands for the Frobenius norm. {Above,we also} used the fact that $\tr(uu^\top)=\|u\|_{\mathrm{F}}^2$. When $p$ is sufficiently small, {solving the} problem~\eqref{eq:factorizedRem} can offer substantial savings in computational speed and storage, compared to a direct implementation of the problem~\eqref{eq:msensing}. This idea has been {extensively studied for} matrix-valued learning problems~\cite{burer2005local,chi2019nonconvex,boumal2020deterministic,haeffele2015global,recht2010guaranteed,li2016super}.

It is not difficult to verify that the factorized problem~\eqref{eq:factorizedRem} coincides with~\eqref{eq:mainNonCvx} for the choice of alphabet~$\A:=\{uu^\top: \|u\|_2=1\}$. In this sense, the proposed machine~\eqref{eq:mainNonCvx} extends the successful idea of Burer-Monteiro factorization to any learning alphabet. 
\end{remark}

\rev{It is worth noting that the Burer-Monteiro factorization is primarily concerned with those values of $p$ for which~\eqref{eq:msensing} and~\eqref{eq:factorizedRem} have the same minimizers. In contrast, as to be seen shortly, we are often interested in those values of~$p$ for which the two problems have different minimizers. 

Let us also add that it is often possible to remove the constraints in problem~\eqref{eq:factorizedRem}, provided that {$p \lesssim \rank(x^\n)$.} Here,~$x^\n$ is the hidden model that satisfies~$\L(x^\n)\approx y$ \rev{and $\lesssim$ hides any constant factors.} In this regime, {which is} known as the {thin}  Burer-Monteiro factorization, a generic operator $u\rightarrow\L(uu^\top)$ is often injective, thus obviating the need for regularization \rev{with} $\|u\|_{\mathrm{F}}$, see for instance~\cite{eftekhari2020implicit}. Finally, let us revisit the failed applications of the classical gauge function theory in Section~\ref{sec:old}.}

\begin{example}[{\sc Examples~\ref{ex:1} and~\ref{ex:2}, revisited}] 
In these negative toy examples, we saw in Section~\ref{sec:old} that the convex machine~\eqref{eq:mainCvx} failed to find any $r$-sparse decomposition of the model $x^\n$ in the learning alphabet~$\A$. In contrast, it is not difficult to verify that the new machine~\eqref{eq:mainNonCvx}  successfully finds the (unique) $r$-sparse decomposition of $x^\n$ in each example ($p=r$ and $\L = \mathrm{id}$). Here,~$\mathrm{id}$ stands for the identity operator.
In general, the machine~\eqref{eq:mainNonCvx} always returns a  $p$-sparse {model}, see Proposition~\ref{prop:gauge2EquivDefn}\ref{prop:gauge_p:sparse}. \end{example}

\subsection{Statistical Guarantees}\label{sec:newGaugeTheory}

In Section~\ref{sec:newMachine}, we introduced the machine~\eqref{eq:mainNonCvx} as a generalization of the convex machine~\eqref{eq:mainCvx}.
This section will develop some statistical guarantees for {this} new learning machine as part of a generalized {gauge function} theory. The first  result of this section, Lemma~\ref{lem:nonCvxLearn} below, provides  certificates of correctness for the new machine~\eqref{eq:mainNonCvx}. \rev{Lemma~\ref{lem:noisyLearnGauge2} later extends Lemma~\ref{lem:nonCvxLearn} to account for noise. Lastly, Theorem~\ref{thm:constructOpt} shows that these certificates exist in certain generic learning problems. 

Let us start with a lemma that posits that the machine~\eqref{eq:mainNonCvx} succeeds if certain  certificates of correctness exist. This result is reminiscent of Lemma~\ref{lem:duality} in the classical gauge function theory.}

\begin{lem}[{\sc Correctness certificates, \rev{exact setup}}]\label{lem:nonCvxLearn}
With $\epsilon=0$, consider the model~$x^\n$ in~\eqref{eq:fullIneaxctModel} and {suppose that} the alphabet~$\A$ satisfies Assumptions~\ref{a:alph}\ref{assumption:containsOrigin} and~\ref{assumption:bounded}.  If $x^\n=0$, then the machine~\eqref{eq:mainNonCvx} correctly returns~$0$. \rev{That is, $x^\n=\sum_{i=1}^p \wh{c}_i\wh{A}_i = 0$, where $\{\wh{c}_i,\wh{A}_i\}_{i=1}^p$ is a solution of the optimization problem~\eqref{eq:mainNonCvx}.}

Otherwise, \rev{suppose that} $p\ge r$, where $r$ is the sparsity level of $x^\n$ in~\eqref{eq:fullIneaxctModel}. For every slice~$\slice\in \Slice_{p}(\A)$, \rev{suppose also} that (one of) the following holds: 
\begin{enumerate}[label=(\roman*), itemsep = 0mm, topsep = 0mm, leftmargin = 12mm]
        
    \item If $x^\n/\gauge_{\A,p}(x^\n) \in \slice$, the linear map~$\L$  is injective \rev{when its domain is restricted to {the slice} ${\slice}$.}
    
    \item If $x^\n / \gauge_{\A,p}(x^\n) \notin \slice$,  the point $x^\n/\gauge_{p,\A}(x^\n)$ and the slice $\slice$ are separated along $\range(\L^*)$, i.e., there exists a correctness certificate~$Q_\slice\in \range(\L^*)$ such that
    \begin{align}
        & \l\langle Q_\slice, x- \frac{x^\n}{\gauge_{\A,p}(x^\n)} \r\rangle <0,
        \qquad \forall x\in \slice.
        \label{eq:certNonCvx}
    \end{align}
\end{enumerate}
Then the machine~\eqref{eq:mainNonCvx} returns~$x^\n$ and a $p$-sparse decomposition of~$x^\n$ in~$\A$.
 \rev{That is, $x^\n=\sum_{i=1}^p \wh{c}_i \wh{A}_i$, where $\{\wh{c}_i,\wh{A}_i\}_{i=1}^p$ is a solution of the optimization problem~\eqref{eq:mainCvx}.}
\end{lem}

The proof of Lemma~\ref{lem:nonCvxLearn} largely mirrors that of Lemma~\ref{lem:duality} for the convex machine~\eqref{eq:mainCvx}. \rev{This symmetry justifies our unusual perspective when we reviewed the classical theory in Section~\ref{sec:old}.} In Lemma~\ref{lem:nonCvxLearn}, the correctness certificates for the machine~\eqref{eq:mainNonCvx} replace the correctness certificate for the convex machine~\eqref{eq:mainCvx}. The remark below compares the two types of certificates. 

\begin{remark}[{\sc Convex versus nonconvex learning}]\label{rem:cvxVsNoncvx}

\rev{When $p > d$, the new machine~\eqref{eq:mainNonCvx} reduces to the convex machine~\eqref{eq:mainCvx}, see Remark~\ref{rem:genGaugeRem}.}  From Lemma~\ref{lem:duality}, recall that the convex machine \eqref{eq:mainCvx} succeeds if  $x^\n/\gauge_\A(x^\n)$ is separated along $\range(\L^*)$ from the rest of $\hull(\A)$. However, when $p\le d$, the new machine~\eqref{eq:mainNonCvx} succeeds under the weaker requirement that $x^\n/\gauge_{\A,p}(x^\n)$ is separated along $\range(\L^*)$ from every  slice of $\hull(\A)$ formed by at most $p$ atoms. 
\end{remark}

\rev{For a family of generic learning problems, we will later construct the correctness certificates  in Lemma~\ref{lem:nonCvxLearn}. This will, in turn, guarantee the success of the new machine~\eqref{eq:mainNonCvx}.} It is not difficult to extend Lemma~\ref{lem:nonCvxLearn} to account for an inexact setup, i.e., the case where $\epsilon >0$ in~\eqref{eq:fullIneaxctModel}.

\rev{
\begin{lem}[{\sc Correctness certificates, inexact setup}]
\label{lem:noisyLearnGauge2}
Consider the model $x^\n$ in~\eqref{eq:fullIneaxctModel} and {suppose that} the alphabet~$\A$ satisfies Assumptions~\ref{a:alph}\ref{assumption:containsOrigin} and~\ref{assumption:bounded}. If $\gauge_{\A,p}(x^\n)=0$, then the machine~\eqref{eq:mainNonCvx} correctly returns $0$. That is, $x^\n=\sum_{i=1}^p \wh{c}_i \wh{A}_i=0$, where $\{\wh{c}_i,\wh{A}_i\}_{i=1}^p$ is a solution of the optimization problem~\eqref{eq:mainNonCvx}. Otherwise, suppose that $p\ge r$,
where $r$ is the sparsity level of $x^\n$ in~\eqref{eq:fullIneaxctModel}. For every slice $\slice\in \Slice_p(\A)$, suppose also that (one of) the following holds:
    
\begin{enumerate}[label=(\roman*), itemsep = 0mm, topsep = 0mm, leftmargin = 12mm]
        \item If $x^\n/\gauge_{\A,p}(x^\n) \in \slice$, the linear map $\L$ is injective when restricted to the slice $\slice$ and  its smallest singular value is at least $\s>0$.

        \item  If $x^\n / \gauge_{\A,p}(x^\n) \notin \slice$,  the point $x^\n/\gauge_{p,\A}(x^\n)$ and the slice  $\slice$    are well-separated along $\range(\L^*)$, i.e., there exists
        a correctness certificate $Q_\slice\in \R^d$ and a vector $q_\slice \in \R^m$ such that 
    \begin{align}
        & Q_\slice = \L^*(q_\slice), \nonumber\\
        & \|q_\slice\|_2 \le q, \nonumber\\
        & \l\langle Q_\slice, x- \frac{x^\n}{\gauge_{\A,p}(x^\n)} \r\rangle < - \g \l\| x - \frac{x^\n}{\gauge_{\A,p}(x^\n)} \r\|_2, 
        \qquad \forall x\in \slice.
        \label{eq:certNonCvxNoise}
    \end{align}
    \end{enumerate}
    Finally, let the $p$-sparse model $\wh{x}$ be a learning outcome of the machine~\eqref{eq:mainNonCvx}. That is, $\wh{x}=\sum_{i=1}^p \wh{c}_i \wh{A}_i$, where $\{\wh{c}_i,\wh{A}_i\}_{i=1}^p$ is a solution of the optimization problem~\eqref{eq:mainNonCvx}.
    Then it holds that 
    \begin{align}
        \| \wh{x} - x^\n \|_2 \le  
        \max\l( \frac{2\epsilon}{\s},   \frac{2\epsilon q  }{\g} \r).
        \label{eq:appxCorrectNoisy}
    \end{align}
\end{lem}
}


{Finding}  the  certificates prescribed in Lemma~\ref{lem:nonCvxLearn} is a problem-specific task, {similar to} the {classical} gauge function theory {in Section~\ref{sec:old}}. Nevertheless, this section provides a somewhat general recipe for constructing these correctness certificates. To that end, {we begin with} a few definitions. 

\begin{defn}[{\sc Angle of a cone}]\label{defn:angleConeDefn} The angle $\angle \K \rev{\in [0,\pi] }$ of a closed cone $\K \subset \R^d$ satisfies
\begin{align}
    \cos(\angle\K) & := \max_{u\in \K \cap \S^{d-1} } \min_{u'\in \K \cap \S^{d-1} } \langle u,u' \rangle,
    \label{eq:coneAngleDefn}
\end{align}
where $\S^{d-1}$ denotes the unit sphere in $\R^d$. 
\end{defn}

\rev{As an example, the positive orthant in $\R^2$ has the angle $\pi/4$. That is, $\angle \R_+^2=\pi/4$. We also recall the Hausdorff metric below~\cite{rockafellar2009variational}. This concept is visualized in Figure~\ref{fig:coneAngle}.}

\rev{ \begin{defn}[{\sc Hausdorff distance}]\label{defn:hausdorff}
The Hausdorff distance of two sets $\mathcal{S}$ and $\mathcal{S}'$  is
\begin{align}
    \dist_H(\mathcal{S},\mathcal{S}') := \max\l( \max_{s\in \mathcal{S}} \min_{s'\in \mathcal{S}'} \|s-s'\|_2\, ,\,  \max_{s'\in \mathcal{S}'} \min_{s\in \mathcal{S}} \|s-s'\|_2 \r)
\end{align}
\end{defn}}

Next,  recall that the metric entropy of a set is a measure for how large or complex that set is~\cite{ledoux2013probability}.



\begin{figure}
    \centering
  \includegraphics[width=.25\textwidth]{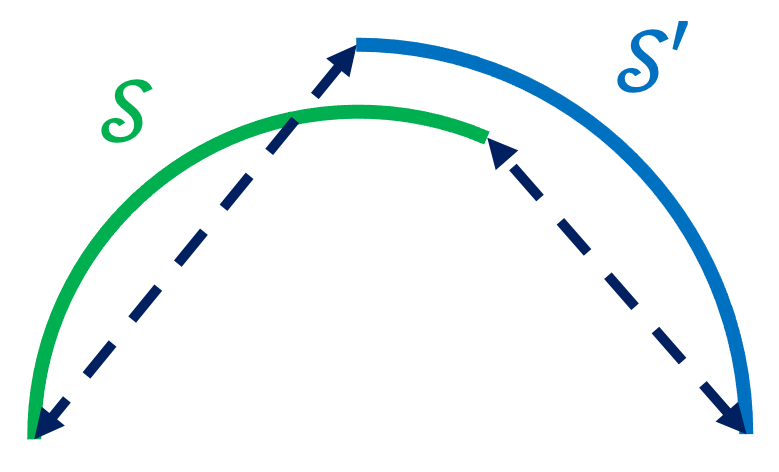} \hspace{2cm}
    \includegraphics[width=.2\textwidth,height=.2\textwidth]{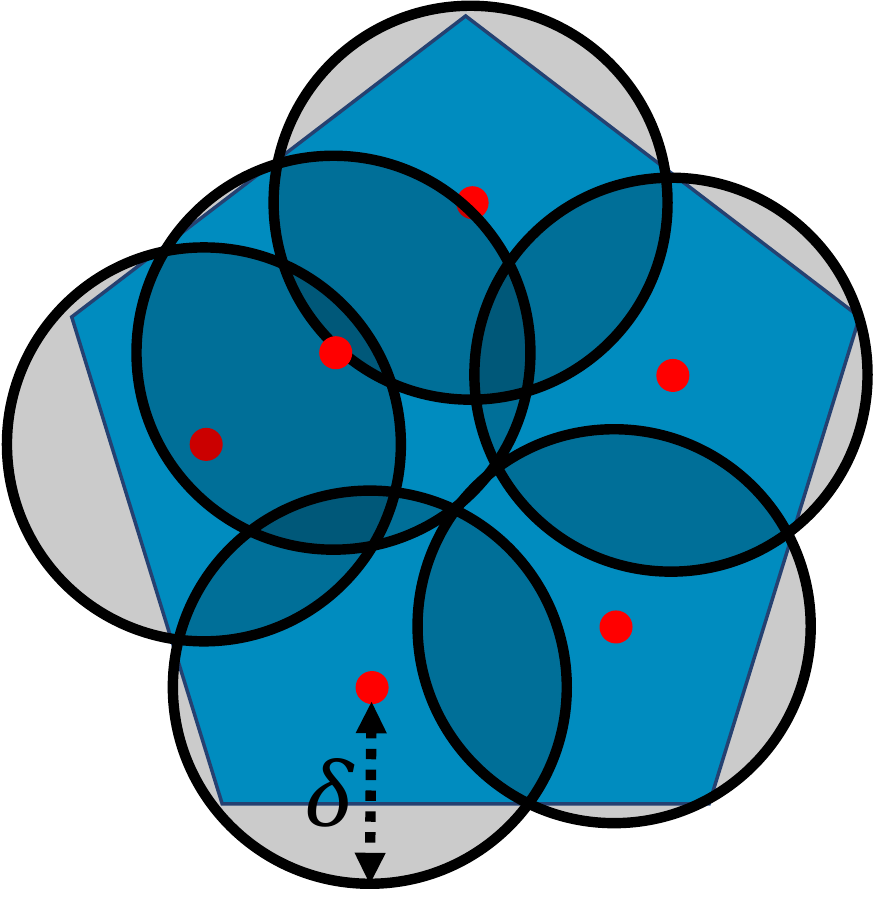}
    \\
    \includegraphics[width=.38    \textwidth,height=.2\textwidth]{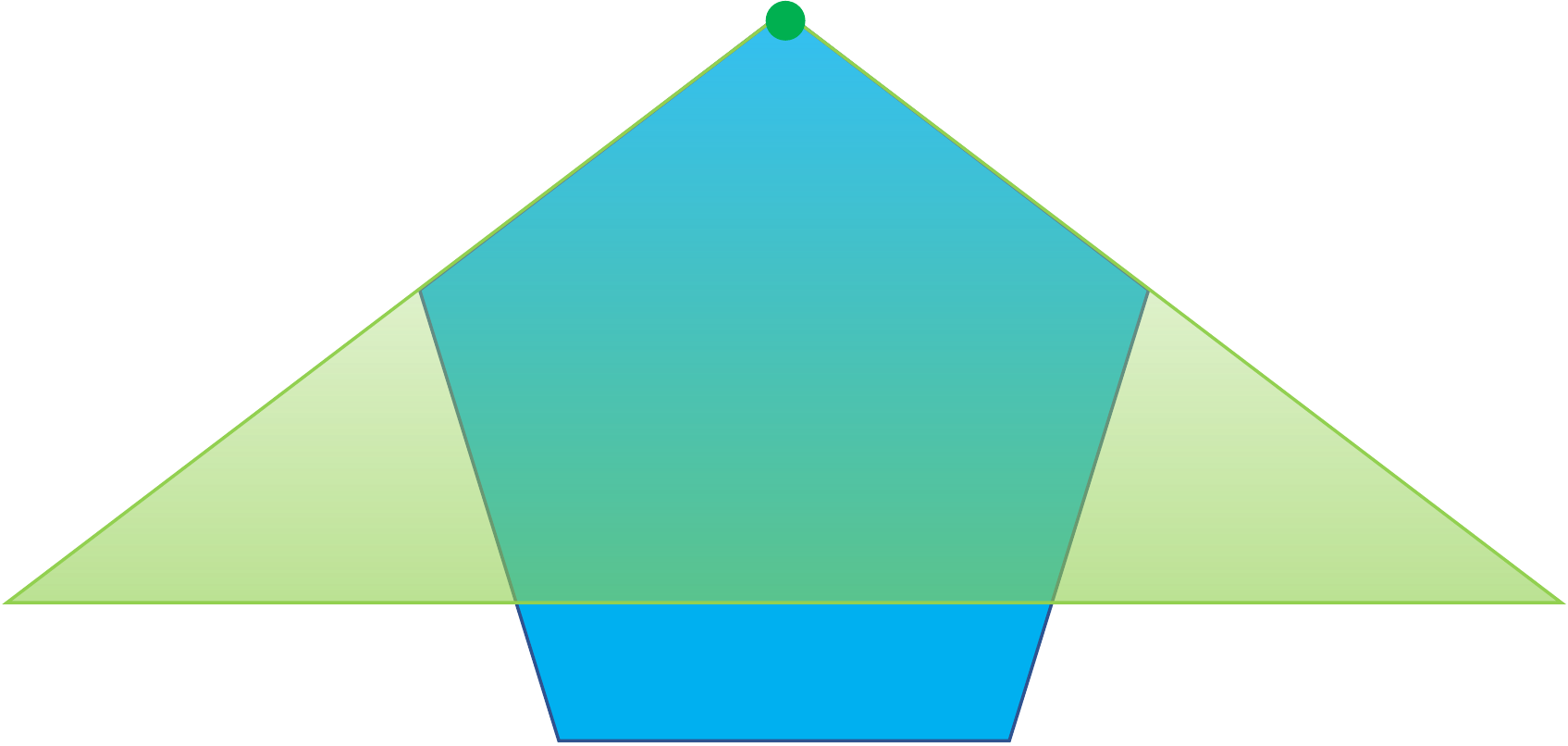}
    \hspace{1.5cm}
    \includegraphics[width=.2\textwidth,height=.2\textwidth]{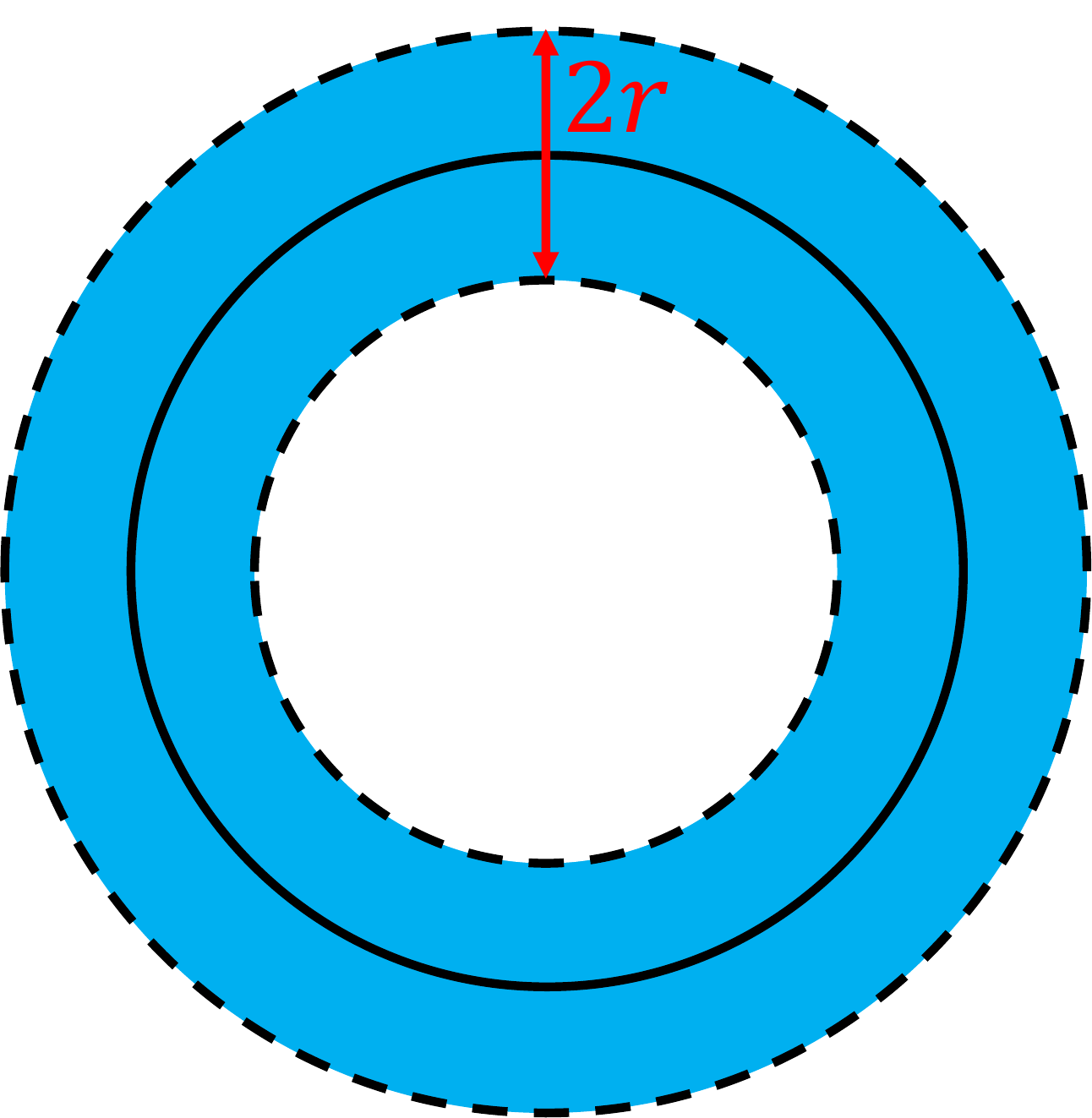}
    
    \caption{\footnotesize{\rev{This figure visualizes some of the geometric concepts in this work: 
    In the top left panel, the Hausdorff distance between the blue and green arcs equals the length of the longer dashed line segment, see Definition~\ref{defn:hausdorff}.
    The (red) dots in the top right panel form a $\delta$-net for the (blue) trapezoid, see Definition~\ref{defn:entropy}. The (green) triangle in the bottom left panel is a section of the tangent cone of the (blue) trapezoid at the (green) dot, see Proposition~\ref{prop:criticalSpecial}. Lastly, in the bottom right panel, the reach of the circle in solid black is its radius, see Definition~\ref{defn:reach}. The blue annulus shows the normal bundle of radius $r$. } }}
    \label{fig:coneAngle}
\end{figure}


\begin{defn}[{\sc Metric entropy}]\label{defn:entropy}
\rev{Consider a metric space with distance $\dist$.
In this space, consider a set $\mathcal{I}$ and its subset $\mathcal{I}'$ that satisfy the following: For every $x\in \mathcal{I}$, there exists
$x'\in \mathcal{I}'$ such that $\dist(x,x')\le \d$. The set $\mathcal{I}'$ is called a $\d$-net for the set $\mathcal{I}$ with respect to the metric $\dist$.} 
\rev{There are often many $\d$-nets for a set $\mathcal{I}$. In particular, let~$\cover(\mathcal{I},\dist,\delta)$ denote a minimal $\d$-net for~$\mathcal{I}$, i.e., the one with the smallest size. The logarithm of the size of this net is called the metric entropy of the set $\mathcal{I}$, denoted by $\entropy(\mathcal{I},\dist,\d)$. That is, $\entropy(\mathcal{I},\dist,\d):=\log [|\cover(\mathcal{I},\dist,\d)|]$.} 
\end{defn}

\rev{As an example, the red dots in Figure~\ref{fig:coneAngle} form a $\delta$-net for the (blue) trapezoid.}
In linear inverse problems, it is not uncommon for the linear operator to be random, see for instance~\cite{candes2008introduction,davenport2016overview,ahmed2013blind}. For our purposes, we quantify the randomness of the operator $\L$ as follows.

\begin{defn}[{\sc Probabilistic {restricted injectivity}}] \label{def:isometry}
\rev{The random linear operator $\L:\R^d\rightarrow\R^m$ satisfies the probabilistic $\d$-restricted isometry property if $\L$ is a near-isometry when restricting its domain to a low-dimensional subspace with high-probability. More specifically, for $\d\in [0,1)$, we say that $\L$ satisfies the probabilistic $\d$-RIP if an arbitrary subspace  $\mathcal{U}\subset \R^d$ with $\dim(\mathcal{U})\le C m \d^2$ satisfies
\begin{align}
    \rev{(1-\delta) \|u\|_2 \le \| \L(u) \|_2 \le (1+\d)\|u\|_2, \qquad \forall u\in \mathcal{U},}
    \label{eq:isoDefn}
\end{align}
except with a probability of at most $\exp(-C'\d^2m)$. Here, $C,C'$ are  universal constants and the probability is over the choice of the operator $\L$.}
\end{defn}

For example, {consider} a matrix {populated by independent Gaussian random variables, which} {can} be identified with a random linear operator. \rev{This matrix satisfies the probabilistic RIP} {if \rev{it} is sufficiently flat and properly scaled,} see Appendix~B.4.
Random matrix theory~\cite{vershynin2010introduction} offers  \rev{similar} statements with broad applications in {statistical inference and signal processing}~\cite{candes2008introduction}.

\rev{Note also that the two-sided nature of \eqref{eq:isoDefn} does not restrict the generality of Definition~\ref{def:isometry}. Indeed, consider an operator $\L'$ that satisfies $0< a\|u\|_2\le \|\L'(u)\|_2$ instead of~\eqref{eq:isoDefn}. Then we can verify that~$\L= \sqrt{2/(a^2+\|\L'\|^2)}\L'$ satisfies the probabilistic $\d$-RIP with $\d=\sqrt{(\|\L'\|^2-a^2)/(\|\L'\|^2+a^2)}$, where $\|\L'\|$ is the operator norm of~$\L'$.} 
Finally, let us define the notion of critical angle below. \rev{As we will see shortly, the critical angle quantifies how difficult it is to  construct the correctness certificates prescribed in Lemma~\ref{lem:nonCvxLearn}.} 

\begin{defn}[{\sc Critical angle}]\label{defn:critAngle} For an integer $p$,  alphabet $\A$ and  model $x^\n$,   \rev{we define}
\begin{align}
    \Slice_{x^\n,p}(\A) := \l\{ \slice \in \Slice_p(\A): x^\n/\gauge_{\A,p}(x^\n) \notin \slice \r\}.
    \label{eq:slicesX}
\end{align}
\rev{In words, $\Slice_{x^\n,p}(\A)$ collects all the slices of $\hull(\A)$ that are formed by at most $p$ atoms and do not contain the point $x^\n/\gauge_{\A,p}(x^\n)$, see Definitions~\ref{defn:slice} and~\ref{def:gauge_p}. The  critical angle of the alphabet $\A$ with respect to the model $x^\n$, denoted by $\theta_{x^\n,p}(\A)\in [0,\pi]$,} is then defined as 
\begin{align}
    \theta_{x^\n,p} = \theta_{x^\n,p}(\A) := \sup\l\{  \angle \cone\l(\slice - \frac{x^\n}{\gauge_{\A,p}(x^\n)} \r): \slice \in \Slice_{x^\n,p}(\A) \r\},
    \label{eq:criticalAngle}
\end{align}
with the conventions that $0/0=0$, and $\theta_{x^\n,p}(\A) = 0$ if $\Slice_{x^\n,p}(\A)$ is empty. 
\end{defn}

\rev{To help visualize this new notion, the next result relates the critical angle to a familiar quantity in convex statistical learning: Tangent cones of the set $\hull(\A)$.}


\begin{prop}[{\sc \rev{Example for critical angle}}]\label{prop:criticalSpecial}
Consider an alphabet $\A\subset \R^d$, an integer $p\ge d+1$, and a model $x^\n\in \R^d$.  The critical angle {of the alphabet $\A$ with respect to the model $x^\n$} is bounded by twice the angle of the corresponding tangent cone, i.e., $$\theta_{x^\n,p}(\A)\le 2 \cdot \angle   \cone\l(\A-\frac{x^\n}{\gauge_{\A}(x^\n)}\r).$$ \rev{Here, $\cone(\A-x^\n/\gauge_\A(x^\n))$ is the tangent cone of $\hull(\A)$ at the point~$x^\n/\gauge_\A(x^\n)$~\cite{rockafellar2009variational}. The angle of this tangent cone, which appears above, should be calculated according to Definition~\ref{defn:angleConeDefn}.}
\end{prop}

As we will see shortly, the smaller the critical angle is, the easier it is to construct the correctness certificates prescribed  in Lemma~\ref{lem:nonCvxLearn}. 
Equipped with Definitions~\ref{defn:angleConeDefn}-\ref{defn:critAngle}, we now  present the last main result of this section. Informally speaking, \rev{the result below states that the new learning machine~\eqref{eq:mainNonCvx} succeeds when the linear operator~$\L:\R^d\rightarrow\R^m$ in~\eqref{eq:fullIneaxctModel} is generic  and $m$ is large enough. For simplicity, the result below is limited to the exact setup, i.e., below we set $\epsilon=0$ in~\eqref{eq:fullIneaxctModel}.}

\begin{thm}[{\sc Exact recovery}]\label{thm:constructOpt}
Consider the model $x^\n$ in~\eqref{eq:fullIneaxctModel} with~$\epsilon=0$, and assume that the alphabet $\A$ satisfies Assumptions~\ref{a:alph}\ref{assumption:containsOrigin} and~\ref{assumption:bounded}. \rev{Assume also that the corresponding critical angle in Definition~\ref{defn:critAngle} satisfies $\theta_{x^\n,p}(\A)<\frac{\pi}{2}$.}
Let us equip $\Slice_{p}(\A)$ in~\eqref{eq:far-right} with the pseudo-metric\footnote{For a pseudo-metric,  $\dist(x,x')=0$ does not imply that $x=x'$.} that assigns the distance
    \begin{align}
        \dist_p(\slice,\slice') := 
        \dist_H\l( \cone\l(\slice - \frac{x^\n}{\gauge_{\A,p}(x^\n)} \r) \cap \mathbb{S}^{d-1},  \cone\l(\slice' - \frac{x^\n}{\gauge_{\A,p}(x^\n)} \r) \cap \mathbb{S}^{d-1} \r),
        \label{eq:metricThm}
    \end{align}
to every pair of slices $\slice,\slice'\in \Slice_{p}(\A)$, with the convention that $0/0=0$.  \rev{Above, $\mathbb{S}^{d-1}$ is the unit sphere and $\dist_\H$ denotes the (Euclidean) Hausdorff distance between two sets~\cite{rockafellar2009variational}. Lastly, for $\d\in [0,1)$, suppose that the random linear operator $\L:\R^d\rightarrow\R^m$ satisfies the probabilistic $\d$-RIP, see Definition~\ref{def:isometry}. If $\gauge_{\A,p}(x^\n)=0$, then the machine~\eqref{eq:mainNonCvx} correctly returns~$0$. That is, $x^\n=\sum_{i=1}^p \wh{c}_i\wh{A}_i = 0$, where $\{\wh{c}_i,\wh{A}_i\}_{i=1}^p$ is a solution of the optimization problem~\eqref{eq:mainNonCvx}. Otherwise, suppose that $p\ge r$, where $r$ is the sparsity level of $x^\n$ in~\eqref{eq:fullIneaxctModel}. Suppose also that
\begin{align}
 m \ge \rev{ \frac{1}{\l[\cos(\theta_{x^\n,p})\r]^{2}} \l(Cp+C'\entropy\l(\Slice_p(\A),\dist_p,\frac{\cos(\theta_{x^\n,p})}{2\max(\|\L\|_\op^2,1)}\r) \r)}.
    \label{eq:ddprimeSmall}
\end{align}
Then the machine~\eqref{eq:mainNonCvx} returns $x^\n$ and a $p$-sparse decomposition of $x^\n$ in $\A$, except with a probability of at most $\exp(-[\cos(\theta_{x^\n,p}) ]^2 m/C')$. That is, with high probability, it holds that $x^\n=\sum_{i=1}^p \wh{c}_i \wh{A}_i$, where $\{\wh{c}_i,\wh{A}_i\}_{i=1}^p$ is a solution of the optimization problem~\eqref{eq:mainCvx}. Here,  $C,C'$ are  universal constants and $\|\L\|_{\mathrm{op}}$ is the operator norm of~$\L$}.
\end{thm}

The proof technique of Theorem~\ref{thm:constructOpt} appears to be new in this context. More specifically, the proof relies on a  covering argument, where we form a fine {net} for all relevant slices of~$\hull(\A)$, with respect to the metric~$\mathrm{dist}_p$ in~\eqref{eq:metricThm}. We then explicitly construct a correctness certificate for each slice, with high probability over the choice of the random operator~$\L$. The failure probabilities are added up via a union bound. {To complete the proof,} we finally show that {the constructed} certificates  qualify as optimality certificates for \emph{all} relevant slices of~$\hull(\A)$, even those not present in the {net}. 

As {discussed} below, we may  consider Theorem~\ref{thm:constructOpt} as  a statistical guarantee for the new machine~\eqref{eq:mainNonCvx} that {is} analogous {to a similar} guarantee for the convex machine~\eqref{eq:mainCvx}. 
 \rev{ The regime of interest in Theorem~\ref{thm:constructOpt} is
\begin{align}
      & p \lesssim r, \quad \text{and} \quad
      0 \le \theta_{x^\n,p} < \frac{\pi}{2},\nonumber\\
      & \text{and}\quad
      \entropy\l(\Slice_p(\A),\dist_p,\frac{\cos(\theta_{x^\n,p})}{2\max(\|\L\|_\op^2,1)}\r) \lesssim r,
      \label{eq:regimeNew}
 \end{align}
in which, for brevity, the symbol $\lesssim$ suppresses any factors that might depend on the alphabet~$\A$. An example of the regime~\eqref{eq:regimeNew} is given in Section~\ref{sec:manifolModels} with $r=p=1$. In that example, the factors suppressed in~\eqref{eq:regimeNew} are given explicitly and their dependence on $\A$ is through elementary geometric attributes, e.g., intrinsic dimension and volume.}

\rev{In the regime~\eqref{eq:regimeNew}, Theorem~\ref{thm:constructOpt}  predicts that the machine~\eqref{eq:mainNonCvx} successfully recovers the {true} model~$x^\n$ and {returns a} $p$-sparse decomposition of~$x^\n$ in the alphabet $\A$, provided that $m \gtrsim r$. That is, in the regime~\eqref{eq:regimeNew}, solving the optimization problem~\eqref{eq:mainNonCvx} with high probability recovers~$x^\n$ and finds a $p$-sparse decomposition of~$x^\n$ if $m\gtrsim r$.   Recall that~$m$ is the number of observations and~$r$ is the sparsity level of~$x^\n$ within the alphabet~$\A$. In contrast, it is not difficult to verify that $m\ge \dim(\face^\n)$ is necessary (rather than sufficient) for the classical convex machine~\eqref{eq:mainCvx} to recover the model~$x^\n$. Here, $\face^\n$ is an exposed face of $\hull(\A)\subset \R^d$ that passes through $x^\n/\gauge_\A(x^\n)$, see Definition~\ref{defn:face}. In particular, in the negative toy examples of Section~\ref{sec:old}, one can verify that $\dim(\face^\n)= d$ which, in turn, necessitates $m\ge d$.

The informal discussion above highlights the potential benefits of the new machine in the regime~\eqref{eq:regimeNew}. Beyond this discussion, the entropy number and the critical angle in Theorem~\ref{thm:constructOpt} should be calculated on a case-by-case basis by taking into account the geometry of the learning alphabet $\A$. One such case is presented in the next section.

Finally, in the regime $p\ge d+1$,  the machine~\eqref{eq:mainNonCvx} reduces to the convex machine~\eqref{eq:mainCvx}, see Remark~\ref{rem:cvxVsNoncvx}. In particular, our bound in~\eqref{eq:ddprimeSmall} is too conservative in this regime and reads as~$m\gtrsim d$. Instead, with high probability, a classical result guarantees that the machine~\eqref{eq:mainCvx} recovers~$x^\n$ if $m \ge w(\Omega)^2+1$, e.g., see~\cite[Corollary~3.3.1]{chandrasekaran2012convex}. Here,~$w(\Omega)$ is the Gaussian width of the set~$\Omega$. In turn,~$\Omega$ is the intersection of the corresponding tangent cone of~$\A$ with the unit sphere, i.e., $\Omega:= \cone(\A - x^\n/\gauge_\A(x^\n)) \cap \mathbb{S}^{d-1}$. For completeness, this result is reviewed in Appendix B.8. Through the celebrated Dudley's inequality, the Gaussian width of the above set $\Omega$ relates to its entropy number with respect to the Euclidean metric~\cite{chandrasekaran2012convex}. This might be contrasted with~\eqref{eq:ddprimeSmall} which involves the entropy number of the set $\Slice_p(\A)$ with respect to the metric $\dist_p$.}

\section{Stylized Applications of  the \texorpdfstring{Gauge$_p$}{Gp} Function Theory}\label{sec:action}
In Section~\ref{sec:newMachine}, we {studied} the new machine~\eqref{eq:mainNonCvx}. Without being exhaustive, this section {applies} the new learning machine to two representative  problems {to showcase its} potential. 

\subsection{Manifold-Like Models}\label{sec:manifolModels}
Despite the {importance of manifold models} in signal processing and machine learning~\cite{eftekhari2015new,peyre2009manifold,iwen2018recovery}, the classical gauge function theory \rev{might} fail to learn manifold models, {as highlighted in Section~\ref{sec:old} with a toy  example.} 
{In this section, we consider a slightly more general family of models and show that the new machine~\eqref{eq:mainNonCvx} succeeds in learning them} \rev{from limited observations.}

Suppose that the alphabet~$\A$ is an arbitrary subset of $\R^d$, and consider the $1$-sparse {setup}
\begin{align}
    y := \L(A^\n) \rev{\in \R^m}, \qquad A^\n  \in \A,\label{eq:manifoldModelApp}
    \tag{manifold-like}
\end{align}
where $\L:\R^d\rightarrow \R^m$ is a linear operator. For simplicity, we have not accounted for  measurement noise {in the {setup} above}. \rev{The equation~\eqref{eq:manifoldModelApp} is common in learning with nonlinear constraints, e.g., when using a generative adversarial network as the prior~\cite{gomez2019fast,bora2017compressed}. In particular, when~$\A$ is an embedded submanifold of~$\R^d$, then~\eqref{eq:manifoldModelApp} reduces to the well-known manifold model~\cite{eftekhari2015new}.

Since the alphabet~$\A$ in~\eqref{eq:manifoldModelApp} can be arbitrary, the $1$-sparsity of the~\eqref{eq:manifoldModelApp} {setup} does not reduce its generality. Indeed, consider another alphabet~$\A'\subset \R^d$. For an integer $r$, let us set $\A := \hull_r(\A')$, according to Definition~\ref{defn:hall}. Then, every $1$-sparse model in the alphabet~$\A$  corresponds to an $r$-sparse model of the alphabet~$\A'$. On the other hand, as we will see shortly, the $1$-sparsity of~\eqref{eq:manifoldModelApp} will simplify the presentation of the main result in this section.}

To recover the atom~$A^\n$ in~\eqref{eq:manifoldModelApp}, we may {implement} the machine~\eqref{eq:mainNonCvx} for any $p\ge 1$. In particular, \rev{the choice of $p=1$ leads us to consider the learning machine}
\begin{align}
    \min_{c,A} \,\, \| c \L(A) - y\|_2^2\,\, \text{subject to} \,\, 0\le c\le 1 \text{ and } A\in \A,
    \tag{gauge$_1$ : manifold-like}
    \label{eq:specialCaseManifold}
\end{align}
which is closely related to those numerically {studied} in~\cite[Equation 12]{eftekhari2015new} and~\cite[Equation 20]{peyre2009manifold}. 

\rev{In the remainder of this section, we will limit ourselves to the special of case of~\eqref{eq:manifoldModelApp} in which the alphabet~$\A$ is a compact embedded  submanifold of~$\R^d$~\cite{lee2018introduction}. For example, the set~$\{x: h(x)=0\}$ is a $k$-dimensional embedded submanifold of $\R^d$ if the Jacobian of $h:\R^d\rightarrow\R^k$ is full-rank everywhere~\cite{absil2009optimization}. This restriction  does not considerably reduce the generality of our results below. Indeed, if a bounded alphabet $\A'\subset \R^d$ is not an embedded submanifold, one can always replace the alphabet $\A'$ with a new alphabet $\A\subset \R^d$ such that $\A$ is a compact embedded submanifold and $\A'\subset \A$. 
For example, a (potentially conservative) choice for $\A$ is a sufficiently large (closed) Euclidean ball that contains~$\A'$.

The main result of this section is a  corollary of Theorem~\ref{thm:constructOpt} for the special case of~$r=p=1$, presented below. This corollary predicts that the machine~\eqref{eq:specialCaseManifold} successfully recovers the true atom~$A^\n$ from {the vector of observations}~$y=\L(A^\n)\in \R^m$, provided that~$\L$ is a generic linear operator and $m$ is sufficiently large. In the corollary below, for tidiness, we assume that $A^\n$ has unit norm. {We make this assumption} without any loss of generality {because} {it can always be enforced by scaling} the alphabet~$\A$. Again for tidiness, we also introduce a new parameter: In the corollary below, instead of directly using the critical angle~$\theta_{A^\n,1}(\A)$ in Definition~\ref{defn:critAngle}, we will make  use of another angle, which is defined as
\begin{align}
\theta'_{A^\n,1}(\A) := \inf \l\{ \angle [A-A^\n,A^\n]: {A\in \A - \{A^\n\}} \r\} {\in [0,\pi]}. 
\label{eq:gammaDefn}
\end{align}
{As shown in the proof of the corollary, the two angles $\theta_{A^\n,1}(\A)$ and $\theta_{A^\n,1}'(\A)$ are} closely related:
\begin{equation}
\theta_{A^\n,1}(\A) \le \frac{\pi}{2} - \frac{\theta'_{A^\n,1}(\A)}{2}.
\end{equation}
{In the corollary below, we will also make use of reach of $\A$, a geometric attribute of the manifold~$\A$ that is reviewed below~\cite{federer1959curvature}. This elementary property, rooted in geometric measure theory, has become somewhat popular in the analysis of manifold models for signal processing~\cite{baraniuk2009random,davenport2007smashed,davenport2010joint}.}

\begin{defn}[{\sc Reach}]\label{defn:reach} 
Suppose that $\A$ is a compact embedded submanifold of $\R^d$. The reach of~$\A$, denoted by $\reach(\A)$, is the largest number $r$ that satisfies the following: The open normal bundle of $\A$ of radius $r$ is embedded in $\R^d$ for all $r<\reach(\A)$. Recall that the normal bundle of radius $r$ is the set of all normal vectors to the manifold $\A$ of length at most $r$~\cite{lee2018introduction}.  
\end{defn}

It is not difficult to verify that the reach of a circle is its radius, see Figure~\ref{fig:coneAngle}. As another example, consider the so-called moment curve $t\rightarrow [1,\cdots,e^{\mathrm{i}2\pi(d-1)t}]$, which can be embedded in $\R^{2d}$. The reach of the moment curve is known to be proportional to $\sqrt{d}$, see~\cite[Section 2.2.2]{eftekhari2015new}. Let us now state the main result of this section.

\begin{corr}[{\sc Manifold-like models}]\label{cor:manifold}
Suppose that Assumptions~\ref{a:alph}\ref{assumption:containsOrigin} and~\ref{assumption:bounded} on the alphabet $\A$ are met. Consider the {setup}~\eqref{eq:manifoldModelApp} and assume without loss of generality that $\|A^\n\|_2=1$. For an integer $k$, suppose also that $\A$ is a compact $k$-dimensional  embedded submanifold of $\R^d$. Let $\vol_k(\A)$ and $\reach(\A)>0$ denote the $k$-dimensional volume of $\A$ and its reach, respectively. We also make the mild technical assumption that $\vol_k(\A)\cdot \reach(\A)^k \ge (10.5/\sqrt{k})^k$. Lastly, assume that $\theta'_{A^\n,1}(\A)>0$, see~\eqref{eq:gammaDefn}. For $\d\in[0,1)$, suppose that the linear operator~$\L:\R^d\rightarrow\R^m$ in~\eqref{eq:manifoldModelApp} satisfies the probabilistic $\d$-RIP, see Definition~\ref{def:isometry}. Then the machine~\eqref{eq:specialCaseManifold} returns~$A^\n$ if
    \begin{align}
        m \ge 
        \frac{Ck
        }{\l[\sin(\theta_{A^\n,1}'(\A))\r]^2}
        \log\l[\frac{\max(\|\L\|_\op^2,1)}{\sin
        (\theta_{A^\n,1}'(\A)) } \l(\vol_k(\A)\r)^{\frac{1}{k}}\reach(\A) \r],
        \label{eq:manifoldFinal}
    \end{align}
except with a probability of at most $\exp\big(-C' [\sin(\theta_{A^\n,1}')]^2 m\big)$ where $C,C'$ are universal constants. That is, with high probability, the pair $(1,A^\n)$ is the unique solution of the optimization problem~\eqref{eq:specialCaseManifold}, provided that $m$ is sufficiently large.
\end{corr}

The proof of Corollary~\ref{cor:manifold} estimates the entropy number on the right-hand side of~\eqref{eq:ddprimeSmall} in Theorem~\ref{thm:constructOpt}. Note that the number $m$ of observations depends logarithmically on the volume and reach of the manifold~$\A$. Moreover, if we ignore the logarithmic term, the number of observations in~\eqref{eq:manifoldFinal} is linear in the dimension $k$ of the manifold~$\A$. Corollary~\ref{cor:manifold}, which specializes Theorem~\ref{thm:constructOpt} to manifolds, is in the same vein as~\cite[Theorem~4]{eftekhari2015new}. To apply Corollary~\ref{cor:manifold}, one only needs to  have access to (estimates of) four geometric attributes of the compact manifold~$\A$, namely, its dimension, volume, reach, and critical angle. Except for perhaps the critical angle, these might be considered widely-studied attributes of a manifold~$\A$. The final remark of this section revisits the toy example in Figure~\ref{fig:manifolds} and computes its critical angle. 

\begin{remark}[{\sc Critical angle}] 
{Consider} an alphabet $\A\subset \R^d$ and an atom~$A^\n\in \A$. Recall from~\eqref{eq:gammaDefn} that, {for tidiness in this section}, {we earlier replaced the critical angle in Definition~\ref{defn:critAngle} with the angle}~$\theta'_{A^\n,1}(\A)$. {This new} angle~$\theta'_{A^\n,1}(\A)$ {evidently} plays a key role in Corollary~\ref{cor:manifold}. More specifically, to recover the atom $A^\n$, Corollary~\ref{cor:manifold} {requires} that~$\theta'_{A^\n,1}(\A) \ne 0$. \rev{That is,} \rev{Corollary~\ref{cor:manifold}} requires that~$\angle [A-A^\n,A^\n] \ne 0$ for every atom~$A\in \A$, see~\eqref{eq:gammaDefn}. 
In Example~\ref{ex:1}, recall that~$\A$ was a~\eqref{eq:spiral} and \rev{the atom} $A^\n$ was specified in~\eqref{eq:exModelManifold}, \rev{represented by} the red dot in Figure~\ref{fig:manifolds}.  For that example, \rev{one can} verify that 
$\theta'_{A^\n,1}(\A)\approx 51^\circ$. In contrast, unless we assume that $\L:\R^d\rightarrow\R^m$ is injective, it is not difficult to {carefully construct an alphabet~$\A$ for which} the convex machine~\eqref{eq:mainCvx} {would fail} to recover the atom~$A^\n$ and~$\angle [A_i-A^\n,A^\n]\le \pi/2$ for some atoms~$\{A_i\}_i\subset \A$.
\end{remark}}

\subsection{Sparse Principal Component Analysis}\label{sec:sparsePCA}
With a negative toy example, we saw in Section~\ref{sec:old} that the classical gauge function theory might fail for the sparse PCA problem. {This happens because the} gauge function might fail to promote sparsity. That is, the minimal decomposition {that} achieves the gauge function value might not {be sparse} in the {corresponding} learning alphabet,  {see~\eqref{eq:spca1}.} In contrast, the generalized  theory, developed in Section~\ref{sec:new}, immediately rectifies this {issue}. {More specifically}, by design, any minimal decomposition that achieves the gauge$_p$ function value {is} always $p$-sparse, see Proposition~\ref{prop:gauge2EquivDefn}\ref{prop:gauge_p:sparse}. \rev{This observation is precisely the improvement offered by the generalized gauge function theory.}

In the remainder of this section, we will show that  the {new}  learning machine {asymptotically approaches} the information-theoretic performance limit of sparse PCA  for the spiked covariance model~\cite{d2005direct,amini2008high,berthet2013optimal,deshpande2014information}. We will also show that the new machine generalizes beyond the spiked covariance model. 
More specifically, for sparsity level $k$ {and dimension $d$}, consider the alphabet
\begin{align}
  \A := \{ uu^\top : \|u\|_2 = 1,\, \|u\|_0 \le k,\,  u\in \R^{d}\} \subset \R^{d\times d},
         \label{eq:spca2A}      
\end{align}
where~$\|u\|_0$ is the number of nonzero entries of the vector~$u$. {Throughout this section,} it is important not to confuse the sparsity level of {the vector} $u$ (number of its nonzero entries) with the sparsity of a statistical model (number of atoms of the alphabet~$\A$ {that are} present in {the} model). For $\theta\in [0,1)$, consider also a  Gaussian random vector of length~$d$, with zero mean and the covariance matrix~$\Sigma \in \R^{d\times d}$. {This} covariance matrix is specified as 
\begin{align}
    &  \Sigma := A^\n + \theta I_{d},  \qquad A^\n = u^\n (u^\n)^\top\in \A,
    \label{eq:defnSigma}
\end{align}
where $I_{d}\in \R^{d\times d}$ is the  identity matrix. Above,~$A^\n$ is the  ``spike'' in the spiked covariance model. Instead of the covariance matrix $\Sigma$, {however,} we have access to the sample covariance matrix 
\begin{align}
    y := \frac{1}{n} \sum_{i=1}^n z_i z_i^\top,
    \label{eq:sampleCovMat}
\end{align}
formed by {the} samples $\{z_i\}_{i=1}^n \subset \R^{d}$, drawn independently from the distribution $\normal(0,\Sigma)$. 

Given the sample covariance matrix~$y$, the objective of sparse PCA is to identify the spike in the covariance matrix $\Sigma$, i.e., our objective is  to identify the atom~$A^\n$ in~\eqref{eq:defnSigma}. \rev{In view of \eqref{eq:fullIneaxctModel},  our inexact {$1$-sparse} {setup} is defined as 
\begin{align}
    y := \mathcal{L}( A^\n) +e,\qquad  A^\n\in \A, \quad \text{where} \quad {\mathcal{L}= \mathrm{id}},\, 
    e := \frac{1}{n} \sum_{i=1}^n z_i z_i^\top - A^\n.
    \tag{spike}\label{eq:spikedModel}
\end{align}
The alphabet $\A$ is specified as in~\eqref{eq:spca2A}. Above, the operator~$\mathrm{id}$ denotes the identity operator. To recover the spike~$A^\n$, we may apply the machine~\eqref{eq:mainNonCvx} for any $p\ge 1$. In particular, the choice of $p=1$ leads us to consider the learning machine}
\begin{align}
    \min_{c,A} \,\|y - cA\|_{\mathrm{F}}^2 \,\,\text{subject to}\,\, 0\le c\le  1\,\, \text{and}\,\, A\in \A.\label{eq:sparsePCAOurs}
\end{align}
Since we are only interested in recovering the atom $A^\n$, and not its {amplitude}~$c^\n$, it suffices to consider the optimization {over} $A$ within~\eqref{eq:sparsePCAOurs}, which reads as
\begin{align}
    \max_{A\in \A}\,\, \langle y, A\rangle  = \max_{u\in \R^{d}} \l\{u^\top y u: \|u\|_2=1,\, \|u\|_0 \le k\r\},
    \qquad \text{(see \eqref{eq:spca2A})}
    \label{eq:matchedFilterSparsePCA}
        \tag{gauge$_1$ : SPCA}
\end{align}
\rev{The above optimization problem is the starting point of the well-known convex relaxation proposed by~\cite{d2005direct}. To see this connection, note that \eqref{eq:matchedFilterSparsePCA} implies that
\begin{align}
     \max_{A\in \A} \,\,\langle y,A\rangle
   & \le \max \l\{ \langle y, A\rangle : \tr(A) = 1 ,\, \|A\|_0\le k^2, \, A\in \psd^{d}_+  \r\},
\end{align}
where $\psd^{d}_+=\{A\in \R^{d\times d}: A=A^\top,\, A\succeq 0\}$ is the cone of positive semi-definite (PSD) matrices, and~$\|A\|_0$ denotes the number of nonzero entries of $A$. We can obtain a convex relaxation of the right-hand side above by replacing $\|A\|_0$ with the $\ell_1$-norm of the matrix $A$, which is $\|A\|_1:=\sum_{i,j} |A_{i,j}|$. By doing so, we obtain the convex relaxation 
\begin{equation}
    \max\l\{ \langle y, A\rangle - \lambda \|A\|_1 : \tr(A) = 1,\, A\in \psd^{d}_+ \r\},
    \label{eq:cvxRelax}
\end{equation}
with $\lambda>0$, which is precisely the  optimization problem studied in~\cite{d2005direct}. As detailed in the result below, under mild assumptions, the machine~\eqref{eq:matchedFilterSparsePCA} provably discovers the spike $A^\n$ in the spiked covariance model. The proof of the result below is  standard in the context of empirical processes and the result itself is in the same vein as \cite[Proposition 1]{amini2008high}.
\begin{prop}[{\sc Spiked covarince model}]\label{prop:infLimit}
    Consider the spiked covariance \rev{setup} in \eqref{eq:spikedModel}. The machine \eqref{eq:matchedFilterSparsePCA} asymptotically returns the spike $A^\n$ in~\eqref{eq:spikedModel}. More specifically, consider a sequence $\{k_l,d_l,n_l\}_l$ such that $\lim_{l\rightarrow \infty} n_l=\infty$. Suppose that 
    \begin{equation}
        \lim_{l\rightarrow\infty} \frac{k_l\log d_l}{n_l}=0.
    \end{equation}
    Then, in the limit of $l\rightarrow\infty$, solving the optimization problem
    $
        \max_{u\in \R^{d}} \l\{u^\top y u: \|u\|_2=1,\, \|u\|_0 \le k\r\}
    $
     returns a vector $u^\n\in \R^{d}$ such that $u^\n (u^\n)^\top = A^\n$, with a probability that approaches one. 
\end{prop}
For the sake of comparison, recall from \cite[Theorem 3]{amini2008high} that it is impossible for \emph{any} method to discover the spike $A^\n$ if 
\begin{equation*}
\frac{k_l \log\l(d_l-k_l\r)}{n_l} > \frac{1}{\theta+\theta^2},
\end{equation*}
where $\theta$ is the noise level, see~\eqref{eq:defnSigma}. Moving on, we have so far focused on the spiked covariance \rev{setup}, i.e., the covariance matrix $\Sigma$ in~\eqref{eq:defnSigma} contains only one spike. 
When $\Sigma$ in~\eqref{eq:defnSigma} contains multiple spikes, the common alternative of deflation~\cite{mackey2009deflation} might be numerically unstable. Yet another alternative to deflation is to search for a subspace with sparse basis vectors, which all together forgoes the individual sparse components in favour of identifying a sparse subspace~\cite{vu2013minimax}. However, as we will see below, the proposed learning machine naturally generalizes to multiple spikes. More specifically, instead of \eqref{eq:defnSigma}, suppose that the covariance matrix $\Sigma$ is specified as
\begin{equation}
    \Sigma:= x^\n + \theta I_{d},
    \qquad x^\n:= \sum_{i=1}^r c_i^\n A_i^\n,
    \qquad c_i^\n\ge 0,\, A_i^\n\in \A, \, i\le r,
\end{equation}
where $r$ is the number of atoms (spikes) present in $\Sigma$. In view of~\eqref{eq:fullIneaxctModel}, this time our  inexact $r$-sparse {setup} is 
\begin{equation}
    y := \L(x^\n)+e, \qquad \L= \mathrm{id},\, e:= \frac{1}{n}\sum_{i=1}^n z_iz_i^\top - x^\n.
    \label{eq:spikedModel2}
    \tag{multiple spikes}
\end{equation}}{To recover the model~$x^\n$ and/or the spikes~$\{A_i^\n\}_{i=1}^r$, we may apply the machine~\eqref{eq:mainNonCvx} for any~$p\ge r$.}
\rev{In particular, the choice of~$p\ge r$ leads us to consider the learning machine
\begin{equation}
        \min\l\{\l\| y - \sum_{i=1}^p c_i A_i\r\|_\mathrm{F}^2: \sum_{i=1}^p c_i \le \tr(x^\n), \, c_i \ge 0, A_i\in \A \r\}.
        \label{eq:spca2}
        \tag{gauge$_p$ : SPCA}
\end{equation}
Note that $\tr(x^\n)$ in~\eqref{eq:spca2} might not be known advance. However, in principle, guarantees for~\eqref{eq:spca2} can be transferred to its basis pursuit reformulation in which the objective function and constraint of~\eqref{eq:spca2} are swapped, as discussed in Section~\ref{sec:old}.  
The following result provides the sufficient conditions for~\eqref{eq:spca2} to successfully recover the spikes.

\begin{prop}[{\sc Generalized spiked covariance setup}]\label{prop:sparcePCAgen}
    Consider the generalized spiked covariance {setup} in \eqref{eq:spikedModel2}. The machine~\eqref{eq:spca2}  asymptotically returns the spikes~$\{A_i^\n\}_{i=1}^r$ in~\eqref{eq:spikedModel2}, provided that $p<\mathrm{spark}(\A)-r$.  
    More specifically, consider  the same sequence $\{k_l,d_l,n_l\}_l$ as in Proposition \ref{prop:infLimit}.
    Then, in the limit of~$l\rightarrow\infty$, solving the optimization problem~\eqref{eq:spca2}
     successfully returns $\{c_i^\n,A_i^\n\}_{i=1}^r$ with probability that approaches one, provided that $p<\mathrm{spark}(\A)-r$.
     Recall that $\mathrm{spark}(\A)$ is the smallest number of atoms in $\A\subset \R^{d\times d}$ that form a linearly dependent subset of $ \R^{d\times d}$~\cite{donoho2003optimally}.
\end{prop}

For the sparse PCA alphabet in~\eqref{eq:spca2A}, we are not aware of any estimates for $\mathrm{spark}(\A)$ and it appears to be nontrivial to obtain one. Nevertheless, Proposition~\ref{prop:sparcePCAgen} posits that the new machine~\eqref{eq:spca2} succeeds when $p$ is sufficiently small. In contrast, the corresponding convex machine, which coincides with~\eqref{eq:spca2} for $p\ge d(d+1)/2+1$, might fail. This last claim about the value of $p$ follows from Remark~\ref{rem:genGaugeRem} and the fact that dimension of $\psd_+^d$ is $d(d+1)/2$.}

\section{Computational Aspects and a Tractable Numerical Scheme}\label{sec:algorithm}

\rev{This section discusses the computational aspects of solving the optimization problem~\eqref{eq:mainNonCvx}. As discussed earlier, this problem might be  nonconvex, particularly for small values of $p$. We identify three classes of alphabets:}

\begin{enumerate}[label=(\roman*), itemsep = 0mm, topsep = 0mm, leftmargin = 5mm]
    
\item For certain alphabets, the optimization landscape of the new machine~\eqref{eq:mainNonCvx} does not have any spurious stationary points and~\eqref{eq:mainNonCvx} is amenable to a variety of {standard} optimization algorithms. A prominent example {was discussed in} Remark~\ref{rem:bm-rem}, i.e., the well-known Burer-Monteiro factorization for certain matrix- or  tensor-valued learning problems, see~\cite{chi2019nonconvex,boumal2020deterministic,eftekhari2020implicit}.  

\item For certain other alphabets, such as smooth manifolds~\cite{eftekhari2015new,shah2011iterative} or shallow neural networks~\cite{safran2018spurious}, the optimization landscape of~\eqref{eq:mainNonCvx} might in general contain spurious stationary points which could potentially trap first- or second-order optimization algorithms, such as gradient descent. Nevertheless, problem~\eqref{eq:mainNonCvx} can be reformulated as a smooth nonconvex optimization problem and {then} solved efficiently to  stationarity (rather than global optimality) with a variety {of first- or second-order algorithms}~\cite{nocedal2006numerical}. This compromise (between optimality and tractability) is common in machine learning: As an example,  empirical risk minimization is known to be NP-hard for neural networks in general and the practitioners instead seek local (rather than global) optimality~\cite[Chapter 20]{shalev2014understanding}.

\item Yet for many other alphabets, such as the one in Example~\ref{ex:1} (sparse regression), the problem~\eqref{eq:mainNonCvx} is NP-hard in general for~$p<d$~\cite{natarajan1995sparse}. Moreover, the {second approach above is not}  directly applicable.  There are, however, compelling reasons to remain optimistic for such alphabets. For example, after decades of research, modern mixed-integer optimization algorithms that directly solve the problem~\eqref{eq:l0machine} for sparse regression are now competitive with convex heuristics in speed and scalability~\cite{bertsimas2020sparse,bertsimas2016best,vreugdenhil2021principal}. {We will pursue this direction in the next section. }
\end{enumerate}

\subsection{\rev{Tractable Numerical Scheme}} \label{sec:algorithm}

\rev{In this section, we  provide a tractable numerical scheme for solving the problem~\eqref{eq:mainNonCvx} for a finite learning alphabet. If the alphabet is infinite, it is sometimes possible to discretize it and apply the algorithm in this section, e.g., in super-resolution~\cite{tang2013sparse}. The algorithm in this section builds on the recent developments in mixed-integer programming~\cite{bertsimas2020sparse,bertsimas2016best,vreugdenhil2021principal}.} To begin, when the learning alphabet $\A$ is finite, the following lemma offers an exact reformulation of the problem~\eqref{eq:mainNonCvx} as a mixed integer quadratic programming~(MIQP). MIQP is in general an NP-hard problem~\cite{papadimitriou1981complexity}, {which comes at no surprise} since the original problem of sparse recovery of the {setup}~\eqref{eq:mainModel} is also known to be hard~\cite{natarajan1995sparse}. Nonetheless, {the lemma below} allows us to deploy the rich literature of computational mathematical programming dedicated to MIQP, see~\cite{bienstock1996computational,del2017mixed}.

\begin{lem}[{\sc MIQP reformulation}]\label{lem:MIQP}
    Suppose that the alphabet $\A \subset \R^d$ is finite {and denote its size by} $|\A|<\infty$. If $M > 0$ is sufficiently large, the machine~\eqref{eq:mainNonCvx} is equivalent to the MIQP optimization {problem}
    \begin{align}\label{gauge:MIQP}
     \min\limits_{c, s} \left\{\Big \|\sum_{i=1}^{|\A|} c_i \L(A_i) - y\Big \|^2_2 ~:~ \sum_{i=1}^{|\A|} c_i \le \gauge_{\A,p}(x^\n), \,\,  |c_i| \le M s_i,
    \,\, \sum_{i=1}^{|\A|} s_i = p ,\,\, c_i \ge 0, \,\, s_i \in \{0,1\}\right\}. 
    \end{align}
\end{lem}

The MIQP reformulation in Lemma~\ref{lem:MIQP} {leverages} the so-called ``big-$M$'' technique {in which} $M$ is only required to be a sufficiently large constant. It is well known that the choice of $M$ 
has a significant impact on the performance of cutting plane algorithms for convex integer optimization~\cite{bigM}. 

To address this issue, inspired by the recent work of~\cite{bertsimas2020sparse}, we {next provide} a dual reformulation of the optimization {problem} in Lemma~\ref{lem:MIQP}. \rev{This dual reformulation supplies} a good starting point (warm start) for branch-and-bound algorithms. \rev{This reformulation is a slight generalization of the one proposed in~\cite{bertsimas2020sparse} which also handles the linear constraints such as $c_i \ge 0$ and $\sum_{i=1}^{|\A|} c_i \le \gauge_{\A,p}(x^\n)$. Below, we will use the notation $[n]:=\{1,\cdots,n\}$ for an integer $n$.}

\begin{prop}[{\sc Tractable algorithm}]\label{prop:alg}
    Let us define the matrices 
    \begin{align*}
        \Dic = \begin{bmatrix}A_1 & A_2 & \dots & A_{|\A|}\end{bmatrix} \in \mathbb{R}^{d\times |\A|}, \quad C = \begin{bmatrix}-I_{|\A|} \vspace{1mm}\\ \ones_{|\A|}^\top \end{bmatrix}, \quad g = \begin{bmatrix}0_{|\A|} \vspace{1mm}\\ \gauge_{\A,p}(x^\n)\end{bmatrix},
    \end{align*}
    where~$|\A|$ is the size of the finite set~$\A$. {Above},~$I_{|\A|}\in \R^{|\A|\times |\A|}$ is the identity matrix,~$\ones_{|\A|}\in \mathbb{R}^{|\A|}$ is a vector of all ones, and~$0_{|\A|}\in \mathbb{R}^{|\A|}$ is a vector of zeros. The optimal value of the \rev{optimization problem}~\eqref{gauge:MIQP} coincides with the {optimal value of the} minimax {problem}
    \begin{align}\label{gauge:minimax}
        \min_{{\footnotesize \begin{array}{c} \Set \subset [|\A|]\\ |\Set| \le p\end{array}}} \max_{\mu \ge 0, \lambda} 
        -\frac{1}{2} \|\lambda\|_2^2 - \langle g, \mu\rangle - \frac{\gamma}{2} \sum_{i \in \Set} \big((\L(\Dic))^\top \lambda - C^\top \mu\big)_i^2,
    \end{align}
    where~$\gamma$ is a sufficiently large constant, $\L(\A)=[\L(A_1)\,\L(A_2)\,\cdots\,\L(A_{|\A|})]\in \mathbb{R}^{m\times |\A|}$, and the \rev{notation}~$(u)_i$ returns the~$i^{\mathrm{th}}$ coordinate of {the} vector~$u$. 
    \rev{Moreover, an optimal vector $s\in \{0,1\}^{|\A|}$ for the problem~\eqref{gauge:MIQP} corresponds to an optimal set $S \subset [|\A|]$ in~\eqref{gauge:minimax}. Lastly, any fixed-point of the algorithm below is a 
    solution of~\eqref{gauge:minimax}.}
    \begin{subequations}\label{gauge:alg} 
    \begin{align}
        \label{gauge:alg:concave}
        \begin{bmatrix} \lambda_{k+1} \\ \mu_{k+1} \end{bmatrix} & = \arg\max_{\mu \ge 0, \,\lambda} \, -\frac{1}{2} \|\lambda\|_2^2 - \langle g, \mu\rangle - \frac{\gamma}{2} \sum_{i \in \Set_k} \big((\L(\Dic))^\top \lambda - C^\top \mu\big)_i^2, \\
        \label{gauge:alg:Set}
        \Set_{k+1} &= \arg\max_{\substack{\Set \subset [|\A|]\\  |\Set| \le p}}\sum_{i \in \Set} \big((\L(\Dic))^\top \lambda_k - C^\top \mu_k\big)_i^2.
    \end{align}
    \end{subequations}
\end{prop}

We emphasize that  the objective function in \eqref{gauge:alg:concave} is concave and quadratic jointly in the variables~$(\lambda,\,\mu)$, {and} the number of the summands {in~\eqref{gauge:alg:concave}} is the sparsity level~$|\Set| = p$. Moreover, and more importantly, the set-valued optimization~\eqref{gauge:alg:Set} admits an (almost) analytic solution as it suffices to select only the first~$p$ coordinates~$i \le |\A|$ {for which} $\big((\mathcal{L}(A))^\top \lambda_k - C^\top \mu_k\big)_i^2$ is maximized. The complexity of this step reduces to a sorting algorithm. Therefore, the algorithm~\eqref{gauge:alg} is indeed computationally a highly tractable implementation of the machine~\eqref{eq:mainNonCvx}, which may merit a more comprehensive numerical investigation in the future. 

If the iteration of~\eqref{gauge:alg} converges, we solve the the learning machine~\eqref{eq:mainNonCvx}. However, it often happens that the algorithm~\eqref{gauge:alg} oscillates between several discrete solutions~$\Set$. We note that these candidates can then be chosen as ``warm start'' for the MIQP problem~\eqref{gauge:MIQP} \`{a} la~\cite{bertsimas2020sparse}.

\subsection{Numerical Examples}\label{sec:numerics}
\rev{We now investigate two numerical examples that support the theoretical findings of this work. 

\subsubsection*{Example 1: sparse PCA}
Our first numerical example is a simplified version of sparse PCA in Section~\ref{sec:sparsePCA}. More specifically, here the alphabet $\A$ of size 35 is sub-sampled from the set
\begin{align*}
    \left\{ uu^\top : \|u\|_2 =1,\, \|u\|_0 \le 2,\, u\in \R^4  \right\} \subset \R^{4\times 4},
\end{align*}
and the model $x^\n$ in this example is a  combination of three atoms from~$\A$, i.e., $x^\n = \frac{1}{3} A_1^\n+\frac{1}{3} A_2^\n + \frac{1}{3}A_3^\n$. The setup is $y = x^\n+e$, where $e$ is the white Gaussian noise with the standard deviation~$\sigma_e\in \{0.01, 0.05, 0.1\}$. For~$p\in \{1,2,3,16\}$ and various values of~$\psi$, we numerically solve the program
\begin{align}
\label{gauge:MIQP3}
     \min\limits_{c, s} \left\{\Big \|\sum_{i=1}^{|\A|} c_i A_i - y\Big \|^2_2 ~:~ \sum_{i=1}^{|\A|} c_i \le \psi, \,\,  |c_i| \le M s_i,
    \,\, \sum_{i=1}^{|\A|} s_i = p ,\,\, c_i \ge 0, \,\, s_i \in \{0,1\}\right\}.
\end{align}
Note that the above optimization problem is an instantiation of the program~\eqref{gauge:MIQP} in which the observation operator is the identity matrix ($\mathcal{L}=\mathrm{id}$) and the the ground-truth gauge value $\gauge_{\A,p}(x^\n)$ is approximated by different values $\psi$. The vector of coefficients that minimize the above problem is denoted by $\widehat{c}_p(\psi)$. Also note that for $p=16$, the problem~\eqref{gauge:MIQP3} coincides with the convex machine~\eqref{eq:mainCvx}, see Remark~\ref{rem:genGaugeRem}.

\begin{figure}[t!]
    \centering
    \subfloat[$\sigma_e = 0.01$]{\includegraphics[scale=0.37, clip]{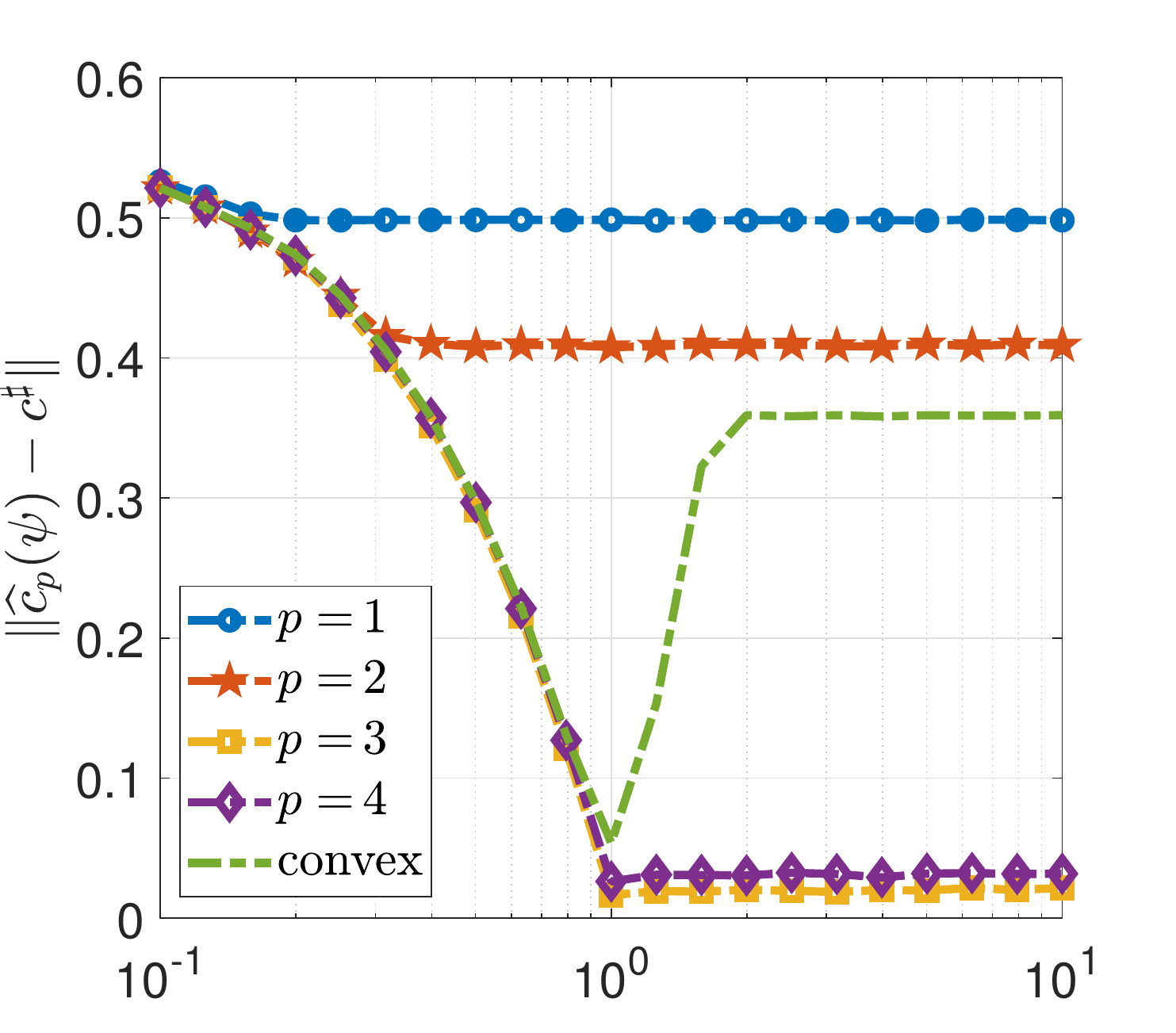}\label{fig:reso_log21}}
    \subfloat[$\sigma_e = 0.05$]{\includegraphics[scale=0.37, clip]{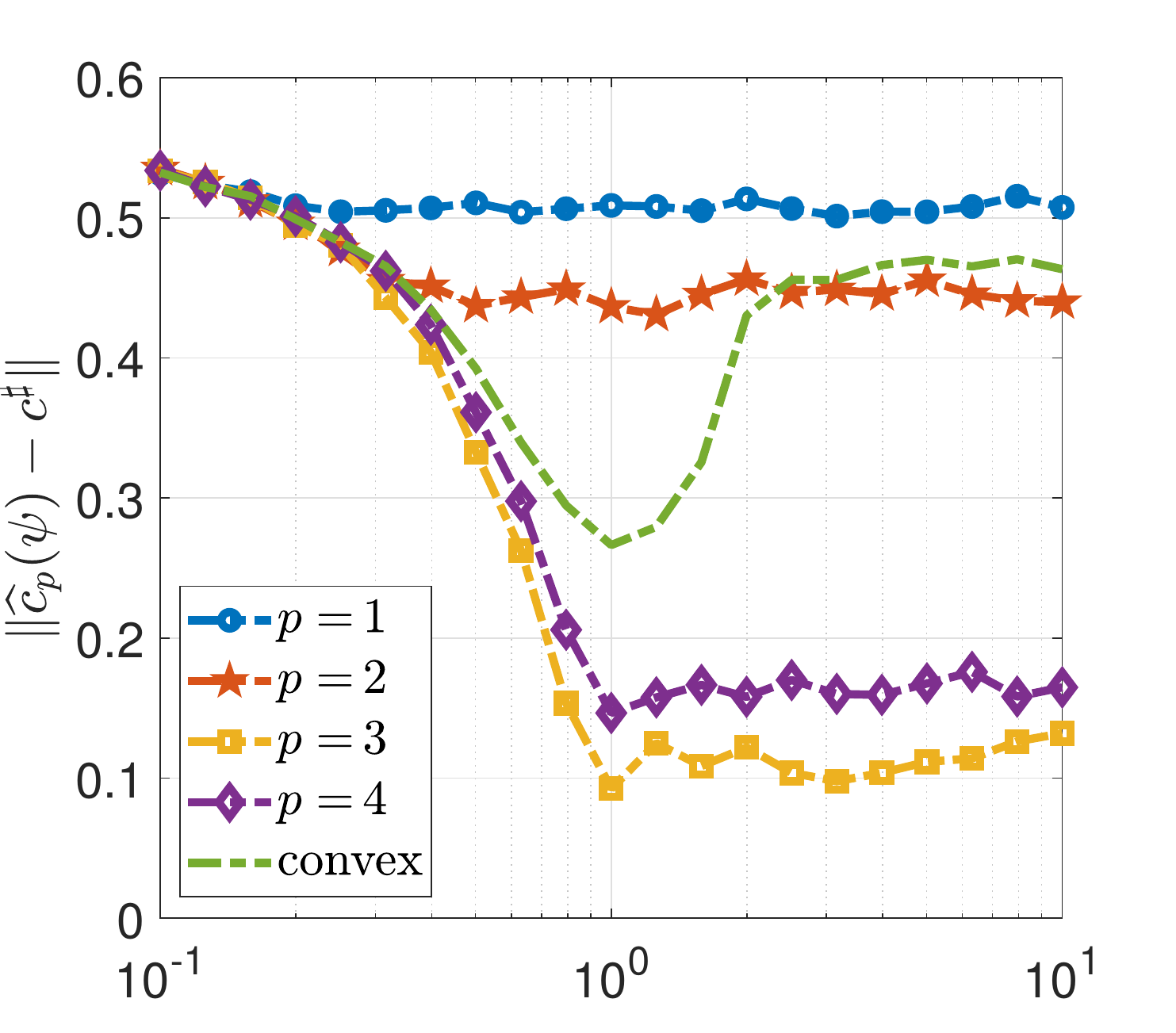}\label{fig:reso_log22}}
    \subfloat[$\sigma_e = 0.1$]{\includegraphics[scale=0.37, clip]{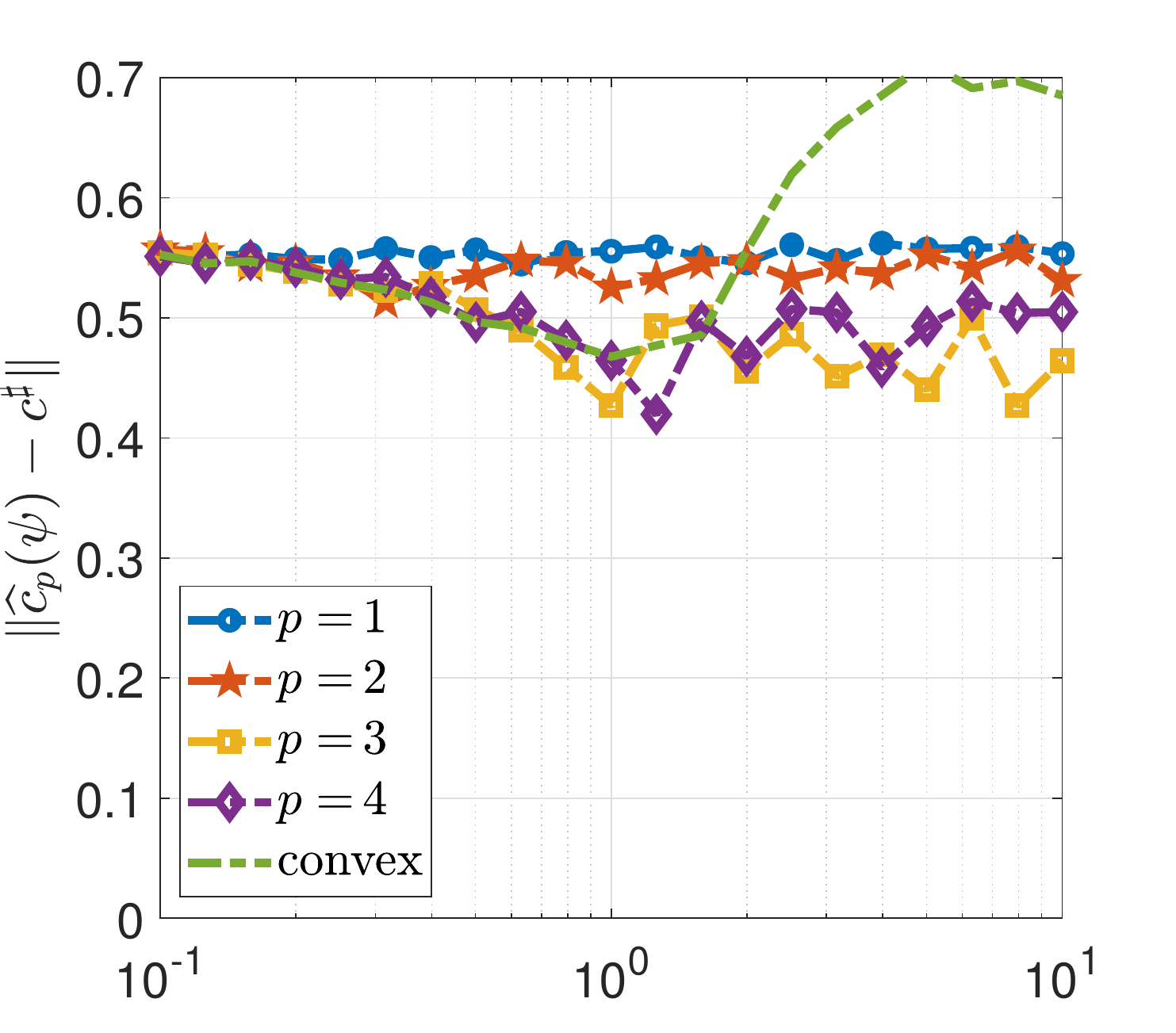}\label{fig:reso_log23}}
    \caption{\rev{Sparse PCA with the sparsity level $\|c^\n\|_0 = 3$: performance of the nonconvex machine~\eqref{eq:mainNonCvx} for $p \in \{1,2,3,4\}$ versus the convex counterpart~\eqref{eq:mainCvx} (or equivalently~\eqref{eq:mainNonCvx} for $p = |\A| = 16$)}}
    \label{fig:sparsePCAnumerics}
\end{figure}

The numerical results are reported in Figure~\ref{fig:sparsePCAnumerics} for different noise levels. Each plot shows the average error across $100$ independent experiments with fresh realizations of the noise vector. 
For~$p\in\{1,2,3\}$, we solved~\eqref{gauge:MIQP3} using the solver MOSEK with the interface of YALMIP~\cite{Lofberg2004}. The sharp transition in the plots can be explained by the fact that the true model~$x^\n$ is not feasible in \eqref{gauge:MIQP3} for small values of~$\psi$. Moreover, the poor performance for $p=1$ is explained by the fact that~$x^\n$, with the sparsity level of two, is never feasible for problem~\eqref{gauge:MIQP3} with~$p=1$. However, for $p\in\{2,3\}$, the proposed machine~\eqref{gauge:MIQP3} considerably outperforms the convex machine~\eqref{eq:mainCvx}.

} 

\subsubsection*{Example 2: Super-resolution}

\rev{Our second numerical example showcases the failure of the classical theory in the context of super-resolution below the diffraction limit. We take the alphabet to be
    $\A:=\{A_\theta\}_{\theta\in \Theta}$,  where $A_\theta(\cdot) = \exp(-( \cdot -\theta)^2/0.35^2)$ and $\Theta=\{\theta_i\}_{i=1}^{20}$ is the uniform grid over the interval~$ [0,1]$. 
    In words, our  alphabet is comprised of twenty (scaled) Gaussian waves centered on a uniform grid over $[0,1]$.}
In this example, we consider the two-sparse model
$x^\n :=  A_{\theta_{10}}- A_{\theta_{11}}$. This model is shown  in blue in Figure~\ref{fig:superRes}. In words, the blue curve in Figure~\ref{fig:superRes} represents the superposition of two (scaled) Gaussian waves and the two red bars show the centers and amplitudes of these two waves. The red bars are scaled to fit in the figure.
\\
\begin{wrapfigure}[11]{r}{0.45\textwidth}
\centering
\includegraphics[width=.3\textwidth]{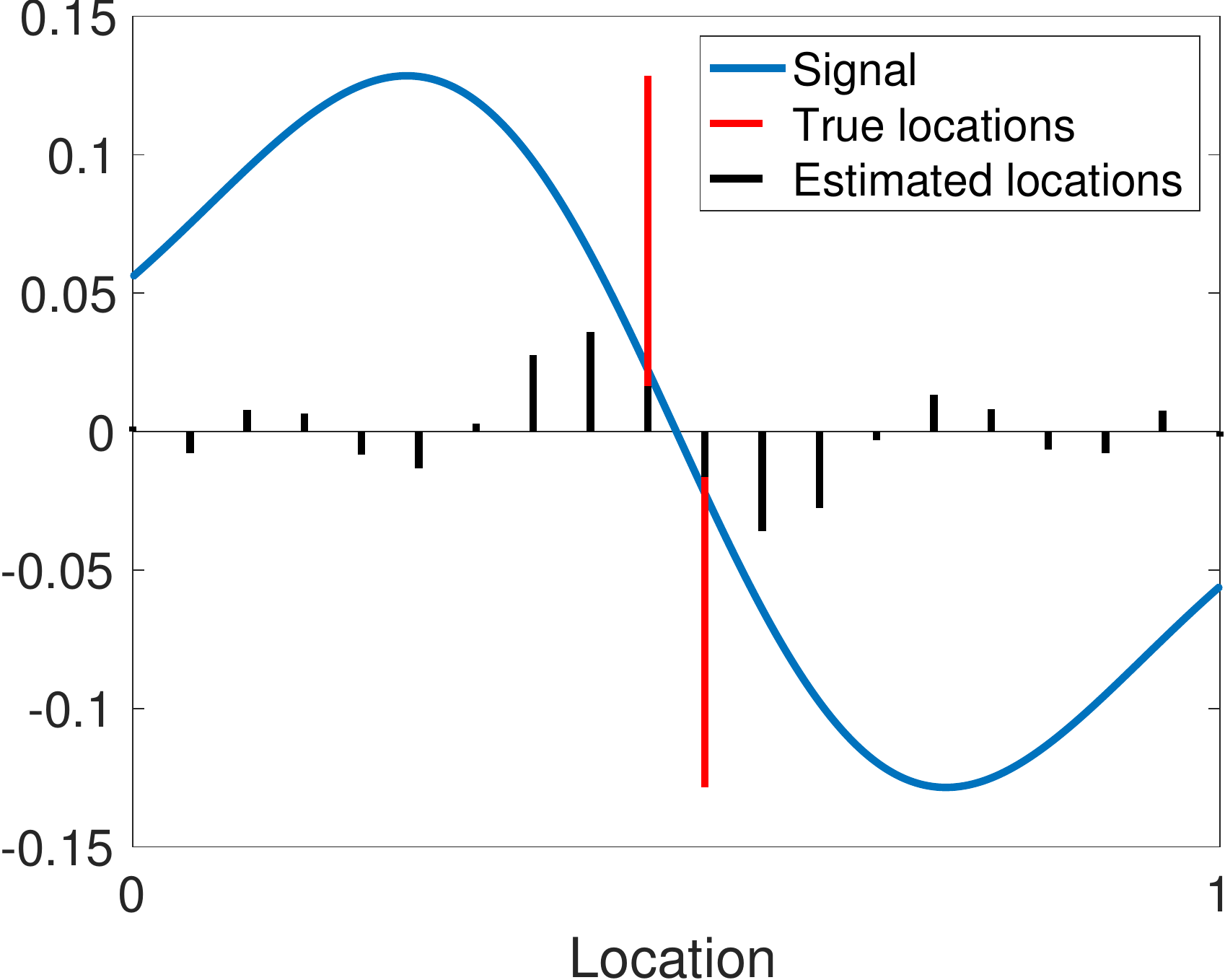}
\caption{\footnotesize{Failure of the classical gauge function theory in  super-resolution, see Section~\ref{sec:numerics}.}}
\label{fig:superRes}
\end{wrapfigure}
\rev{The (exact) values of the blue curve are then observed at~$m=40$ random locations on the interval~$[0,1]$ and then stored in a vector~$y\in \R^{40}$. 
That is, we set $e=0$ and the corresponding linear operator~$\L$ in~\eqref{eq:fullIneaxctModel} evaluates and stores the values of its input function at~$40$ random locations on the interval~$[0,1]$.}

We then estimated the true centers and amplitudes (in red) by solving the convex quadratic problem~\eqref{eq:mainCvx} with YALMIP in MATLAB~\cite{Lofberg2004}. The estimated centers and amplitudes are shown \rev{with the} black bars in Figure~\ref{fig:superRes}. The black bars are also scaled to fit in the figure. \rev{The resounding failure of the convex machine in Figure~\ref{fig:superRes} in learning the location of the red bars is an example that visualizes the difficulty of super-resolution below the diffraction limit~\cite{eftekhari2019sparse}. 
Motivated by this observation, we now apply the framework developed in this paper.
We let the noise $e$ in~\eqref{eq:fullIneaxctModel} be a vector of zero-mean and independent Gaussian random variables with standard deviation~$\sigma_e\in\{10^{-3},10^{-2}, 10^{-1}\}$.} 
For various values of~$\psi$ and integer~$p\in \{1,2,3,20\}$, we then numerically solve the problem
\begin{align}\label{gauge:MIQP2}
     \min\limits_{c, s} \left\{\Big \|\sum_{i=1}^{|\A|} c_i \L(A_i) - y\Big \|^2_2 ~:~ \sum_{i=1}^{|\A|} c_i \le \psi, \,\,  |c_i| \le M s_i,
    \,\, \sum_{i=1}^{|\A|} s_i = p ,\,\, c_i \ge 0, \,\, s_i \in \{0,1\}\right\},
\end{align}
and collect the optimal coefficients in the vector $\widehat{c}$. \rev{The above problem should be compared with~\eqref{gauge:MIQP}, which was an equivalent reformulation of the proposed learning machine~\eqref{eq:mainNonCvx}.} 

Each plot {in Figure~\eqref{fig:numerics}} is obtained by averaging the recovery errors over~$200$ independent experiments with {fresh} realizations of the noise vector. For~$p\in\{1,2,3\}$, we {solved~\eqref{gauge:MIQP2} using} MOSEK with the interface of YALMIP~\cite{Lofberg2004}.  For~$p=|\A|=20$, {however,} the problem~\eqref{gauge:MIQP2} is  a convex program. Indeed, for the choice of~$p=20$, the problem~\eqref{gauge:MIQP2} coincides with the convex machine~\eqref{eq:mainCvx}, {see Remark~\ref{rem:genGaugeRem}.} {As in Figure~\ref{fig:sparsePCAnumerics},} the sharp transitions in Figure~\ref{fig:numerics} {can be} explained by the fact that the true model~$x^\n$ is not feasible {in \eqref{gauge:MIQP2}} for small values of~$\psi$. Similarly, the poor performance for $p=1$ is explained by the fact that~$x^\n$, with the sparsity level of two, is never feasible for problem~\eqref{gauge:MIQP2} with~$p=1$. Lastly, for $p\in\{2,3\}$, the proposed machine~\eqref{gauge:MIQP2} considerably outperforms the convex machine~\eqref{eq:mainCvx}. 
\begin{figure}[t!]
    \centering
    \subfloat[$\sigma_e = 0.001$]{\includegraphics[scale=0.36, clip]{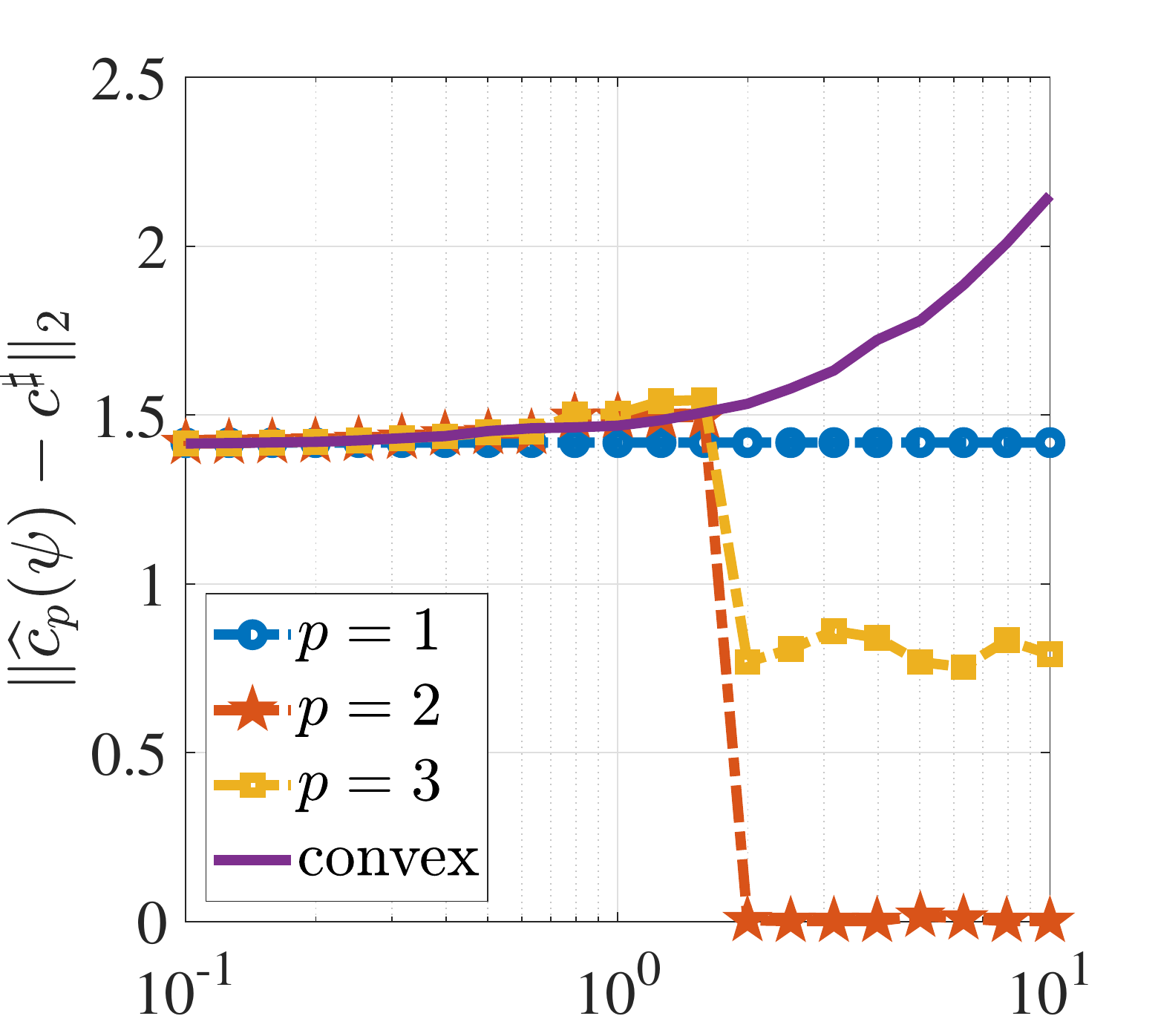}\label{fig:reso_log21}}
    \subfloat[$\sigma_e = 0.01$]{\includegraphics[scale=0.36, clip]{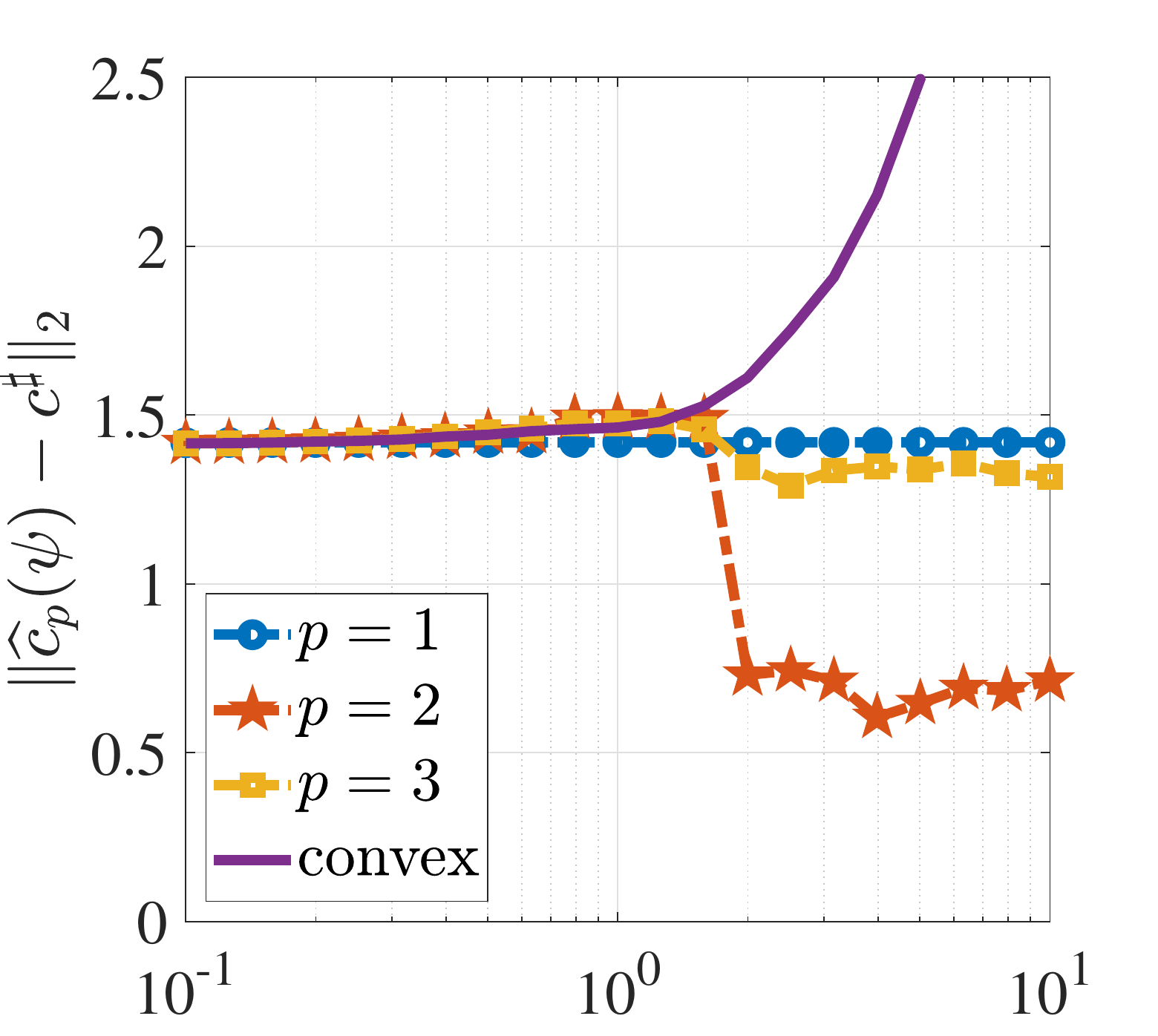}\label{fig:reso_log22}}
    \subfloat[$\sigma_e = 0.1$]{\includegraphics[scale=0.36, clip]{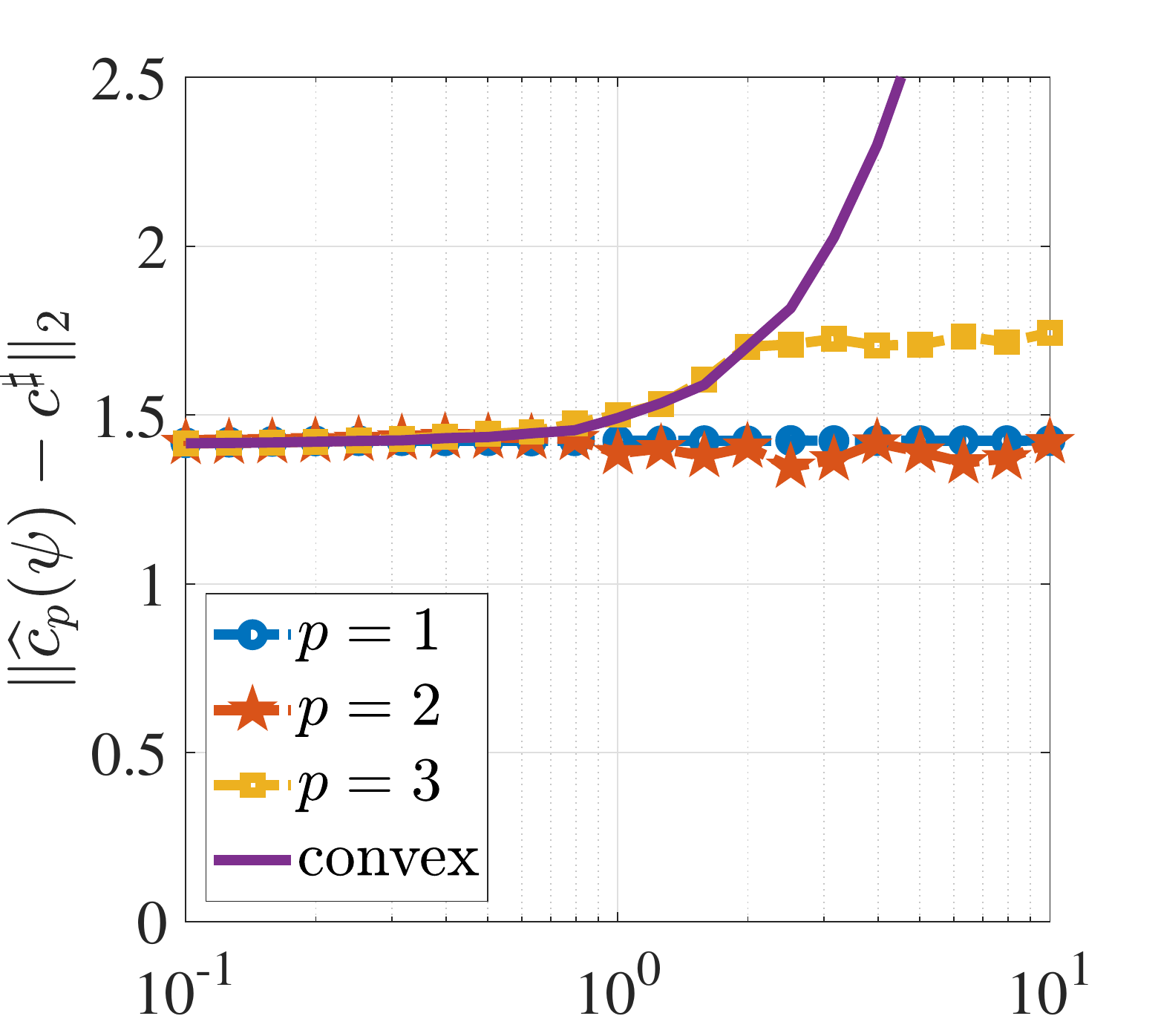}\label{fig:reso_log23}}
    \caption{\rev{Super-resolution with the sparsity level $\|c^\n\|_0 = 2$: performance of the nonconvex machine~\eqref{eq:mainNonCvx} for $p \in \{1,2,3\}$ versus the convex counterpart~\eqref{eq:mainCvx} (or equivalently~\eqref{eq:mainNonCvx} for $p = |\A| = 20$)}}
    \label{fig:numerics}
\end{figure}

\section*{Acknowledgements} 
The authors are grateful to Gongguo Tang, Michael Wakin and Konstantinos Zygalakis for helpful discussions and their valuable feedback. Peyman Mohajerin Esfahani acknowledges the support of the ERC
grant TRUST-949796.



\appendix

\section{Technical Details of Section~\ref{sec:old}}

\subsection{Group Sparsity}\label{sec:failure}

This section presents a third toy example for which the classical gauge function theory fails. {Here, the objective is to decompose the model $x^\n$ into a small number of vectors with known supports~\cite{huang2010benefit,jacob2009group}. To be concrete,} for a {factor} $C>0$ and a collection of index sets $\Omega\subset 2^{[d]}$,  the model considered in group sparsity~\cite{bach2012optimization} is 
    \begin{align}
        & x^\n \in \cone(\slice^\n), \qquad \slice^\n \in \Slice_r(\A), \nonumber\\
        & \A := \{ u: \|u\|_2 =1, \, \|u\|_\infty \le C ,\,  \text{supp}(u) \in \Omega \} \subset \R^d,
        \label{eq:groupAlphabet}
    \end{align}
    where the above bound on $\ell_\infty$-norm ensures that the atoms are diffuse on their supports. Recall that the set $\mathrm{supp}(u) \subset [d]$ denotes the support of $u$, i.e.,  the index set over which $u$ is nonzero.

\begin{example}[{\sc Group sparsity}]\label{ex:groupSparse}

For the last failed application of the gauge function theory in this section, let  us revisit group sparsity, introduced earlier in this section. As an example of the model~\eqref{eq:groupAlphabet} with $d=3$, consider the collection of index sets  
\begin{align}
    \Omega := \l\{ \{1\},\{2\},\{3\},\{1,2\}, \{2,3\} \r\} \subset 2^{[3]},
    \label{eq:exampleGroupSparse}
\end{align}
and the alphabet
\begin{align}
    \A := \l\{ u : \|u\|_2 =1,\, \|u\|_\infty \le \|u\|_0^{-\frac{1}{3}},\, \supp(u) \in \Omega  \r\},
    \label{eq:groupSparseAlphEx}
\end{align}
where the bound on $\ell_\infty$-norm above ensures that the atoms are diffuse on their support. 
With this alphabet, consider the model 
\begin{align}
    x^\n := \frac{A_1^\n}{2} +\frac{A_2^\n}{2} = \l[
    \begin{array}{ccc}
        \frac{1}{2\sqrt{2}} & \frac{1}{2\sqrt{2}}-\frac{\sqrt{7}}{8}  
         & \frac{3}{8}
    \end{array}
    \r]^\top
    ,
    \label{eq:counterExGroup}
\end{align}
where the atoms $\{A_1^\n,A_2^\n\}\subset \A$ are specified as
\begin{align}
    A_1^\n   := \l[ 
    \begin{array}{lll}
        \frac{1}{\sqrt{2}} & \frac{1}{\sqrt{2}} & 0
    \end{array}
    \r]^\top, \qquad
    A_2^\n := \l[ 
    \begin{array}{lll}
        0 & -\frac{\sqrt{7}}{4} & \frac{3}{4}
    \end{array}
    \r]^\top.
    \label{eq:atomsGroupSparsityEx}
\end{align}
Evidently, the model~$x^\n$ in~\eqref{eq:counterExGroup} has the alternative decomposition
\begin{align}
    x^\n & = \frac{A_1}{2\sqrt{2}} + \l(\frac{1}{2\sqrt{2}} - \frac{\sqrt{7}}{8}\r) A_2+
     \frac{3}{8}  A_3,
    \label{eq:counterExGroup2}
\end{align}
where $\{A_i\}_{i=1}^3 \subset \A$ are the {three} canonical vectors in $\R^3$. By comparing the two alternative representations of $x^\n$ in~(\ref{eq:counterExGroup}) and~(\ref{eq:counterExGroup2}), we find that 
\begin{align}
    \gauge_\A(x^\n) \le \min\l(\frac{1}{2}+\frac{1}{2}, \frac{1}{2\sqrt{2}}+ \l|\frac{1}{2\sqrt{2}}- \frac{\sqrt{7}}{8}\r| + \frac{3}{8}\r)=0.7514<1, 
     \qquad \text{(see \eqref{eq:gaugeFcn1})}
    \label{eq:exGroupSparseGauge}
\end{align}
and thus the $2$-sparse decomposition in~\eqref{eq:counterExGroup} is not minimal. In fact, it is not difficult to verify that the machine~\eqref{eq:cvxFacNoiseFre} fails to find any $2$-sparse decomposition for $x^\n$.
\end{example}

\subsection{Proof of Lemma \ref{lem:duality}}

Let $\wh{x}$ be a minimizer of problem~\eqref{eq:mainCvx}. 
Suppose that $\gauge_\A(x^\n)=0$.  By feasibility of $\wh{x}$ for problem~\eqref{eq:mainCvx}, it holds that $\gauge_\A(\wh{x})\le \gauge_\A(x^\n) = 0$. 
By Assumption~\ref{a:alph}\ref{assumption:alphabet0}, $\A$ is symmetric and $\gauge_\A$ is thus a norm in $\R^d$.  Because $\gauge_\A$ is a norm, $\gauge_\A(\wh{x}) = \gauge_\A (x^\n)=0$ implies that  $\wh{x} = x^\n=0$.  We thus assume that $\gauge_\A(x^\n)> 0$ from now on. By definition of the gauge function in~\eqref{eq:gaugeFcn1}, we have
\begin{align*}
    \wh{x}/\gauge_\A (\wh{x}) \in \hull(\A),
\end{align*}
with the convention that $0/0=0$. Since $\A$ is symmetric by Assumption~\ref{a:alph}\ref{assumption:alphabet0}, it also holds that 
\begin{align}
    t \cdot \wh{x}/\gauge_\A (\wh{x}) \in \hull(\A), \qquad \forall t\in [-1,1],
            \label{eq:lineSeg}
\end{align}
i.e.,  the line segment connecting $\pm \wh{x}/\gauge_\A (\wh{x})$ also belongs to $\hull(\A)$. Moreover, by feasibility of $\wh{x}$ in problem~\eqref{eq:mainCvx}, we have that 
\begin{align*}
    \gauge_\A (\wh{x}) \le \gauge_\A (x^\n). 
\end{align*}
In view of the above relation, for the choice of $t=\gauge_\A(\wh{x})/\gauge_\A(x^\n)\in [0,1]$,~\eqref{eq:lineSeg} reduces to 
\begin{align}
    \wh{x}/\gauge_\A (x^\n) \in \hull(\A). \label{eq:scaledHatStrong}
\end{align}
For future reference, note also that the feasibility of $x^\n$ and optimality of $\wh{x}$ in problem~\eqref{eq:mainCvx}  imply that $\L(x^\n)=\L(\wh{x})=y$ and, consequently, 
\begin{align}
    \L\l( \frac{\wh{x}}{\gauge_\A(x^\n)} - \frac{x^\n}{\gauge_\A(x^\n)} \r) = \L(\wh{x} - x^\n ) = 0. 
    \qquad (\gauge_\A(x^\n)>0)
    \label{eq:feasCvx}
\end{align}
We now consider two cases:
\begin{enumerate}[leftmargin=*,wide]
\item Suppose that $\wh{x}/\gauge_\A(x^\n) \in \face^\n$. Then we find that 
\begin{align}
      \frac{\wh{x}}{\gauge_\A(x^\n)}-\frac{x^\n}{\gauge_\A(x^\n)} \in \face^\n-\face^\n \subset \lin(\face^\n). \label{eq:samePlane}
\end{align}
Consequently, \eqref{eq:feasCvx}, \eqref{eq:samePlane} and the injectivity of $\L$ on $\lin(\face^\n)$ together imply that $\wh{x} = x^\n$.

\item Suppose that $\wh{x}/\gauge_\A(x^\n) \notin \face^\n$. We can therefore strengthen~\eqref{eq:scaledHatStrong} as
\begin{align}
    \wh{x}/\gauge_\A(x^\n) \in \hull(\A) - \face^\n. \label{eq:scaledHatStrong2}
\end{align}
By assumption of Lemma~\ref{lem:duality}, there exists a certificate $Q= \L^*(q)$ that satisfies~\eqref{eq:dualCertCnds}. Recalling~\eqref{eq:feasCvx}, we then write that 
\begin{align*}
    0 & = \langle q, \L(\wh{x} - x^\n ) \rangle \qquad \text{(see \eqref{eq:feasCvx})}  \nonumber\\ & 
    = \langle \L^*(q), \wh{x}- x^\n  \rangle 
    = \langle Q, \wh{x} - x^\n \rangle 
    = \gauge_\A (x^\n) \l\langle Q, \frac{\wh{x}}{\gauge_\A (x^\n)} - \frac{x^\n}{\gauge_\A (x^\n)} \r\rangle \quad (\gauge_\A (x^\n) >0) \nonumber\\
    & = \l\langle Q, \frac{\wh{x} }{\gauge_\A (x^\n)} - \frac{x^\n}{\gauge_\A (x^\n)} \r\rangle<0 
\end{align*}
where above we used the assumption that $ \gauge(x^\n) >0$ as well as \eqref{eq:dualCertCnds} and \eqref{eq:scaledHatStrong2}. To avoid the above contradiction, it must hold that $\wh{x}/\gauge_\A(x^\n)\in \face^\n$ which again implies that $\wh{x} = x^\n$.
\end{enumerate}
This completes the proof of Lemma~\ref{lem:duality}.

\section{Technical Details of Section~\ref{sec:new}}
\subsection{Proof of Proposition \ref{prop:propsHullP}}\label{sec:caratheodory}
The nested property of $\{\hull_{p}(\A)\}_p$ in~\eqref{eq:nested} is evident from its definition in~\eqref{eq:far-right}. To show the far-left identity in~\eqref{eq:nested}, we use~\eqref{eq:far-right} for $r=1$ to write that 
\begin{align}
    \hull_1(\A) & = \bigcup_{A\in \A} \hull(\{A,0\}) 
    \qquad \text{(see \eqref{eq:far-right})} \nonumber\\
    & = \bigcup_{A\in \A} \bigcup_{0\le \tau \le 1}  \tau A = \bigcup_{0\le \tau \le 1} \bigcup_{A\in \A} \tau A = \bigcup_{0\le \tau \le 1} \tau \A.
\end{align}
To show the far-right identity in~\eqref{eq:nested}, recall that every point $\hull(\A)\subset \R^d$ can be expressed as a convex combination of at most $d+1$ atoms in the alphabet $\A$, by Carath\'eodory theorem~\cite{barvinok2002course}.  We now use both~\eqref{eq:far-right} and the Carath\'eodory theorem to write that 
\begin{align*}
    \hull_{d+1}(\A) 
    & = \bigcup_{\{A_i\}_{i=1}^{d+1} \subset \A} \hull(\{A_i\}_{i=1}^d\cup\{0\}) 
    = \bigcup_{\{A_i\}_{i=1}^{d+1} \subset \A} \bigcup_{0\le \tau \le 1} \tau\cdot  \hull(\{A_i\}_{i=1}^{d+1}) \nonumber\\ 
    & = \bigcup_{0\le \tau \le 1}  \bigcup_{\{A_i\}_{i=1}^{d+1} \subset \A} \tau \cdot \hull(\{A_i\}_{i=1}^{d+1}) 
    = \bigcup_{0\le \tau \le 1}  \tau \cdot \hull(\A) 
    = \hull(\A \cup \{0\}) 
    = \hull(\A),
\end{align*}
where the last line holds because $0\in \A$ by Assumption~\ref{a:alph}\ref{assumption:containsOrigin}. This establishes~\eqref{eq:nested} and  completes the proof of Proposition~\ref{prop:propsHullP}. 

\subsection{Proof of Proposition~\ref{prop:gauge2EquivDefn}}
To prove~\eqref{eq:gauge2}, we use the expression for $\hull_p(\A)$ in~\eqref{eq:far-right} to rewrite the definition of gauge$_p$ function in~\eqref{eq:gauge2Hull} as 
\begin{align*}
    \gauge_{\A,p}(x) 
    & = \inf\l\{ t: x/t \in \bigcup_{\slice\in \Slice_r(\A)} \slice ,\,  t\ge 0 \r\} 
        \qquad \text{(see (\ref{eq:far-right},\ref{eq:gauge2Hull}))} 
    \nonumber\\
    & =   
      \inf \l\{ t: x =  \sum_{i=1}^p c_i A_i,\,  \sum_{i=1}^p c_i \le t ,\,  c_i \ge 0, A_i \in \A, \,\forall i\in [p] \r\} 
    \nonumber\\
    & = \inf \l\{ \sum_{i=1}^p c_i : x =  \sum_{i=1}^p c_i A_i ,\, c_i \ge 0, A_i \in \A, \,\forall i\in [p] \r\},
\end{align*}
which proves~\eqref{eq:gauge2}. To show Proposition~\ref{prop:gauge2EquivDefn}\ref{prop:gauge_p:origin}, suppose that $\gauge_{\A,p}(x)=0$ which implies by definition in~\eqref{eq:gauge2Hull} that $x/t\in  \hull_p(\A)$ for every $t>0$.  Since the alphabet $\A$ and, consequently, $\hull_p(\A)$ in~\eqref{eq:far-right} are both bounded by Assumption~\ref{a:alph}\ref{assumption:bounded}, we conclude that $x=0$. 

To prove Proposition~\ref{prop:gauge2EquivDefn}\ref{prop:gauge_p:conjugate}, we begin by writing down the convex conjugate of $\gauge_{\A,p}$ as 
\begin{align*}
    \gauge_{\A,p}^*(z) & = \sup_{x}\,\, \langle x,z\rangle - \gauge_{\A,p}(x) 
    \nonumber\\
    & = \sup\l\{ \langle x, z\rangle - \sum_{i=1}^p c_i : x = \sum_{i=1}^p c_i A_i,\,  c_i \ge 0,\, A_i\in \A, \, \forall i\in [p] \r\} 
    \qquad \text{(see \eqref{eq:gauge2})}
    \nonumber\\
    & = \sup\l\{ \sum_{i=1}^p c_i (\langle A_i, z \rangle - 1 ):  c_i \ge 0,\, A_i\in \A, \, \forall i\in [p] \r\} 
    = \begin{cases}
    0 & \text{if } \sup_{A\in \A}\,\, \langle A,z\rangle \le 1\\
    \infty & \text{if } \sup_{A\in \A}\,\, \langle A,z\rangle > 1
    \end{cases} \nonumber\\
    & = 
    \ind_{\B(\dual_\A)}(z),
\end{align*}
where $\B(\dual_\A)$ is the unit ball for the dual norm of the gauge function, i.e.,  $\dual_\A$ in~\eqref{eq:dual}. It also immediately follows that 
\begin{align*}
    \gauge_{\A,p}^{**} & = (\ind_{\B(\dual_{\A} )} )^* = \gauge_{\A},
\end{align*}
which proves Proposition~\ref{prop:gauge2EquivDefn}\ref{prop:gauge_p:envelop}. Proposition~\ref{prop:gauge2EquivDefn}\ref{prop:gauge_p:infeasible} and~\ref{prop:gauge_p:sparse} trivially follow from the definition of the gauge$_p$ function in~\eqref{eq:gauge2}. Lastly, the nested property of the gauge$_p$ functions in~\eqref{eq:nestedGauges} follows immediately from~\eqref{eq:gauge2}. The identity on the far-right of~\eqref{eq:nestedGauges} follows by combining the far-right identity in~\eqref{eq:nested} with~\eqref{eq:gauge2Hull}. This completes the proof of Proposition~\ref{prop:gauge2EquivDefn}. 

\subsection{Proof of Lemma \ref{lem:nonCvxLearn}}
Let $\wh{x}$ be a minimizer of problem~\eqref{eq:mainNonCvx}. 
Suppose that $\gauge_{\A,p}(x^\n)=0$. 
By feasibility of $\wh{x}$ in problem~\eqref{eq:mainNonCvx}, it holds that $\gauge_{\A,p}(\wh{x}) \le \gauge_{\A,p}(x^\n)=0$ and, consequently, $\wh{x}=x^\n=0$ by  Proposition~\ref{prop:gauge2EquivDefn}\ref{prop:gauge_p:origin}. We thus assume that $\gauge_{\A,p}(x^\n)>0$ from now on. 

By definition of the gauge$_p$ function in~\eqref{eq:gauge2Hull}, it holds that 
\begin{align*}
    x^\n /\gauge_{\A,p}(x^\n) \in \hull_p(\A).
\end{align*}
Again by definition of the gauge$_p$ function and using also the definition of $\hull_p(\A)$ in~\eqref{eq:far-right}, there  exists a slice $\slice\in \Slice_p(\A)$ such that 
\begin{align}
    \wh{x}/\gauge_{\A,p}(\wh{x}) \in \slice \subset  \hull_p(\A).
    \label{eq:sliceHat}
\end{align}
Moreover, since the slice $\slice$ is a convex set containing the origin, see Definition~\ref{defn:slice}, it follows from~\eqref{eq:sliceHat} that  
\begin{align}
    t \cdot \wh{x}/\gauge_{\A,p}(\wh{x}) \in \slice , \qquad \forall t\in [0,1]. 
    \label{eq:lineSeg2}
\end{align}
By feasibility of $\wh{x}$ in problem~\eqref{eq:mainNonCvx}, we have that $\gauge_{\A,p}(\wh{x}) \le \gauge_{\A,p}(x^\n)$, and we can thus take $t=\gauge_{\A,p}(\wh{x})/\gauge_{\A,p}(x^\n) \in [0,1]$ in~\eqref{eq:lineSeg2} to find that 
\begin{align}
    \wh{x}/\gauge_{\A,p}(x^\n) \in \slice\,.
    \label{eq:usefulLater}
\end{align}
For future reference, note also that the feasibility of $x^\n$ and optimality of $\wh{x}$ in problem~\eqref{eq:mainCvx}  implies that $\L(x^\n)=\L(\wh{x})=y$ and, consequently, 
\begin{align}
    \L\l( \frac{\wh{x}}{\gauge_{\A,p}(x^\n)} - \frac{x^\n}{\gauge_{\A,p}(x^\n)} \r) = \L( \wh{x}-x^\n) = 0. \qquad ( \gauge_{\A,p}(x^\n) >0)
    \label{eq:feasNonCvx}
\end{align}
We can now proceed to the body of the proof by considering two cases:
\begin{enumerate}[leftmargin=*,wide]
    \item Suppose that  $x^\n/\gauge_{\A,p}(x^\n) \in \slice$. Then it follows from~\eqref{eq:sliceHat} that 
\begin{align}
      \frac{\wh{x}}{\gauge_{\A,p}(x^\n)}-\frac{x^\n}{\gauge_{\A,p}(x^\n)} \in \slice - \slice \subset \lin(\slice) . \qquad \text{(see \eqref{eq:sliceHat})}
    \label{eq:affHat}
\end{align}
Then,~(\ref{eq:feasNonCvx}) and~(\ref{eq:affHat}) together with the injectivity of $\L$ on $\lin(\slice)$  imply that $\wh{x} = x^\n$. 

\item  Suppose that  $x^\n/\gauge_{\A,p}(x^\n) \notin \slice$. 
It therefore exists a certificate $Q_{{\slice}}=\L^*(q_\slice)$ that satisfies~\eqref{eq:certNonCvx}. With this in mind and after recalling~\eqref{eq:feasNonCvx}, we   write that  
\begin{align*}
    0 & = \l\langle q_\slice, \L\l(\frac{\wh{x}}{\gauge_{\A,p}(x^\n)} -  \frac{x^\n }{\gauge_{\A,p}(x^\n)}  \r) \r\rangle  
    = \l\langle \L^*(q_\slice), \frac{\wh{x}}{\gauge_{\A,p}(x^\n)} -  \frac{x^\n }{\gauge_{\A,p}(x^\n)} \r\rangle 
    \nonumber\\ 
    & = \l\langle Q_\slice , \frac{\wh{x}}{\gauge_{\A,p}(x^\n)} -  \frac{x^\n }{\gauge_{\A,p}(x^\n)} \r\rangle < 0 
\end{align*}
which leads to a contradiction. We conclude that $x^\n/\gauge_{\A,p}(x^\n) \in \slice$  and, consequently, $\wh{x}=x^\n$. 
\end{enumerate}

This completes the proof of Lemma~\ref{lem:nonCvxLearn}.

\subsection{Example of \rev{an  Operator That Satisfies  the RIP}}\label{sec:near-iso}
Let $G\in \R^{m\times d}$ be a standard random Gaussian matrix, i.e.,  the entries of $G$ are independent Gaussian random variables with zero mean and unit variance. Consider a (linear) subspace $\subs \subset \R^d$ and let the $d\times \dim(\subs)$ matrix $U$ be an orthonormal basis for the span of this subspace. We  then write that 
\begin{align}
    \sup_{u\in \subs} \frac{\| G u\|_2}{\|u\|_2} = \sup_{v} \frac{\| G U v\|_2}{\|v\|_2} = \s_{\max}(GU) =: \s_{\max}(G'),
    \label{eq:maxSigmaMax}
\end{align}
where we set $G'=GU$ for short, and  $\s_{\max}(G')$ is the largest singular value of $G'$. 
Likewise, 
\begin{align*}
    \inf_{u\in \subs} \frac{\| G u\|_2}{\|u\|_2} = \inf_{v} \frac{\| G U v\|_2}{\|v\|_2} = \s_{\min}(GU) = \s_{\min}(G').
\end{align*}
Note that the $m\times \dim(\subs)$ matrix $G'=GU$ too is a standard random Gaussian matrix because $U^\top U=I_{\dim(\subs)}$ by construction. The largest and smallest singular values of a standard random Gaussian matrix are well-known~\cite[Corollary 5.35]{vershynin2010introduction}. In particular,  it holds that 
\begin{align}
    (1-\d)\sqrt{m}\le \s_{\min}(G') \le \s_{\max}(G') \le (1+\d) \sqrt{m},
    \label{eq:CoMGaussian}
\end{align}
provided that $m\ge C \dim(\subs)/\d^2$ and except with a probability of at most $\exp(-C'\d^2 m)$. \rev{Here,  $C,C'$ are universal constants.}  By combining~(\ref{eq:maxSigmaMax}) and~(\ref{eq:CoMGaussian}) for the linear operator $\L=G/\sqrt{m}$, we \rev{finally} arrive at
\begin{align*}
    (1-\d)\|u\|_2 \le \| \L(u)\|_2 \le (1+\d) \|u\|_2,\qquad \forall u\in \subs,
\end{align*}
provided that $m\ge C \dim(\subs)/\d^2$ and except with a probability of at most $\exp(-C'\d^2 m)$. The random linear operator $\L$ constructed above thus \rev{satisfies the probabilistic $\d$-RIP.}

\subsection{Proof of Proposition \ref{prop:criticalSpecial}}
Since $p\ge d+1$ by assumption, recall from~\eqref{eq:nestedGauges} that $\gauge_{\A,p}(x^\n)=\gauge_\A(x^\n)$. In particular, $x^\n/\gauge_{\A,p}(x^\n) = x^\n / \gauge_\A(x^\n) =: \wt{x}$. 
For a slice $\slice\in \Slice_{p}(\A)$, note that 
\begin{align}
    \slice& \subset \hull_p(\A) \qquad \text{(see \eqref{eq:far-right})} \nonumber\\
    & = \hull(\A),\qquad \text{(see \eqref{eq:nested})}
    \label{eq:obviousInclusion}
\end{align}
where the identity above holds by~\eqref{eq:nested} and because $p\ge d+1$. Let us assume that $\wt{x}\notin \slice$.  An  immediate implication of~\eqref{eq:obviousInclusion} is that  $\slice - \wt{x} \subset \hull(\A)- \wt{x}$ and, consequently, 
\begin{align*}
\cone(\slice - \wt{x}) \subset \cone(\hull(\A)-\wt{x}) = \cone( \hull( \A - \wt{x} ) ) = \cone(\A - \wt{x}),
\end{align*}
and, in turn, 
$\angle \cone(\slice - \wt{x}) \le 2 \angle \cone(\A - \wt{x}) $, where we invoked Lemma~\ref{lem:inclusion} below to obtain the last inequality. Using this last inequality and \eqref{eq:criticalAngle}, we arrive at
\begin{align*}
    \theta_{x^\n,p}(\A) & = \sup\l\{  \angle \cone\l(\slice - \wt{x} \r): \slice \in \Slice_{x^\n,p}(\A) \r\} 
    \le 2\angle \cone\l( \A - \wt{x} \r),
\end{align*}
which completes the proof of Proposition~\ref{prop:criticalSpecial}. 

To prove Lemma~\ref{lem:inclusion} below, in addition to~\eqref{eq:coneAngleDefn}, we first introduce two other notions of angle for a closed cone $\K\subset \R^d$, i.e., 
\begin{align}
\cos (\phi(\K))& := \max_{u\in \S^{d-1}} \min_{u'\in \K\cap \S^{d-1}} \,\, \langle u, u'\rangle,     \nonumber\\
\cos (\psi(\K)) & := \min_{u\in \K\cap \S^{d-1}} \min_{u'\in \K\cap \S^{d-1}} \,\, \langle u, u'\rangle. 
\label{eq:diffAngles}
\end{align}
These quantities are related as follows.
\begin{lem}\label{lem:relationAngles}
    For a closed cone $\K$, it holds that 
    \begin{align}
        \frac{\psi(\K)}{2} \le \phi(\K) \le \angle \K \le \psi(\K).
        \label{eq:relationAngles}
    \end{align}
\end{lem}
\begin{proof}
    Note that 
    \begin{align*}
        \cos(\psi(\K)) & = \min_{u\in \K\cap \S^{d-1}} \min_{u'\in \K \cap \S^{d-1}} \langle u,u'\rangle \qquad \text{(see \eqref{eq:diffAngles})} \nonumber\\
        & \le \max_{u\in \K\cap \S^{d-1}} \min_{u'\in \K \cap \S^{d-1}} \langle u, u'\rangle = \cos(\angle \K)
        \qquad \text{(see \eqref{eq:coneAngleDefn})}
        \nonumber\\
        & \le \max_{u\in \S^{d-1}} \min_{u'\in \K \cap \S^{d-1}} \langle u, u'\rangle 
        = \cos (\phi(\K)),
    \end{align*}
    which establishes the last two inequalities in~\eqref{eq:relationAngles}. On the other hand, note that $\phi(\K)$ is the angle of the smallest spherical cap that contains $\K\cap \S^{d-1}$, whereas  $\psi(\K)$ is the largest pairwise angle in~$\K$. The two quantities are thus related as 
    \begin{align*} 
        \psi(\K) \le 2\phi(\K),
    \end{align*}
    which proves the remaining inequality in~\eqref{eq:relationAngles}, and completes the proof of Lemma~\ref{lem:relationAngles}. 
\end{proof}

We next prove a weak inclusion result for cones.
\begin{lem}\label{lem:inclusion}
    For closed cones $\K_1 \subset \K_2$, it holds that $\angle \K_1 \le 2\angle \K_2$.
\end{lem}

\begin{proof}
    Note that by \eqref{eq:diffAngles} and \eqref{eq:relationAngles} we have
    \begin{align*}
        \cos(\angle \K_1) & \ge \cos(\psi(\K_1)) 
        = \min_{u\in \K_1\cap \S^{d-1}} \min_{u'\in \K_1 \cap \S^{d-1}} \langle u,u'\rangle \nonumber\\
        &  \ge \min_{u\in \K_2\cap \S^{d-1}} \min_{u'\in \K_2 \cap \S^{d-1}} \langle u,u'\rangle 
        = \cos(\psi(\K_2)) 
        \ge \cos(2 \angle \K_2), 
    \end{align*}
    which completes the proof of Lemma~\ref{lem:inclusion}.
\end{proof}

In general, Lemma~\ref{lem:inclusion} cannot be improved. For example, consider the example $\K_1=\cone(\{e_1,e_2\})$ and $\K_2=\cone(\{e_1,e_2,e_1+e_2\})$. It is easy to verify that $\angle \K_1 = \pi/2$ whereas $\angle \K_2 = \pi/4$, and thus the inequality in Lemma~\ref{lem:inclusion} holds with equality. However, we strongly suspect that Lemma~\ref{lem:inclusion} could be improved to $\angle \K_1\le \angle \K_2$ when $\K_1$ is a closed  convex cone.

\subsection{Proof of Theorem \ref{thm:constructOpt} }\label{sec:recepieProof}

\begin{figure}
    \centering
    \includegraphics[width=.3\textwidth]{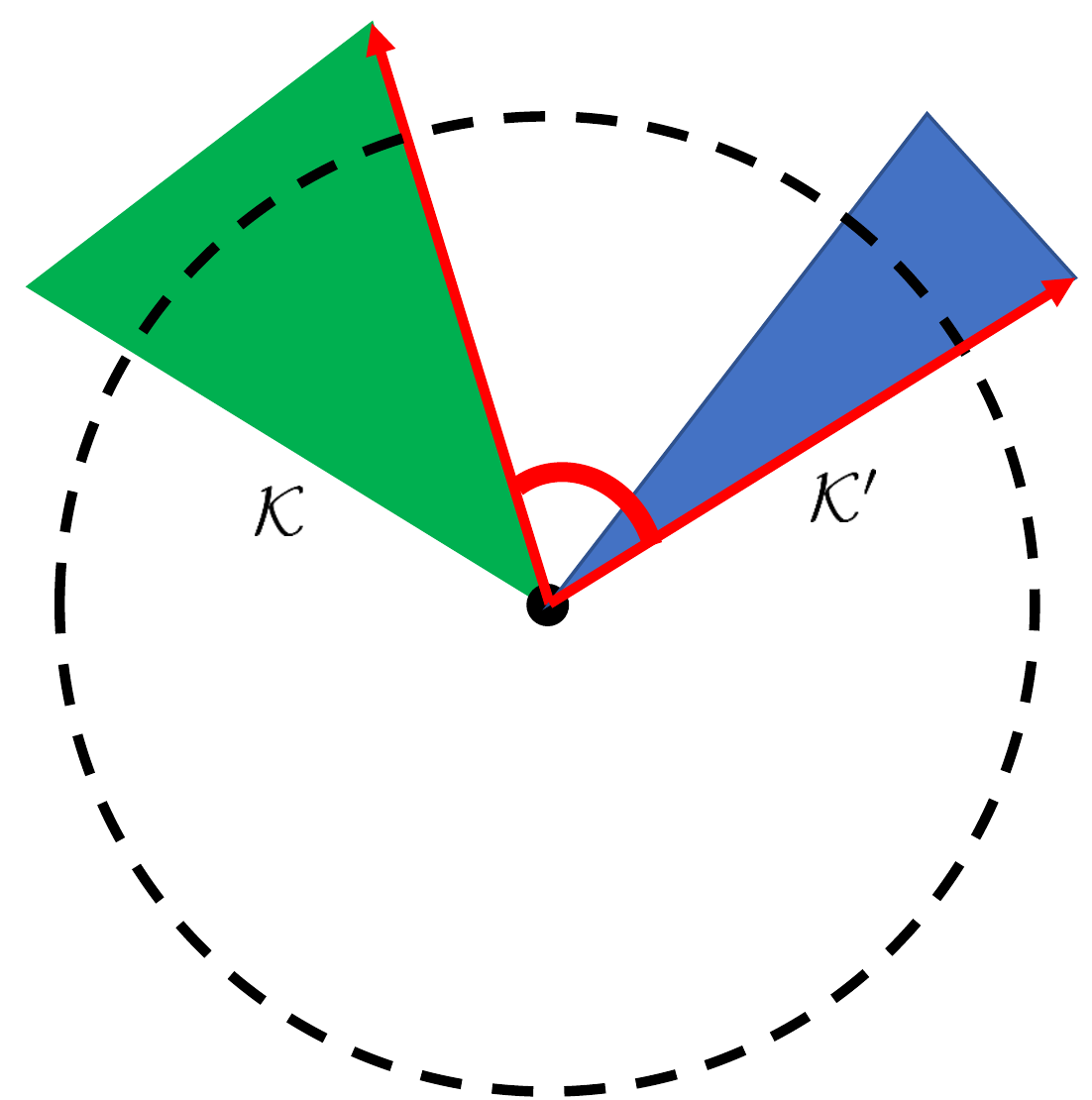}
    \caption{The angle between the green and blue cones equals the angle formed by the red arrows, see Definition~\ref{defn:angleCone}. }
    \label{fig:angleCones}
\end{figure}

Introducing the concept of angle between two cones is beneficial for this proof.
\begin{defn}[{\sc Angle between two cones}]\label{defn:angleCone} The angle between two closed cones $\K,\K'\subset \R^d$, \rev{denoted by $\angle [\K,\K'] \in [0,\pi]$, satisfies} 
\begin{align}
    \cos(\angle [\K,\K'] ) & := \min\l( \min_{u\in \K\cap \S^{d-1} } \max_{u'\in \K'\cap \S^{d-1}} \langle u,u'\rangle  ,  
    \min_{u'\in \K'\cap \S^{d-1} } \max_{u\in \K\cap \S^{d-1}} \langle u, u'\rangle 
    \r) \nonumber\\
    & = 1-\frac{1}{2} \l( \dist_\H (\K\cap \S^{d-1}, \K' \cap \S^{d-1}) \r)^2,
    \label{eq:angleConesDefn}
\end{align}
where $\dist_\H$ denotes the (Euclidean) Hausdorff distance between two sets~\cite{rockafellar2009variational}. 
\end{defn}


In words, the angle between two {closed} cones is the Hausdorff distance of their intersections with the unit sphere. For example,  Figure~\ref{fig:angleCones} \rev{shows sections of a blue cone and a green cone in $\R^2$.} The angle  between the blue and green cones equals the angle formed by the red arrows \rev{in the figure.}

When $\K$ and $\K'$ are two \rev{(linear)} subspaces of $\R^d$~\cite{golub2012matrix}, $\angle[\K,\K'] $ coincides with the (largest) principal angle between the two subspaces. 
Throughout the proof, we {will} frequently use the shorthand
\begin{align}
    \wt{x} := \frac{x^\n}{\gauge_{\A,p}(x^\n)},
    \qquad 
    u_x & := \frac{x-\wt{x}}{\|x-\wt{x}\|_2},
    \label{eq:shorthand}
\end{align}
for $x\ne \wt{x}$. As in Lemma~\ref{lem:nonCvxLearn}, we assumed above without loss of generality that $\gauge_{\A,p}(x^\n)>0$. The proof of Theorem~\ref{thm:constructOpt} relies on the following technical result, which is  similar to~\cite[Lemma 2.1]{candes2008restricted}. 
\begin{lem}\label{lem:fixedSliceIso} For $\d'\in [0,1)$, suppose that the random linear operator $\L$ \rev{satisfies the probabilistic $\delta'$-RIP}. Then,  for a slice $\slice'\in \Slice_p(\A)$,  it holds that
\begin{align}
    \frac{\langle Q_{\slice'}, u_{x'} \rangle}{\|x_{\slice'} - \wt{x}\|_2} \le -  \cos(\angle  \cone (\slice' - \wt{x}) ) + {\delta'}, \qquad \forall x'\in \slice',
    \label{eq:bndQSprime}
\end{align}
provided that $m\ge C p/\d'^2$, and except with a probability of $\exp(-C'\d'^2m)$. {Above, we set}
\begin{align}
 Q_{\slice'} := \L^*(\L(\wt{x}-x_{\slice'})),
    \label{eq:defnQSFixed}
\end{align}
and $x_{\slice'}$ is selected such that
\begin{align}
    \cos (\angle \cone(\slice' - \wt{x}) ) = \min_{x'\in \slice'} \,\, \l\langle u_{x_{\slice'}}, u_{x'} \r\rangle . 
    \label{eq:achievesConeAngle}
\end{align}
\end{lem}

Before proving Lemma~\ref{lem:fixedSliceIso} in {the next appendix},
{let us} first complete the proof of Theorem~\ref{thm:constructOpt}. Recall that $\Slice_{p}(\A)$ in~\eqref{eq:far-right} denotes the set of all slices of $\hull(\A)$ formed by at most $p$ atoms. For a resolution $\d>0$, let $\cover( \Slice_{p}(\A),\dist, \delta )$ denote a minimal $\delta$-net for $\Slice_{p}(\A)$ with respect to the pseudo-metric $\dist$, specified as 
\begin{align}
    \dist_p(\slice,\slice') & := \sqrt{2  - 2 \cos( \angle [ \cone(\slice - \wt{x}), \cone(\slice'-\wt{x})  ]) },
    \qquad \forall \slice,\slice'\in \Slice_{p}(\A). 
    \label{eq:pMetric}
\end{align}
Indeed, $\dist$ above is a pseudo-metric {because} it coincides with the Hausdorff distance between the intersection of the cones on the right-hand side above and the unit sphere in $\R^d$, see~\eqref{eq:angleConesDefn}. \rev{For $\d'\in [0,1)$, suppose that the random linear operator $\L$ satisfies the probabilistic $\d'$-RIP, see~\eqref{eq:isoDefn}. Consequently, by applying the union bound to all slices in $\cover( \Slice_{p}(\A),\dist, \delta )$}, we find that
\begin{align}
    & (1-\d')\|u\|_2 \le \| \L(u)\|_2 \le (1+\d') \|u\|_2, \nonumber\\
    & \qquad \forall u\in \lin(\slice'-\wt{x}), \qquad \forall \slice' \in  \cover(\Slice_{p}(\A),\dist,\d),
    \label{eq:nearIsoProof}
\end{align}
provided that $m\ge Cp/\d'^2$ and except with a probability of at most $$ \exp(-C'\d'^2m + \entropy(\Slice_{p}(\A),\dist,\d) ),
$$
where we used the fact that $\cover( \Slice_{p}(\A),\dist, \delta )$ is a minimal net by construction. \rev{The size of this net is therefore $\exp(\entropy(\Slice_{p}(\A),\dist,\d))$ by Definition~\ref{defn:entropy}.}

Consider an arbitrary slice $\slice\in \Slice_{p}(\A)$.
By Definition~\ref{defn:entropy}, there exists another slice $\slice'\in \cover( \Slice_{p}(\A),\dist, \delta )$ such that $\dist_p(\slice,\slice') \le \delta$.
\rev{After recalling \eqref{eq:pMetric}, this observation leads to}
\begin{align}
    \cos(\angle[\cone(\slice - \wt{x}),\cone(\slice'-\wt{x})]) \ge 1- \frac{\d^2}{2}. \qquad \text{(see \eqref{eq:pMetric})}
    \label{eq:angleAdjacent}
\end{align}
Consider  also an arbitrary point $x\in \slice$ and the corresponding unit-norm vector $u_{x}$, see~\eqref{eq:shorthand}. By~\eqref{eq:angleAdjacent} and after recalling the definition of angle between cones in~\eqref{eq:angleConesDefn}, there exists a unit-norm vector $u'\in \cone(\slice' - \wt{x})$ such that 
\begin{align}
    \langle u_{x}, u' \rangle \ge 1- \frac{\d^2}{2},
    \label{eq:usedCoverHere}
\end{align}
which also immediately implies that 
\begin{align}
    \|u_{x}-u'\|_2^2 & = \|u_{x}\|_2^2+\|u'\|_2^2 - 2 \langle u_{x}, u' \rangle  \nonumber\\
    & \le  2 - 2 \l(1-\frac{\d^2}{2}\r) =  \d^2,
    \qquad \text{(see (\ref{eq:shorthand}). (\ref{eq:usedCoverHere}))}
    \label{eq:usedCoverAgain}
\end{align}
\rev{Alternatively, it is also easy to arrive at the above conclusion from the Hausdorff distance interpretation of $\dist_p$ that we discussed earlier.}  
We next distinguish  two cases:
\begin{enumerate}[leftmargin=*,wide]
    \item Suppose that $\wt{x} \in \slice$. 
    Without loss of generality, we may assume that the ray passing through~$\wt{x}$ belongs to the interior of $\cone(\slice)$. (Indeed, otherwise there exists a lower-dimensional slice, to the interior of which the ray passing through $\wt{x}$ {would} belong.) Consequently, {we can use the} definition of $\wt{x}$ in~\eqref{eq:shorthand} {to write that}
    \begin{align}
         \cone(\slice - \wt{x}) \cup - \cone(\slice - \wt{x})= \lin(\slice - \wt{x}).
        \label{eq:wlog}
    \end{align}
     Moreover, note that  
    \begin{align}
        \lin(\slice- \wt{x}) = \lin(\slice),
        \label{eq:coneToSubs}
    \end{align}
 {because} the definition of slice in Definition~\ref{defn:slice} \rev{implies that $0\in \slice$. On the other hand, note that} 
    \begin{align}
        \| \L(u_x)\|_2 & \ge  \|\L(u')\|_2 -  \| \L(u_x - u')\|_2
        \qquad \text{(triangle inequality)} \nonumber\\
        & \ge 1- \d' - \|\L\|_\op \cdot \|u_x -u'\|_2 
        \qquad \text{(see \eqref{eq:nearIsoProof})}
        \nonumber\\
        & = 1- \d' - \|\L\|_\op \cdot \d > 0,
        \qquad \text{(see \eqref{eq:usedCoverAgain})}
        \label{eq:injProved}
    \end{align}
    where the last line above holds if $\d' + \|\L\|_\op \cdot \d <1$.
    Since the choice of the point $x\in \slice$ in~\eqref{eq:injProved} was arbitrary,  we conclude that $\L$ is an injective operator when restricted to $\cone(\slice - \wt{x})$ and, consequently, also when restricted to $-\cone(\slice - \wt{x})$. In view of~\eqref{eq:wlog}, $\L$ is also injective when restricted to $\lin(\slice - \wt{x})$ and, by~\eqref{eq:coneToSubs},  when restricted to $\lin(\slice)$.
    \item Suppose that $\wt{x}\notin \slice$. Recall $Q_{\slice'}$ {from}~\eqref{eq:defnQSFixed} and note that 
\begin{align*}
    \frac{\langle Q_{\slice'}, u_{x} \rangle }{\|x_{\slice'} - \wt{x}\|_2} & = 
     \frac{\langle \L^*(\L(\wt{x}-x_{\slice'} )), u_{x} \rangle }{\|x_{\slice'} - \wt{x} \|_2} 
    \qquad \text{(see \eqref{eq:defnQSFixed})}
    \nonumber\\
    & = - \langle \L( u_{x_{\slice'}}), \L(u_{x}) \rangle 
    \qquad \text{(see \eqref{eq:shorthand})}
    \nonumber\\
    & = - \langle \L (u_{x_{\slice'}}), \L(u') \rangle + \langle \L(u_{x_{\slice'}}), \L( u'-u_{x}  ) \rangle \nonumber\\
    & =  \frac{\langle Q_{\slice'} , u' \rangle }{\|x_{\slice'} - \wt{x} \|_2} + \langle \L(u_{x_{\slice'}}), \L( u'- u_{x} ) \rangle \nonumber\\ 
    & \le   \frac{\langle Q_{\slice'} , u' \rangle }{\|x_{\slice'} - \wt{x} \|_2} + \|\L\|_\op^2 \cdot \|u_{x_{\slice'}}\|_2 \cdot \| u'-u_{x}\|_2  \nonumber\\
    & \le -\cos( \cone (\slice' - \wt{x}) ) + \delta'+  \|\L\|_\op^2  \d 
    \qquad \text{(see \eqref{eq:bndQSprime} and \eqref{eq:usedCoverAgain})}\nonumber\\
   & \le - \inf_{\slice'' \in \Slice_{x^\n,p}(\A) } \cos( \cone (\slice'' - \wt{x}) ) + \delta' +  \|\L\|_\op^2 \d \nonumber\\
   & = - \cos(\theta_{x^\n,p}(\A) ) + \delta' + \|\L\|_\op^2 \d <0,
   \qquad \text{(see \eqref{eq:criticalAngle})}
\end{align*}
where the last line above holds if
\begin{align}
    \delta' + \|\L\|_\op^2 \d < \cos(\theta_{x^\n,p}(\A) ).
\end{align}
Since the choice of the point $x\in \slice$ above was arbitrary, we arrive at the following: There exists $Q_\slice\in \range(\L^*) $ such that 
\begin{align*}
    \langle Q_\slice, x - \wt{x} \rangle <0,
    \qquad \forall x\in \slice.
\end{align*}
\end{enumerate}
We may now invoke Lemma~\ref{lem:nonCvxLearn} to conclude that $x^\n$ is the unique minimizer of problem~\eqref{eq:mainNonCvx}. This completes the proof of Theorem~\ref{thm:constructOpt} \rev{after choosing $\d'=\cos(\theta_{x^\n,p}(\A))/2$.}

\subsection{Proof of Lemma~\ref{lem:fixedSliceIso}}\label{sec:proofFixedSliceIso}
For $x'\in \slice'$, we write that 
\begin{align}
    & \frac{\langle Q_{\slice'}, x' - \wt{x} \rangle}{\|x_{\slice'} - \wt{x}\|_2 \| x'-\wt{x}\|_2 }  \nonumber\\
    & = \frac{\langle \L^*(\L(\wt{x}-x_{\slice'} )) , x' - \wt{x} \rangle }{\|x_{\slice'} - \wt{x}\|_2 \|x'-\wt{x}\|_2} 
    \qquad \text{(see \eqref{eq:defnQSFixed})}
    \nonumber\\
    & = -\langle \L( u_{x_{\slice'}}  ) , \L( u_{x'}) \rangle 
    \qquad \text{(see \eqref{eq:shorthand})}
    \nonumber\\
    & = \frac{-1}{4}( \| \L(u_{x_{\slice'}}+u_{x'}) \|_2^2 - \|\L( u_{x_{\slice'}}-u_{x'} )\|_2^2  ) 
    \qquad \text{(parallelogram identity)}
    \nonumber\\
    & \le -\frac{1}{4} ( (1-\delta')\|u_{x_{\slice'}} + u_x \|_2^2 - \|\L\|_{\op} \|u_{x_{\slice'}} - u_{x'}\|_2^2 ) 
    \qquad (\L \text{ satisfies  the probabilistic RIP, see \eqref{eq:isoDefn}})
    \nonumber\\
    & = - \frac{1}{4}( \| u_{x_{\slice'}}+u_{x'} \|_2^2 -\| u_{x_{\slice'}} - u_{x'}\|_2^2 ) + \frac{\delta'}{4} \|u_{x_{\slice'}} +u_{x'} \|_2^2 +
    \frac{\|\L\|_\op-1}{4}
    \|u_{x_{\slice'}} - u_{x'}\|_2^2   
    \nonumber\\
    & \le  - \frac{1}{4}( \| u_{x_{\slice'}}+u_{x'} \|_2^2 -\| u_{x_{\slice'}} - u_{x'}\|_2^2 ) + \frac{1}{4}\max(\d,\|\L\|_\op-1) \l(\|u_{x_{\slice'}} +u_{x'} \|_2^2 +
    \|u_{x_{\slice'}} - u_{x'}\|_2^2\r)
    \nonumber\\
    & = - \langle u_{x_{\slice'}} , u_{x'} \rangle + \frac{1}{2}\max(\d,\|\L\|_\op-1)\l(\|u_{x_{\slice'}} \|_2^2 + \| u_{x'} \|_2^2 \r) \nonumber\\
    & = - \langle u_{x_{\slice'}} , u_{x'} \rangle + \max(\delta' ,\|\L\|_\op-1)
    \qquad \text{(see \eqref{eq:shorthand})}
    \nonumber\\
    & \le - \min_{u_{x''} \in \slice'} \,\,\langle u_{x_{\slice'}}, u_{x''} \rangle +  \max(\delta' ,\|\L\|_\op-1) \nonumber\\
    & = - \max_{u_{x'} \in \slice'} \min_{u_{x''} \in \slice'} \,\, \langle u_{x'}, u_{x''} \rangle  + \max(\delta' ,\|\L\|_\op-1)
    \qquad \text{(see (\ref{eq:coneAngleDefn}) and (\ref{eq:achievesConeAngle}))}
    \nonumber\\
    & = - \cos( \angle \cone(\slice' - \wt{x}) ) + \max(\delta' ,\|\L\|_\op-1),
    \qquad \text{(see \eqref{eq:coneAngleDefn})}
\end{align}
which completes the proof of Lemma~\ref{lem:fixedSliceIso}. 

\subsection{\rev{Review of} {Corollary~3.3.1 in~\cite{chandrasekaran2012convex}}}\label{sec:review}
For completeness, below we review  Corollary 3.3.1 in~\cite{chandrasekaran2012convex}, adapted to our notation. 

\textbf{Corollary 3.3.1 in \cite{chandrasekaran2012convex}.}
\emph{Suppose that the alphabet $\A\subset \R^d$ is a compact set and that~\eqref{eq:extInclusionGen} holds with equality, i.e., $\ext(\hull(\A))=\A$. 
Consider the  model $x^\n\in \R^d$ in~\eqref{eq:mainModel} and let $\L:\R^d\rightarrow\R^m$ be the linear map associated with the $m\times d$ Gaussian random matrix, populated with independent and zero-mean  normal random variables with the variance of $1/m$.  Then the learning machine 
\begin{align*}
    \min_x\,\, \gauge_\A(x)\,\, \textup{subject to}\,\, y=\L(x),
\end{align*}
returns $x^\n$, provided that 
\begin{equation*}
m \ge w(\Omega)^2+1,
\end{equation*}
and except with a probability of at most $\exp(-C(\sqrt{m}-w(\Omega))^2)$. \rev{Here,  $C$ is a universal constant.} Above, we set $\Omega := \cone( \A - x^\n/\gauge_\A(x^\n) ) \cap \S^{d-1}$. \rev{In words, $\Omega$ is} the intersection of the unit sphere with the tangent cone of $\hull(\A)$ at $x^\n/\gauge_\A(x^\n)$. \rev{Moreover,} $w(\Omega)$ is the Gaussian width of $\Omega$, i.e., 
\begin{align*}
    w(\Omega) = \E_g \l[\sup_{x\in \Omega}\,\, \langle g, x\rangle\r],
\end{align*}
where $g\in \R^d$ is a standard Gaussian random vector, i.e., populated with independent, zero-mean and  unit-variance normal random variables. 
}
\section{Technical Details of Section~\ref{sec:action}}

\subsection{Proof of Corollary~\ref{cor:manifold}}
First note that 
\begin{align}
    \gauge_{\A,1}(A^\n) & = 1.
    \label{eq:unitGaugeManifoldCor}
\end{align}
\rev{We can establish \eqref{eq:unitGaugeManifoldCor} by way of contradiction:} \rev{Recall the definition of gauge$_p$ function in~\eqref{eq:gauge2}. It is easy to see that $\gauge_{\A,1}(A^\n)\le 1$.}
\rev{In particular, } if $\gauge_{\A,1}(A^\n)<1$, then there exists an atom $A'\in\A$ aligned with $A^\n$ \rev{such that} $\|A'\|_2>\|A^\n\|_2=1$. \rev{This} contradicts the assumption that $\angle[A'-A^\n,A^\n]>0$ \rev{and establishes~\eqref{eq:unitGaugeManifoldCor}.}  

For atoms $A,A'\in \A$, consider the corresponding one-dimensional slices $\slice,\slice'\in \Slice_1(\A)$, \rev{which are} specified as
\begin{align*}
    \slice = \bigcup_{0\le \tau \le 1} \tau A ,
    \qquad 
    \slice' = \bigcup_{0\le \tau \le 1} \tau A'.
\end{align*}
\rev{Note that $\slice$ and $\slice'$} are simply the line segments connecting $A$ and $A'$ to the origin, \rev{respectively}. With $p=1$, the distance in~\eqref{eq:metricThm} \rev{between these two slices} is 
\begin{align}
    \dist_1(\slice,\slice') & = \sqrt{
    2 - 2 \cos \l(
    \angle \l[
    \cone(\slice - A^\n), \cone( \slice' -  A^\n)
    \r]
    \r)
    }\qquad \text{(see (\ref{eq:metricThm},\ref{eq:unitGaugeManifoldCor}))} \nonumber\\
    & = \dist_\H \l( \cone (\slice - A^\n ) \cap \S^{d-1}, \cone(\slice'-A^\n)\cap \S^{d-1}  \r)
    \qquad \text{(see \eqref{eq:angleConesDefn})} \nonumber\\
    & \le \l\| \frac{A-A^\n}{\|A-A^\n\|_2} - \frac{A'-A^\n}{\|A'-A^\n\|_2} \r\|_2,
    \label{eq:metricManifold}
\end{align}
with the convention that \rev{$0/0=-A^\n$}. \rev{In the last line above, we used the fact that the arguments of $\dist_\H$ are two arcs on the unit sphere. The first arc passes through $-A^\n$ and $(A-A^\n)/\|A-A^\n\|_2$. The second arc passes through $-A^\n$ and $(A'-A^\n)/\|A'-A^\n\|_2$. The Hausdorff distance between these two arcs is bounded by the distance of their end points, see the last line of~\eqref{eq:metricManifold}.}
Since each one-dimensional slice can be identified with its corresponding atom, it  follows that\rev{
 \begin{align*}
    & \entropy(\Slice_1(\A),\dist_1,\d) =  \entropy(U_{A^\n}(\A),\|\cdot\|_2,\d),
    \nonumber\\
    & \text{where}\quad 
    U_{A^\n}(\A) = \l\{ \frac{A-A^\n}{\|A-A^\n\|_2}: A\in \A \r\},
\end{align*}}for every $\d>0$ \rev{and with the convention that $0/0=-A^\n$.}
\rev{We can now invoke Lemma 15 from~\cite{eftekhari2015new} to find that 
\begin{align}
    \entropy(U_{A^\n}(\A),\|\cdot\|_2,\d) \le C'' k \log \l( \d^{-1} (\vol_k(\A))^{\frac{1}{k}} \reach(\A) \r),
    \label{eq:manifoldLemmaInvoked}
\end{align}
for a universal constant $C''$ and every $\d>0$.}


On the other hand, \rev{by definition,} the critical angle in~\eqref{eq:criticalAngle} satisfies 
\begin{align}
    \theta_{1,A^\n}(\A) & = \sup\l\{
     \frac{1}{2} \angle[ A-A^\n, -A^\n] : A\in \A-\{A^\n\} \r\}
     \qquad \text{(see \eqref{eq:criticalAngle})}
     \nonumber\\
     & = \sup \l\{ \frac{\pi}{2} - \frac{1}{2}\angle [A-A^\n,A^\n]: A\in \A - \{A^\n\}  \r\}  \nonumber\\
     & = \frac{\pi}{2}- \frac{\theta'_{A^\n,1}(\A)}{2}
     \qquad \rev{\text{(see \eqref{eq:gammaDefn})}}
     \nonumber\\
     & < \frac{\pi}{2},
    \label{eq:critAngleManifold}
\end{align}
where the last line above uses  the assumption on the alphabet $\A$. In view of~\eqref{eq:manifoldLemmaInvoked} and \eqref{eq:critAngleManifold}, we may now invoke Theorem~\ref{thm:constructOpt} to complete the proof of Corollary~\ref{cor:manifold}.

\rev{\subsection{Proof of Proposition \ref{prop:infLimit}}}

\rev{
Let 
\begin{equation}
\mathcal{U}:=\{u\in \R^{d}: \|u\|_2=1, \, \|u\|_0\le k\},
\label{eq:defnBigU}
\end{equation}
for short. For a fixed vector $u\in \mathcal{U} $, note that $\E[u^\top z_i]=0$ and that 
\begin{align}
    \E[(u^\top z_i)^2 ] & = u^\top\E[z_i z_i^\top] u = u^\top \Sigma u \nonumber\\
    & = u^\top A^\n u+ \theta \|u\|_2^2 = (u^\top u^\n)^2 + \theta, 
    \qquad \text{(see \eqref{eq:defnSigma})}
    \label{eq:fixedVar}
\end{align}
for every $i\le n$. In the second line above, 
we used the fact that any vector $u\in \mathcal{U}$ has unit $\ell_2$-norm.  
We are particularly interested in the deviation of the random variable $(u^\top z_i)^2$ from its expectation.

Throughout, $C$ is a universal constant, the value of which might change in every appearance. 
Recall from~\cite[
Definition 2.7]{vershynin2018high} the notion of sub-exponential norm of a random variable, which we denote by~$\|\cdot\|_{\psi_1}$. In particular, the sub-exponential norm of the random variable $(u^\top z_i)^2-(u^\top u^\n)^2-\theta$ can be calculated as 
\begin{align}
    \l\| (u^\top z_i)^2 - (u^\top u^\n)^2-\theta \r\|_{\psi_1} &\le C \l\| (u^\top z_i)^2\r\|_{\psi_1} 
    \qquad \text{\cite[Lemma 2.6.8]{vershynin2018high}}
    \nonumber\\
    & = C\| u^\top z_i\|_{\psi_2}^2
    \qquad \text{\cite[
Lemma 2.7.6]{vershynin2018high}} \nonumber\\
& \le C \E[(u^\top z_i)^2],
\qquad \text{\cite[Example 2.5.8]{vershynin2018high}}
\label{eq:subE}
\end{align}
where $\|\cdot\|_{\psi_2}$ returns  the sub-Gaussian norm of a random variable~\cite[Definition 2.5.6]{vershynin2018high}.  In view of~\eqref{eq:fixedVar}, we can revisit~\eqref{eq:subE} and write that 
\begin{equation}
    \l\| (u^\top z_i)^2 - (u^\top u^\n)^2-\theta \r\|_{\psi_1} \le C (u^\top u^\n)^2+ C\theta\le C, 
\end{equation}
where the last inequality above uses the fact that $u,u^\n\in \mathcal{U}$ are unit-norm vectors and that $\theta<1$. 
Because $\{z_i\}_{i=1}^n$ are independent random variables, we can now apply the Bernstein inequality~\cite[Corollary 2.8.3]{vershynin2018high} and find that
\begin{align}
    \l| u^\top y u - (u^\top u^\n)^2 - \theta \r| = \l|\frac{1}{n}\sum_{i=1}^n (u^\top z_i)^2 - (u^\top u^\n)^2 - \theta \r| \le \d, 
    \qquad \text{(see \eqref{eq:sampleCovMat})}
    \label{eq:bernie2}
\end{align}
except with the probability of at most $\exp( -C\min(\d^2,\d)n)$ and for every $\d>0$. 
Here, we  used the fact that~$\theta<1$ to simplify the failure probability. 

The remainder of the proof is a standard covering argument.
For $\epsilon>0$ to be set later, let $\mathcal{U}_\epsilon$ denote a minimal $\epsilon$-net for $\mathcal{U}$, with respect to the $\ell_2$-norm, see Definition~\ref{defn:entropy} or~\cite[Definition 4.2.1]{vershynin2018high}. By construction, $\log|\mathcal{U}_\epsilon|=\mathrm{entropy}(\mathcal{U},\epsilon)$, where $|\cdot|$ returns the size of a finite set and the right-hand side denotes the entropy number of $\mathcal{U}$ at resolution~$\epsilon$.  Using the definition of $\mathcal{U}$ in \eqref{eq:defnBigU}, it is not difficult to calculate that
\begin{equation}
\log|\mathcal{U}_\epsilon| = \mathrm{entropy}(\mathcal{U},\epsilon) \le C k \l( \log d+ \log(1/\epsilon)\r).
\qquad \text{(see \cite[Corollary 4.2.13]{vershynin2018high})}
\label{eq:coverNumU}
\end{equation}
Applying the union bound to \eqref{eq:bernie2} and using \eqref{eq:coverNumU}, we find that 
\begin{align}
    \max_{u\in \mathcal{U}_\epsilon}  \l|u^\top y u - (u^\top u^\n)^2 - \theta \r| \le \d,
    \label{eq:onTheCover}
\end{align}
except with a probability of at most $\exp(C k \log d - C k \log(1/\epsilon)- C\min(\d^2,\d)n)$. 

Next, consider an arbitrary $u\in \mathcal{U}$ and choose $u_\epsilon\in \mathcal{U}_\epsilon$ such that $\|u-u_\epsilon\|_2\le \epsilon$. Such a point $u_\epsilon$ is guaranteed to exist by construction of the $\epsilon$-net $\mathcal{U}_\epsilon$. Using the reverse triangle inequality, 
we then write that  
\begin{align}
     & \l|u^\top y u - (u^\top u^\n)^2 - \theta \r| - \l|u^\top_\epsilon y u_\epsilon  - (u_\epsilon^\top u^\n)^2-\theta\r| \nonumber\\
     & 
    \le  \l|\l( u^\top y u - u^\top_\epsilon y u_\epsilon\r)-
    \l( (u^\top u^\n)^2 -  (u_\epsilon^\top u^\n)^2 
    \r)
    \r| \qquad \text{(reverse triangle inequality)} \nonumber\\
    & \le \l| u^\top y u-u_\epsilon^\top y u_\epsilon\r| +  \l| (u^\top u^\n)^2 - (u_\epsilon^\top u^\n)^2  \r|
    \qquad \text{(triangle inequality)} \nonumber\\
    & =\frac{1}{n} \l| \sum_{i=1}^n (u^\top z_i)^2-(u_\epsilon^\top z_i)^2 \r| 
    + \l| (u^\top u^\n)^2 - (u_\epsilon^\top u^\n)^2  \r|
    \qquad \text{(see \eqref{eq:sampleCovMat})} \nonumber\\
    & \le \| u- u_\epsilon\|_2 (\|u\|_2+\|u_\epsilon\|_2) \frac{1}{n}\sum_{i=1}^n \|z_i\|_2^2
    + \|u-u_\epsilon\|_2 (\|u\|_2+\|u_\epsilon\|_2) \|u^\n\|_2^2, 
\end{align}
where the last line above uses the Cauchy-Schwarz's inequality multiple times. By construction, $u$ and $u_\epsilon$ satisfy $\|u\|_2=\|u_\epsilon\|_2=\|u^\n\|_2=1$ and $\|u-u_\epsilon\|_2\le \epsilon$. With this in mind, we bound the last line above
\begin{align}
    \l|u^\top y u - (u^\top u^\n)^2 - \theta \r| - \l|u^\top_\epsilon y u_\epsilon  - (u_\epsilon^\top u^\n)^2-\theta\r|
    & \le \frac{2\epsilon}{n} \sum_{i=1}^n \|z_i\|_2^2 + 2\epsilon. 
    \label{eq:beforeZ'}
\end{align}
Recall that $z_i\sim \mathrm{normal}(0,\Sigma)$, which allows us to write that $z_i \overset{\mathrm{dist.}}{=} \Sigma^{\frac{1}{2}}z_i'$ for a standard Gaussian random variable $z_i'\sim \mathrm{normal}(0,I_d)$. Because $\{z_i\}_{i=1}^n$ are statistically independent, then so are the new random variables $\{z_i'\}_{i=1}^n$. We now revisit \eqref{eq:beforeZ'} and write its right-hand side as
\begin{align}
\l|u^\top y u - (u^\top u^\n)^2 - \theta \r| - \l|u^\top_\epsilon y u_\epsilon  - (u_\epsilon^\top u^\n)^2-\theta\r|
    & \le \frac{2\epsilon}{n} \sum_{i=1}^n \|\Sigma^{\frac{1}{2}} z_i'\|_2^2 + 2\epsilon \nonumber\\
    & \le \frac{2\epsilon \|\Sigma\|}{n} \sum_{i=1}^n  \|z_i'\|_2^2 + 2\epsilon \nonumber\\
    & \le \frac{4\epsilon}{n} \sum_{i=1}^n  \|z_i'\|_2^2 + 2\epsilon,
\end{align}
where the last line above follows because $\|\Sigma\|=1+\theta< 2$, see \eqref{eq:defnSigma}. Note that $\sum_{i=1}^n \|z_i'\|_2^2 $ is a Chi-square random variable of degree $nd$ because $\{z'_i\}_{i=1}^n$ are independent standard Gaussian random vectors. The tail probability of a Chi-square random variable is well-known and we therefore have that    
\begin{align}
    \l|u^\top y u - (u^\top u^\n)^2 - \theta \r| - \l|u^\top_\epsilon y u_\epsilon  - (u_\epsilon^\top u^\n)^2-\theta\r|
    & \le C\epsilon d,
    \label{eq:diffCover}
\end{align}
except with a probability of at most $\exp(-Cnd)$. Together, \eqref{eq:onTheCover} and \eqref{eq:diffCover} imply that 
\begin{align}
    \max_{u\in \mathcal{U}}\l|u^\top y u - (u^\top u^\n)^2 - \theta \r| & 
    \le \max_{u\in \mathcal{U}_\epsilon}\l|u^\top y u - (u^\top u^\n)^2 - \theta \r| + C\epsilon d \nonumber\\
    & \le \d + C\epsilon d,
    \label{eq:uppBnd}
\end{align}
except with a probability of at most 
\begin{align}
    \exp\l(Ck\log d+Ck \log(1/\epsilon)- C\min(\d^2,\d)n\r) + \exp(-Cnd).
    \label{eq:failProp}
\end{align}
Consider a sequence $\{k_l,d_l,n_l\}_l$ such that $\lim_{l\rightarrow \infty}n_l=\infty$.
Consider also the sequence  $\{\epsilon_l\}_l$ with $\epsilon_l=1/d_l^2$. 
As $l\rightarrow\infty$, in view of \eqref{eq:uppBnd} and~\eqref{eq:failProp}, there exists a sequence $\{\d_l\}_l$ such that 
\begin{align}
     \max_{u\in \mathcal{U}}\l|u^\top y u - (u^\top u^\n)^2 - \theta \r| \rightarrow 0,
\end{align}
with a probability that converges to one and provided that 
$$
\lim_{l\rightarrow \infty}k_l\log(d_l)/n_l=0.
\label{eq:goodRatio}
$$
Lastly, we note that $u^\n$ is the unique maximizer of the function $u\rightarrow (u^\top u^\n)^2+\theta$. This completes the proof of Proposition \ref{prop:infLimit}.}

\rev{\subsection{Proof of Proposition \ref{prop:sparcePCAgen}}

Consider nonnegative coefficients $\{c_i\}_{i=1}^p$ and atoms $\{A_i\}_{i=1}^p \subset \A$ that are feasible for the optimization problem~\eqref{eq:mainNonCvx}, i.e., $\sum_{i=1}^p c_i \le \gauge_{\A,p}(x^\n)$.
Note that $$\l\|y - \sum_{i=1}^p c_i A_i\r\|_\mathrm{F}^2=\|y\|_\mathrm{F}^2 - 2 \sum_{i=1}^p c_i \langle y,  A_i \rangle + \l\| \sum_{i=1}^p c_i A_i\r\|_\mathrm{F}^2,$$
where only the second and third components depend on $\{c_i,A_i\}_i$.
Consequently, the problem~\eqref{eq:mainNonCvx} has the same solutions as 
\begin{align}
    \min\l\{ -2 \sum_{i=1}^p c_i \langle y,A_i \rangle + \l\|\sum_{i=1}^p c_i A_i\r\|_\mathrm{F}^2 : \sum_{i=1}^p c_i \le \gauge_{\A,p}(x^\n),\, c_i \ge 0,\, A_i\in \mathcal{A} \r\}.
    \label{eq:rewritten1}
\end{align}
Since $y$ in \eqref{eq:sampleCovMat} is random, the only random term above is $\sum_{i=1}^p c_i \langle y,A_i\rangle$. We will focus on this random term first. 
For every $i\le p$, $A_i\in \mathcal{A}$ implies that there exists $u_i\in \mathcal{U}$ such that $A_i=u_i u_i^\top$. Recall that the sets $\A$ and $\mathcal{U}$ were defined in~\eqref{eq:spca2} and~\eqref{eq:defnBigU}, respectively. We can now rewrite the only random term in~\eqref{eq:rewritten1} as
\begin{equation}
    \sum_{i=1}^p c_i \langle y,A_i\rangle = \frac{1}{n}\sum_{i=1}^p\sum_{j=1}^n c_i \l\langle z_j z_j^\top , u_i u_i^\top \r\rangle =
    \sum_{i,j} c_i \langle u_i, z_j \rangle^2 .
    \qquad \text{(see \eqref{eq:sampleCovMat})}
\end{equation}
Recall from \eqref{eq:fixedVar} that 
\begin{align}
\E[\langle y,A_i\rangle]=\E[\langle u_i, z_j \rangle^2] = \langle \Sigma,A_i \rangle 
& = u_i^\top \Sigma u_i \qquad \text{(see \eqref{eq:fixedVar})} \nonumber\\
& = u_i^\top \l( \sum_{j=1}^r c_j^\n u_j^\n (u_j^\n)^\top + \theta I\r) u_i 
\qquad \text{(see \eqref{eq:spikedModel2})}
\nonumber\\
& = \sum_{j=1}^r c_j^\n (u_i^\top u_j^\n)^2 + \theta \sum_{j=1}^r c_j^\n \|u_i\|_2^2 \nonumber\\
& = \sum_{j=1}^r c_j^\n (u_i^\top u_j^\n)^2 + \theta \sum_{j=1}^r c_j^\n, \qquad \text{(see \eqref{eq:defnBigU})}
\label{eq:exp2i}
\end{align}
where the last line follows because $u_i\in \mathcal{U}$ is a unit-length vector for every $i$. 
On the other hand, following the same steps as in the proof of Proposition~\ref{prop:infLimit}, it is easy to verify that 
\begin{align}
    \max_{u\in \mathcal{U}} \l| 
    u^\top y u - \sum_{j=1}^r c_j^\n (u^\top u_j^\n)^2 - \theta \sum_{j=1}^r c_j^\n
    \r| & \le \d+ C \epsilon d,
    \label{eq:maxOnUEps2}
\end{align}
except with a probability of at most 
\begin{align}
    \exp\l(Ck\log d+Ck \log(1/\epsilon)- C\min(\d^2,\d)n\r) + \exp(-Cnd).
    \label{eq:failProp2}
\end{align}
Only this time, the factor $C$ in \eqref{eq:maxOnUEps2} and \eqref{eq:failProp2} may depend on $\gauge_p(x^\n)$. 
With \eqref{eq:maxOnUEps2} and \eqref{eq:failProp2} at hand, we now write that 
\begin{align}
    \l| \sum_{i=1}^p c_i \langle y,A_i \rangle  - \sum_{i=1}^p c_i \langle \Sigma,A_i\rangle \r| 
    & = \l| \sum_{i=1}^p c_i \langle y,A_i \rangle  - \sum_{i=1}^p c_i \E[\langle y,A_i \rangle] \r| 
    \qquad \text{(see the first line of \eqref{eq:exp2i})}\nonumber\\
   & \le \sum_{i=1}^p c_i \cdot \max_{A\in \A}\l| \langle y,A \rangle - \E[\langle y,A \rangle] \r| 
   \qquad \text{(Holder's inequality)}
   \nonumber\\
   & = \sum_{i=1}^p c_i \max_{u\in \mathcal{U}} \l| u^\top y u - \sum_{j=1}^r c_j^\n (u^\top u_j^\n)^2 - \theta \sum_{j=1}^r c_j^\n
    \r| 
    \qquad \text{(see \eqref{eq:exp2i})}
    \nonumber\\
    & \le \sum_{i=1}^p c_i \l( \d+ C \epsilon d \r)
    \qquad \text{(see \eqref{eq:maxOnUEps2})}
    \nonumber\\
    & \le \gauge_{\A,p}(x^\n)\l( \d+ C \epsilon d\r),
    \qquad (\text{feasibility of }\{c_i\}_{i=1}^p)
\end{align}
except with the failure probability specified in \eqref{eq:failProp2}. With the same argument as in the proof of Proposition~\ref{prop:infLimit}, we find that \eqref{eq:rewritten1} has asymptotically the same minimizers as
\begin{align}
    \min\l\{ - 2\sum_{i=1}^p c_i \langle \Sigma, A_i \rangle +  \l\| \sum_{i=1}^p c_i A_i \r\|^2_{\mathrm{F}}: \sum_{i=1}^p c_i \le \gauge_{\A,p}(x^\n), \, c_i \ge 0,\, A_i\in \A \r\},
    \label{eq:beforeCompletedSquares}
\end{align}
provided that \eqref{eq:goodRatio} holds. We next focus on the deterministic optimization problem~\eqref{eq:beforeCompletedSquares}.  Recall from~\eqref{eq:spikedModel2} that 
$$\Sigma = x^\n+\theta I= \sum_{i=1}^r c_i^\n A_i^\n+\theta I.
$$
Substituting for $\Sigma$ in \eqref{eq:beforeCompletedSquares}, we find that 
\begin{align}
    & \min\l\{ - 2\l \langle \sum_{j=1}^r c_j^\n A_j^\n, \sum_{i=1}^p c_i A_i \r\rangle - 2\theta \sum_{i=1}^p c_i \tr(A_i) +  \l\| \sum_{i=1}^p c_i A_i \r\|^2_{\mathrm{F}}: \sum_{i=1}^p c_i \le \gauge_p(x^\n), \, c_i \ge 0,\, A_i\in \A \r\} \nonumber\\
    & = \min\l\{ - 2\l \langle \sum_{j=1}^r c_j^\n A_j^\n, \sum_{i=1}^p c_i A_i \r\rangle - 2\theta \sum_{i=1}^p c_i  +  \l\| \sum_{i=1}^p c_i A_i \r\|^2_{\mathrm{F}}: \sum_{i=1}^p c_i \le \gauge_{\A,p}(x^\n), \, c_i \ge 0,\, A_i\in \A \r\}  \nonumber\\
    &  \ge \min\l\{ - 2\l \langle \sum_{j=1}^r c_j^\n A_j^\n, \sum_{i=1}^p c_i A_i \r\rangle - 2\theta \gauge_p(x^\n)  +  \l\| \sum_{i=1}^p c_i A_i \r\|^2_{\mathrm{F}}: \sum_{i=1}^p c_i \le \gauge_{\A,p}(x^\n),  c_i \ge 0, A_i\in \A \r\},
    \label{eq:tightRelaxationSparsePCA}
\end{align}
where the second line above uses the fact that $A_i\in \A$ satisfies $\tr(A_i)=1$. (We will later show that the relaxation in the last line above is, in fact, tight.)
In the last line above, note also that we can remove the term $-2\theta \gauge_{\A,p}(x^\n)$ without changing the minimizers. That is, instead of the optimization problem in the last line above, we can solve 
\begin{align}
    \min\l\{ - 2\l \langle \sum_{j=1}^r c_j^\n A_j^\n, \sum_{i=1}^p c_i A_i \r\rangle   +  \l\| \sum_{i=1}^p c_i A_i \r\|^2_{\mathrm{F}}: \sum_{i=1}^p c_i \le \gauge_{\A,p}(x^\n), \, c_i \ge 0,\, A_i\in \A \r\}. 
    \label{eq:beforeCompletedSquares2}
\end{align}
We can add and subtract $\|\sum_{j=1}^r c_j^\n A_j^\n\|_\mathrm{F}^2$ to the objective function above. By doing so, we observe that~\eqref{eq:beforeCompletedSquares2} has, in turn, the same minimizers as 
\begin{align}
    \min\l\{   \l\|\sum_{i=1}^r c_i^\n A_i^\n  -  \sum_{i=1}^p c_i A_i\r\|^2_{\mathrm{F}} :\sum_{i=1}^p c_i \le \gauge_{\A,p}(x^\n), \, c_i \ge 0,\, A_i\in \A  \r\},
    \label{eq:afterCompletedSquares}
\end{align}
Recall from~\eqref{eq:spikedModel2} that $x^\n=\sum_{i=1}^r c_i^\n A_i^\n$. Let $x^\n = \sum_{i=1}^p c_i A_i$ be a different decomposition of $x^\n$, where $c_i\ge 0$ and $A_i\in \A$ for every $i\le p$. If we take the trace of both sides of the last identity, we find that 
\begin{align}
     \sum_{i=1}^p c_i = \sum_{i=1}^p c_i \tr(A_i) =  \tr(x^\n) = \sum_{i=1}^r c_i^\n \tr(A_i^\n) = \sum_{i=1}^r c_i^\n,
    \label{eq:fixedGauge}
\end{align}
where we also used the fact that every atom $A\in \A$ satisfies $\tr(A)=1$. After recalling the definition of the gauge$_p$ function in~\eqref{eq:gauge2}, note that \eqref{eq:fixedGauge} implies 
\begin{equation}
    \gauge_{\A,p}(x^\n)=\tr(x^\n). 
    \label{eq:gaugeSparsePCA}
\end{equation}
It follows from \eqref{eq:gaugeSparsePCA} that $\{c_i,A_i\}_{i=1}^p$ is feasible for the problem~\eqref{eq:afterCompletedSquares}. Note also that 
\begin{equation}
    \sum_{i=1}^r c_i^\n A_i^\n - \sum_{i=1}^p c_i A_i = x^\n - x^\n=0.
\end{equation}
That is, $\{A_i^\n\}_{i=1}^r \cup \{A_i\}_{i=1}^p$ are linearly dependent. Recalling the definition of spark, we conclude that $r+p \ge  \mathrm{spark}(\A)$. In other words, if we take $p<\mathrm{spark}(\A)-r$, then $x^\n=\sum_{i=1}^r c_i^\n A_i^\n$ is the unique $p$-sparse decomposition of $x^\n$ in the alphabet $\A$. Consequently, $\{c_i^\n,A_i^\n\}_{i=1}^r$ is the unique solution of~\eqref{eq:afterCompletedSquares}. In fact, as we saw earlier, \eqref{eq:afterCompletedSquares} has the same minimizers as the problem in the last line of \eqref{eq:tightRelaxationSparsePCA}. Therefore, $\{c_i^\n,A_i^\n\}_{i=1}^r$ is also the unique solution of the problem in the last line of~\eqref{eq:tightRelaxationSparsePCA}. Recall from  \eqref{eq:fixedGauge} and \eqref{eq:gaugeSparsePCA} that $\sum_{i=1}^r c_i^\n = \gauge_p(x^\n)$. Therefore, for the choice of $\{c_i^\n,A_i^n\}_{i=1}^r$, the objective function in the second and third lines of \eqref{eq:tightRelaxationSparsePCA} coincide. That is, the relaxation in \eqref{eq:tightRelaxationSparsePCA} is tight and we can replace the inequality in \eqref{eq:tightRelaxationSparsePCA} with equality.
This completes the proof of Proposition \ref{prop:sparcePCAgen}.
}
\section{Technical Details of Section~\ref{sec:algorithm}}

\subsection{Proof of Lemma~\ref{lem:MIQP}}
It is straightforward to see that the continuous variables $c_i = 0$ if and only if the binary variables~$s_i = 0$. Therefore, the constraint $\sum_{i=1}^{l}s_i = p$ of the binary variables directly imposes the required $p$-sparsity condition on $c$. 

\subsection{Proof of Proposition~\ref{prop:alg}}
First, observe that the machine~\eqref{eq:mainNonCvx} (or equivalently the MIQP reformulation~\eqref{gauge:MIQP}) can be  rewritten as 
\begin{align*}
    \min_{c} \Big\{\|\L (\Dic) c - y\|_2^2 ~:~ c \ge 0, \quad \ones^\top c \le \gauge_{\A,p}(x^\n), \quad \|c\|_0 \le p \Big\}. 
\end{align*}
Above, as usual,~$\|c\|_0$ denotes the number of nonzero entries of the vector~$c$. One can encode the nonzero elements of the vector~$c$ as a subset~$\Set \subset [|\A|]$. That is, we can introduce a new vector $c_{\Set}$ such that~$(c_\Set)_i = (c)_i$, when $i \in \Set$, otherwise $[c_\Set]_i= 0$. In this way, the above optimization program can be rewritten as
\begin{equation}\label{gauge:primal}
    \min_{{\footnotesize \begin{array}{c} \Set \subset [|\A|]\\ |\Set| \le p\end{array}}} \min_{c_{\Set}, z} \Big\{\|z\|_2^2 ~:~ z = \L (\Dic) c_{\Set} - y,\quad C c_{\Set} \le g \Big\}, 
\end{equation}
where the matrix $C$ and the vector~$g$ were defined in the proposition. Note that the inner optimization program in~\eqref{gauge:primal} is indeed a convex quadratic programming. This observation allows us to claim two things: (i)~We can add an additional term~$\|c_\Set\|_2^2/\gamma$ in the objective function where for all sufficiently large~$\gamma$ the optimal solution does not change. Indeed, the objective value of the convex inner problem does not change by adding the constraint~$\|c_S\|_2\le \epsilon$ for a sufficiently large~$\epsilon$. By  convexity of the inner problem, this is equivalent to adding the penalty term ~$\|c_\Set\|_2^2/\gamma$ for a sufficiently large~$\gamma$. (ii)~Thanks to the convexity, we can dualize the linear constraints and arrive at the equivalent optimization program
\begin{align*}
    \min_{{\footnotesize \begin{array}{c} \Set \subset [|\A|]\\ |\Set| \le p\end{array}}} \max_{\mu \ge 0, \lambda} \min_{c_{\Set}, z} \Big(\|z\|_2^2 + \frac{1}{\gamma}\|c_\Set\|_2^2 + \lambda^\top\big(\L( \Dic) c_{\Set} - y - z \big) + \mu^\top\big( C c_{\Set} - g \big)\Big).
\end{align*}
Note that the most inner minimization above is an unconstrained convex quadratic program. Computing the analytical solution for the variables~$(c_{\Set}, z)$ yields the desired program~\eqref{gauge:minimax}. With regards to the algorithm described through the dynamics~\eqref{gauge:alg}, first observe that the relation~\eqref{gauge:alg:Set} is the same as the maximizer of the objective function~\eqref{gauge:minimax} when the set is fixed to~$\Set_k$. Note further that the fixed point of~\eqref{gauge:alg} is indeed a saddle-point equilibrium for the zero-sum game between the player~$\Set$ and $(\mu,\lambda)$. Therefore, the equilibrium~${\Set}^\star$ is in fact also a ``policy security", i.e., the pair is the solution to the minimax program~\eqref{gauge:minimax} and its dual when the order of the minimization and maximization operators are changed~\cite[Proposition 4.2]{ref:Basar}.

{\bibliographystyle{unsrt}
\bibliography{ref}}

\end{document}

%% file: style_arXiv.tex
\makeatletter
\def\@settitle{\begin{center}%
		\baselineskip14\p@\relax
		\normalfont\LARGE\scshape\bfseries
		\@title
	\end{center}%
}
\makeatother

\makeatletter

\def\subsection{\@startsection{subsection}{2}%
	\z@{.5\linespacing\@plus.7\linespacing}{.5\linespacing}%
	{\normalfont\large\bfseries}}

\def\subsubsection{\@startsection{subsubsection}{3}%
	\z@{.5\linespacing\@plus.7\linespacing}{.5\linespacing}%
	{\normalfont\itshape}}

\usepackage[colorlinks	= true,
			raiselinks	= true,
			linkcolor	= MidnightBlue,
			citecolor	= Mahogany,
			urlcolor	= ForestGreen,
			pdfauthor	= {Peyman Mohajerin Esfahani},
			pdftitle	= {},
			pdfkeywords	= {},   
			pdfsubject	= {},
			plainpages	= false]{hyperref}
\usepackage[usenames, dvipsnames]{xcolor}
\usepackage{dsfont,amssymb,amsmath,graphicx,fancyhdr,mdframed}
\usepackage{amsfonts,dsfont,mathtools, mathrsfs,amsthm,wrapfig} 
\usepackage{mdframed,ragged2e,subfig}
\usepackage{enumitem}
\usepackage{bm}
\allowdisplaybreaks
\usepackage{algorithm}
\usepackage[noend]{algpseudocode}

\allowdisplaybreaks
\date{\today}

\patchcmd\@setauthors
  {\MakeUppercase{\authors}}
  {\authors}
  {}{}
  
\newtheorem{thm}{Theorem}[section]
\newtheorem{prop}[thm]{Proposition}
\newtheorem{defn}[thm]{Definition}

\newtheorem{lem}[thm]{Lemma}

\newtheorem{assumption}[thm]{Assumption}

\newtheorem{example}[thm]{Example}
\newtheorem{remark}[thm]{Remark}
\newtheorem{corr}[thm]{Corollary}

%% file: commands.tex
\def \l{\left}
\def \r{\right}

\def \R{\mathbb{R}}
\def \wh{\widehat}
\def \E{\mathbb{E}}

\def \B{\mathrm{ball}}
\def \g{\gamma}

\def \tr{\operatorname{trace}}
\def \vol{\text{d}}
\def \wt{\widetilde}
\def \s{\sigma}

\def \normal{\operatorname{normal}}
\def \A{\mathcal{A}}

\def \F{\text{face}}

\definecolor{darkcerulean}{rgb}{0.03, 0.27, 0.49}

\newcommand{\rev}[1]{{#1}}

\def \dist{\operatorname{dist}}
\def \supp{\operatorname{supp}}

\def\S{\mathbb{S}}

\def\d{\delta}

\def\range{\mathrm{range}}

\def\rank{\mathrm{rank}}

\def \n{\sharp}

\def\gauge{\mathcal{G}}
\renewcommand{\H}{\mathrm{H}}
\def\hull{\mathrm{conv}}
\def \L {\mathcal{L}}
\def\ext{\mathrm{ext}}
\def\face{\mathcal{F}}
\def\cone{\mathrm{cone}}
\def\cover{\mathrm{net}}

\def\aff{\mathrm{aff}}
\def\lin{\mathrm{lin}}
\def\int{\mathrm{int}}
\def\dual{\mathcal{D}}
\def\ind{\mathrm{indicator}}
 
\def\slice{\mathcal{S}}
\def\Slice{\mathrm{slice}}
\def\normal{\mathrm{normal}}
\def\op{\mathrm{op}}

\def\K{\mathcal{K}}
\def\subs{\mathcal{U}}

\def\entropy{\mathrm{entropy}}
\def\psd{\bm{\mathrm{S}}}
\def\vol{\mathrm{vol}}
\def\reach{\mathrm{reach}}

\usepackage{comment}
\usepackage{tikz}

\numberwithin{equation}{section}
\DeclareRobustCommand{\gobblefive}[5]{}

\usepackage{dsfont}

\newcommand{\Let}{:=}
\newcommand{\Dic}{{A}}
\newcommand{\ones}{\mathds{1}}
\newcommand{\Set}{{S}}
\newcommand{\notebox}[2]{\begin{mdframed}[backgroundcolor=#1, linewidth=1pt]{#2}\end{mdframed}}

%% file: Revision_v4_arxiv.bbl
\begin{thebibliography}{10}

\bibitem{chandrasekaran2012convex}
Venkat Chandrasekaran, Benjamin Recht, Pablo~A Parrilo, and Alan~S Willsky.
\newblock The convex geometry of linear inverse problems.
\newblock {\em Foundations of Computational mathematics}, 12(6):805--849, 2012.

\bibitem{bach2012optimization}
Francis Bach, Rodolphe Jenatton, Julien Mairal, and Guillaume Obozinski.
\newblock Optimization with sparsity-inducing penalties.
\newblock {\em Foundations and Trends{\textregistered} in Machine Learning},
  4(1):1--106, 2012.

\bibitem{rockafellar2015convex}
R.T. Rockafellar.
\newblock {\em Convex Analysis: (PMS-28)}.
\newblock Princeton Landmarks in Mathematics and Physics. 2015.

\bibitem{hastie2015statistical}
Trevor Hastie, Robert Tibshirani, and Martin Wainwright.
\newblock {\em Statistical learning with sparsity: the lasso and
  generalizations}.
\newblock Chapman and Hall/CRC, 2015.

\bibitem{negahban2012unified}
Sahand~N Negahban, Pradeep Ravikumar, Martin~J Wainwright, Bin Yu, et~al.
\newblock A unified framework for high-dimensional analysis of $ m $-estimators
  with decomposable regularizers.
\newblock {\em Statistical Science}, 27(4):538--557, 2012.

\bibitem{wainwright2009sharp}
Martin~J Wainwright.
\newblock Sharp thresholds for high-dimensional and noisy sparsity recovery
  using l1-constrained quadratic programming (lasso).
\newblock {\em IEEE transactions on information theory}, 55(5):2183--2202,
  2009.

\bibitem{bhaskar2013atomic}
Badri~Narayan Bhaskar, Gongguo Tang, and Benjamin Recht.
\newblock Atomic norm denoising with applications to line spectral estimation.
\newblock {\em IEEE Transactions on Signal Processing}, 61(23):5987--5999,
  2013.

\bibitem{candes2013simple}
Emmanuel Candes and Benjamin Recht.
\newblock Simple bounds for recovering low-complexity models.
\newblock {\em Mathematical Programming}, 141(1-2):577--589, 2013.

\bibitem{ahmed2013blind}
Ali Ahmed, Benjamin Recht, and Justin Romberg.
\newblock Blind deconvolution using convex programming.
\newblock {\em IEEE Transactions on Information Theory}, 60(3):1711--1732,
  2013.

\bibitem{shah2012linear}
Parikshit Shah, Badri~Narayan Bhaskar, Gongguo Tang, and Benjamin Recht.
\newblock Linear system identification via atomic norm regularization.
\newblock In {\em 2012 IEEE 51st IEEE conference on decision and control
  (CDC)}, pages 6265--6270. IEEE, 2012.

\bibitem{candes2008introduction}
Emmanuel~J Cand{\`e}s and Michael~B Wakin.
\newblock An introduction to compressive sampling.
\newblock {\em IEEE signal processing magazine}, 25(2):21--30, 2008.

\bibitem{bach2008convex}
Francis Bach, Julien Mairal, and Jean Ponce.
\newblock Convex sparse matrix factorizations.
\newblock {\em arXiv preprint arXiv:0812.1869}, 2008.

\bibitem{richard2014tight}
Emile Richard, Guillaume~R Obozinski, and Jean-Philippe Vert.
\newblock Tight convex relaxations for sparse matrix factorization.
\newblock In {\em Advances in neural information processing systems}, pages
  3284--3292, 2014.

\bibitem{schiebinger2018superresolution}
Geoffrey Schiebinger, Elina Robeva, and Benjamin Recht.
\newblock Superresolution without separation.
\newblock {\em Information and Inference: A Journal of the IMA}, 7(1):1--30,
  2018.

\bibitem{li2016super}
Qiuwei Li, Ashley Prater, Lixin Shen, and Gongguo Tang.
\newblock A super-resolution framework for tensor decomposition.
\newblock {\em arXiv preprint arXiv:1602.08614}, 2016.

\bibitem{burer2005local}
Samuel Burer and Renato~DC Monteiro.
\newblock Local minima and convergence in low-rank semidefinite programming.
\newblock {\em Mathematical Programming}, 103(3):427--444, 2005.

\bibitem{peyre2009manifold}
Gabriel Peyr{\'e}.
\newblock Manifold models for signals and images.
\newblock {\em Computer vision and image understanding}, 113(2):249--260, 2009.

\bibitem{zou2006sparse}
Hui Zou, Trevor Hastie, and Robert Tibshirani.
\newblock Sparse principal component analysis.
\newblock {\em Journal of computational and graphical statistics},
  15(2):265--286, 2006.

\bibitem{chi2019nonconvex}
Yuejie Chi, Yue~M Lu, and Yuxin Chen.
\newblock Nonconvex optimization meets low-rank matrix factorization: An
  overview.
\newblock {\em IEEE Transactions on Signal Processing}, 67(20):5239--5269,
  2019.

\bibitem{eftekhari2020implicit}
Armin Eftekhari and Konstantinos Zygalakis.
\newblock Implicit regularization in matrix sensing: A geometric view leads to
  stronger results.
\newblock {\em arXiv preprint arXiv:2008.12091}, 2020.

\bibitem{jin2017escape}
Chi Jin, Rong Ge, Praneeth Netrapalli, Sham~M Kakade, and Michael~I Jordan.
\newblock How to escape saddle points efficiently.
\newblock {\em arXiv preprint arXiv:1703.00887}, 2017.

\bibitem{eftekhari2015new}
Armin Eftekhari and Michael~B Wakin.
\newblock New analysis of manifold embeddings and signal recovery from
  compressive measurements.
\newblock {\em Applied and Computational Harmonic Analysis}, 39(1):67--109,
  2015.

\bibitem{shalev2014understanding}
Shai Shalev-Shwartz and Shai Ben-David.
\newblock {\em Understanding machine learning: From theory to algorithms}.
\newblock Cambridge university press, 2014.

\bibitem{bertsimas2020sparse}
Dimitris Bertsimas, Bart Van~Parys, et~al.
\newblock Sparse high-dimensional regression: Exact scalable algorithms and
  phase transitions.
\newblock {\em The Annals of Statistics}, 48(1):300--323, 2020.

\bibitem{bertsimas2016best}
Dimitris Bertsimas, Angela King, and Rahul Mazumder.
\newblock Best subset selection via a modern optimization lens.
\newblock {\em The annals of statistics}, pages 813--852, 2016.

\bibitem{barvinok2002course}
A.~Barvinok.
\newblock {\em A Course in Convexity}.
\newblock Graduate studies in mathematics. American Mathematical Society, 2002.

\bibitem{bi2016refined}
Yingjie Bi and Ao~Tang.
\newblock Refined shapley-folkman lemma and its application in duality gap
  estimation.
\newblock {\em arXiv preprint arXiv:1610.05416}, 2016.

\bibitem{eldar2012compressed}
Yonina~C Eldar and Gitta Kutyniok.
\newblock {\em Compressed sensing: theory and applications}.
\newblock Cambridge university press, 2012.

\bibitem{van2000asymptotic}
A.W. van~der Vaart.
\newblock {\em Asymptotic Statistics}.
\newblock Asymptotic Statistics. Cambridge University Press, 2000.

\bibitem{chen2001atomic}
Scott~Shaobing Chen, David~L Donoho, and Michael~A Saunders.
\newblock Atomic decomposition by basis pursuit.
\newblock {\em SIAM review}, 43(1):129--159, 2001.

\bibitem{candes2006robust}
Emmanuel~J Cand{\`e}s, Justin Romberg, and Terence Tao.
\newblock Robust uncertainty principles: Exact signal reconstruction from
  highly incomplete frequency information.
\newblock {\em IEEE Transactions on information theory}, 52(2):489--509, 2006.

\bibitem{davenport2016overview}
Mark~A Davenport and Justin Romberg.
\newblock An overview of low-rank matrix recovery from incomplete observations.
\newblock {\em IEEE Journal of Selected Topics in Signal Processing},
  10(4):608--622, 2016.

\bibitem{li2015overcomplete}
Qiuwei Li, Ashley Prater, Lixin Shen, and Gongguo Tang.
\newblock Overcomplete tensor decomposition via convex optimization.
\newblock In {\em 2015 IEEE 6th International Workshop on Computational
  Advances in Multi-Sensor Adaptive Processing (CAMSAP)}, pages 53--56. IEEE,
  2015.

\bibitem{boumal2020introduction}
NICOLAS Boumal.
\newblock An introduction to optimization on smooth manifolds.
\newblock {\em Available online, May}, 2020.

\bibitem{kerdreux2017approximate}
Thomas Kerdreux, Igor Colin, and Alexandre d'Aspremont.
\newblock An approximate shapley-folkman theorem.
\newblock {\em arXiv preprint arXiv:1712.08559}, 2017.

\bibitem{ekeland1999convex}
Ivar Ekeland and Roger Temam.
\newblock {\em Convex analysis and variational problems}.
\newblock SIAM, 1999.

\bibitem{pisier1981notes}
Gilles Pisier.
\newblock Remarks on an unpublished result {\ 'e} of b. maurey.
\newblock {\em S {\ 'e} minaire Functional analysis (called
  "Maurey-Schwartz")}, pages 1--12, 1981.

\bibitem{mallat2008wavelet}
S.~Mallat.
\newblock {\em A Wavelet Tour of Signal Processing: The Sparse Way}.
\newblock Elsevier Science, 2008.

\bibitem{argyriou2012sparse}
Andreas Argyriou, Rina Foygel, and Nathan Srebro.
\newblock Sparse prediction with the $ k $-support norm.
\newblock {\em Advances in Neural Information Processing Systems},
  25:1457--1465, 2012.

\bibitem{zou2005regularization}
Hui Zou and Trevor Hastie.
\newblock Regularization and variable selection via the elastic net.
\newblock {\em Journal of the royal statistical society: series B (statistical
  methodology)}, 67(2):301--320, 2005.

\bibitem{efron2016computer}
Bradley Efron and Trevor Hastie.
\newblock {\em Computer age statistical inference}, volume~5.
\newblock Cambridge University Press, 2016.

\bibitem{boumal2020deterministic}
Nicolas Boumal, Vladislav Voroninski, and Afonso~S Bandeira.
\newblock Deterministic guarantees for burer-monteiro factorizations of smooth
  semidefinite programs.
\newblock {\em Communications on Pure and Applied Mathematics}, 73(3):581--608,
  2020.

\bibitem{haeffele2015global}
Benjamin~D Haeffele and Ren{\'e} Vidal.
\newblock Global optimality in tensor factorization, deep learning, and beyond.
\newblock {\em arXiv preprint arXiv:1506.07540}, 2015.

\bibitem{recht2010guaranteed}
Benjamin Recht, Maryam Fazel, and Pablo~A Parrilo.
\newblock Guaranteed minimum-rank solutions of linear matrix equations via
  nuclear norm minimization.
\newblock {\em SIAM review}, 52(3):471--501, 2010.

\bibitem{rockafellar2009variational}
R.T. Rockafellar, M.~Wets, and R.J.B. Wets.
\newblock {\em Variational Analysis}.
\newblock Grundlehren der mathematischen Wissenschaften. Springer Berlin
  Heidelberg, 2009.

\bibitem{ledoux2013probability}
M.~Ledoux and M.~Talagrand.
\newblock {\em Probability in Banach Spaces: Isoperimetry and Processes}.
\newblock Classics in Mathematics. Springer Berlin Heidelberg, 2013.

\bibitem{vershynin2010introduction}
Roman Vershynin.
\newblock Introduction to the non-asymptotic analysis of random matrices.
\newblock {\em arXiv preprint arXiv:1011.3027}, 2010.

\bibitem{iwen2018recovery}
Mark~A Iwen, Felix Krahmer, Sara Krause-Solberg, and Johannes Maly.
\newblock On recovery guarantees for one-bit compressed sensing on manifolds.
\newblock {\em arXiv preprint arXiv:1807.06490}, 2018.

\bibitem{gomez2019fast}
Fabian~Latorre G{\'o}mez, Armin Eftekhari, and Volkan Cevher.
\newblock Fast and provable admm for learning with generative priors.
\newblock {\em arXiv preprint arXiv:1907.03343}, 2019.

\bibitem{bora2017compressed}
Ashish Bora, Ajil Jalal, Eric Price, and Alexandros~G Dimakis.
\newblock Compressed sensing using generative models.
\newblock In {\em Proceedings of the 34th International Conference on Machine
  Learning-Volume 70}, pages 537--546. JMLR. org, 2017.

\bibitem{lee2018introduction}
John~M Lee.
\newblock {\em Introduction to Riemannian manifolds}.
\newblock Springer, 2018.

\bibitem{absil2009optimization}
P-A Absil, Robert Mahony, and Rodolphe Sepulchre.
\newblock {\em Optimization algorithms on matrix manifolds}.
\newblock Princeton University Press, 2009.

\bibitem{federer1959curvature}
Herbert Federer.
\newblock Curvature measures.
\newblock {\em Transactions of the American Mathematical Society},
  93(3):418--491, 1959.

\bibitem{baraniuk2009random}
Richard~G Baraniuk and Michael~B Wakin.
\newblock Random projections of smooth manifolds.
\newblock {\em Foundations of computational mathematics}, 9(1):51--77, 2009.

\bibitem{davenport2007smashed}
Mark~A Davenport, Marco~F Duarte, Michael~B Wakin, Jason~N Laska, Dharmpal
  Takhar, Kevin~F Kelly, and Richard~G Baraniuk.
\newblock The smashed filter for compressive classification and target
  recognition.
\newblock In {\em Computational Imaging V}, volume 6498, page 64980H.
  International Society for Optics and Photonics, 2007.

\bibitem{davenport2010joint}
Mark~A Davenport, Chinmay Hegde, Marco~F Duarte, and Richard~G Baraniuk.
\newblock Joint manifolds for data fusion.
\newblock {\em IEEE Transactions on Image Processing}, 19(10):2580--2594, 2010.

\bibitem{d2005direct}
Alexandre d'Aspremont, Laurent~E Ghaoui, Michael~I Jordan, and Gert~R
  Lanckriet.
\newblock A direct formulation for sparse pca using semidefinite programming.
\newblock In {\em Advances in neural information processing systems}, pages
  41--48, 2005.

\bibitem{amini2008high}
Arash~A Amini and Martin~J Wainwright.
\newblock High-dimensional analysis of semidefinite relaxations for sparse
  principal components.
\newblock In {\em 2008 IEEE International Symposium on Information Theory},
  pages 2454--2458. IEEE, 2008.

\bibitem{berthet2013optimal}
Quentin Berthet, Philippe Rigollet, et~al.
\newblock Optimal detection of sparse principal components in high dimension.
\newblock {\em The Annals of Statistics}, 41(4):1780--1815, 2013.

\bibitem{deshpande2014information}
Yash Deshpande and Andrea Montanari.
\newblock Information-theoretically optimal sparse pca.
\newblock In {\em 2014 IEEE International Symposium on Information Theory},
  pages 2197--2201. IEEE, 2014.

\bibitem{mackey2009deflation}
Lester~W Mackey.
\newblock Deflation methods for sparse pca.
\newblock In {\em Advances in neural information processing systems}, pages
  1017--1024, 2009.

\bibitem{vu2013minimax}
Vincent~Q Vu, Jing Lei, et~al.
\newblock Minimax sparse principal subspace estimation in high dimensions.
\newblock {\em The Annals of Statistics}, 41(6):2905--2947, 2013.

\bibitem{donoho2003optimally}
David~L Donoho and Michael Elad.
\newblock Optimally sparse representation in general (nonorthogonal)
  dictionaries via l1 minimization.
\newblock {\em Proceedings of the National Academy of Sciences},
  100(5):2197--2202, 2003.

\bibitem{shah2011iterative}
Parikshit Shah and Venkat Chandrasekaran.
\newblock Iterative projections for signal identification on manifolds: Global
  recovery guarantees.
\newblock In {\em 2011 49th Annual Allerton Conference on Communication,
  Control, and Computing (Allerton)}, pages 760--767. IEEE, 2011.

\bibitem{safran2018spurious}
Itay Safran and Ohad Shamir.
\newblock Spurious local minima are common in two-layer relu neural networks.
\newblock In {\em International Conference on Machine Learning}, pages
  4433--4441. PMLR, 2018.

\bibitem{nocedal2006numerical}
Jorge Nocedal and Stephen Wright.
\newblock {\em Numerical optimization}.
\newblock Springer Science \& Business Media, 2006.

\bibitem{natarajan1995sparse}
Balas~Kausik Natarajan.
\newblock Sparse approximate solutions to linear systems.
\newblock {\em SIAM journal on computing}, 24(2):227--234, 1995.

\bibitem{vreugdenhil2021principal}
Robbie Vreugdenhil, Viet~Anh Nguyen, Armin Eftekhari, and Peyman~Mohajerin
  Esfahani.
\newblock Principal component hierarchy for sparse quadratic programs.
\newblock In {\em International Conference on Machine Learning}, pages
  10607--10616, 2021.

\bibitem{tang2013sparse}
Gongguo Tang, Badri~Narayan Bhaskar, and Benjamin Recht.
\newblock Sparse recovery over continuous dictionaries-just discretize.
\newblock In {\em 2013 Asilomar Conference on Signals, Systems and Computers},
  pages 1043--1047. IEEE, 2013.

\bibitem{papadimitriou1981complexity}
Christos~H Papadimitriou.
\newblock On the complexity of integer programming.
\newblock {\em Journal of the ACM (JACM)}, 28(4):765--768, 1981.

\bibitem{bienstock1996computational}
Daniel Bienstock.
\newblock Computational study of a family of mixed-integer quadratic
  programming problems.
\newblock {\em Mathematical programming}, 74(2):121--140, 1996.

\bibitem{del2017mixed}
Alberto Del~Pia, Santanu~S Dey, and Marco Molinaro.
\newblock Mixed-integer quadratic programming is in np.
\newblock {\em Mathematical Programming}, 162(1-2):225--240, 2017.

\bibitem{bigM}
Big-{M} and convex hulls.
\newblock \url{https://yalmip.github.io/tutorial/bigmandconvexhulls/}.

\bibitem{Lofberg2004}
J.~L{\"{o}}fberg.
\newblock Yalmip : A toolbox for modeling and optimization in matlab.
\newblock In {\em In Proceedings of the CACSD Conference}, Taipei, Taiwan,
  2004.

\bibitem{eftekhari2019sparse}
Armin Eftekhari, Jared Tanner, Andrew Thompson, Bogdan Toader, and Hemant
  Tyagi.
\newblock Sparse non-negative super-resolution—simplified and stabilised.
\newblock {\em Applied and Computational Harmonic Analysis}, 2019.

\bibitem{huang2010benefit}
Junzhou Huang, Tong Zhang, et~al.
\newblock The benefit of group sparsity.
\newblock {\em The Annals of Statistics}, 38(4):1978--2004, 2010.

\bibitem{jacob2009group}
Laurent Jacob, Guillaume Obozinski, and Jean-Philippe Vert.
\newblock Group lasso with overlap and graph lasso.
\newblock In {\em Proceedings of the 26th annual international conference on
  machine learning}, pages 433--440, 2009.

\bibitem{golub2012matrix}
Gene~H Golub and Charles~F Van~Loan.
\newblock {\em Matrix computations}, volume~3.
\newblock JHU press, 2012.

\bibitem{candes2008restricted}
Emmanuel~J Candes et~al.
\newblock The restricted isometry property and its implications for compressed
  sensing.
\newblock {\em Comptes rendus mathematique}, 346(9-10):589--592, 2008.

\bibitem{vershynin2018high}
Roman Vershynin.
\newblock {\em High-dimensional probability: An introduction with applications
  in data science}, volume~47.
\newblock Cambridge university press, 2018.

\bibitem{ref:Basar}
Tamer Basar and Geert~Jan Olsder.
\newblock {\em Dynamic Noncooperative Game Theory, 2nd Edition}.
\newblock Society for Industrial and Applied Mathematics, 1998.

\end{thebibliography}
